\pdfoutput=1
\documentclass[twoside]{article}
\usepackage[letterpaper, left=3.5cm, right=3.5cm, top=4cm, bottom=2.5cm]{geometry}
\usepackage{graphicx}
\usepackage{amsmath,amssymb,amsthm}
\usepackage{authblk}
\usepackage[utf8]{inputenc}
\usepackage{microtype}
\usepackage{mathrsfs} 
\usepackage{enumitem}
\usepackage{calc}
\usepackage{esint}
\usepackage{fancyhdr}
\usepackage{xcolor}
\usepackage[labelfont = bf]{caption}
\usepackage{doi}
\usepackage{subfigure}
\usepackage[numbers]{natbib}
\usepackage{float}
\usepackage{stmaryrd}
\usepackage{mathtools}
\usepackage{booktabs}
\usepackage{colortbl}
\usepackage{tabulary}
\usepackage{diagbox}
\usepackage{caption}

\newtheorem{theorem}{Theorem}[section]
\numberwithin{equation}{section}
\newtheorem{proposition}[equation]{Proposition}

\newtheorem{corollary}[equation]{Corollary}
\newtheorem{remark}[equation]{Remark}
\newtheorem{lemma}[equation]{Lemma}

\usepackage{titlesec}
\titleformat{\section}{\normalfont\scshape\centering}{\thesection.}{0.5em}{}
\titleformat*{\subsection}{\itshape}
\titleformat*{\subsubsection}{\itshape}

\providecommand{\keywords}[1]
{
	{\small\hspace{-5mm}\textit{Keywords:} #1}
}

\providecommand{\MSC}[1]
{
	{\small\hspace{-5mm}\textit{AMS MSC (2020):~~} #1}
}

\usepackage{hyperref}
\hypersetup{
	colorlinks=true,
	linkcolor=blue,
	citecolor = blue,
	filecolor=magenta,  
	urlcolor=cyan,
}

\providecommand{\jumptmp}[2]{#1\llbracket{#2}#1\rrbracket}
\providecommand{\jump}[1]{\jumptmp{}{#1}}

\AtBeginDocument{%
	\def\MR#1{}
}

\newcommand{\AAA}{\boldsymbol{\mathcal{A}}}

\begin{document}
	\setlength{\abovedisplayskip}{5.5pt}
	\setlength{\belowdisplayskip}{5.5pt}
	\setlength{\abovedisplayshortskip}{5.5pt}
	\setlength{\belowdisplayshortskip}{5.5pt}

	\title{\vspace{-5mm}Error analysis for a Crouzeix--Raviart approximation of the variable exponent Dirichlet~problem}
	\author[1]{Anna Kh.\ Balci\thanks{Email: \texttt{akhripun@math.uni-bielefeld.de}
	
	This research is Funded by the Deutsche Forschungsgemeinschaft (DFG, German Research Foundation) – Project-ID 317210226 – SFB 1283  and   by Charles University  PRIMUS/24/SCI/020 and Research Centre program No. UNCE/24/SCI/005.}}
	\author[2]{Alex Kaltenbach\thanks{Email: \texttt{kaltenbach@math.tu-berlin.de\vspace*{-4mm}}}}
	\date{\today}
	\affil[1]{\small{Department of Mathematical Analysis, Charles University in Prague,  Sokolovska 83, 186 75 Praha 8, Czech Republic 
 and Faculty of Mathematics, University of Bielefeld, Universitätsstr.~25, 33615 Bielefeld }}
	\affil[2]{\small{Institute of Mathematics, Technical University of Berlin, Stra\ss e des 17.\ Juni 135\\ 10623 Berlin}}
	\maketitle

	\pagestyle{fancy}
	\fancyhf{}
	\fancyheadoffset{0cm}
	\addtolength{\headheight}{-0.25cm}
	\renewcommand{\headrulewidth}{0pt} 
	\renewcommand{\footrulewidth}{0pt}
	\fancyhead[CO]{\textsc{CR for the variable exponent Dirichlet problem}}
	\fancyhead[CE]{\textsc{A. Kh.\ Balci and A. Kaltenbach}}
	\fancyhead[R]{\thepage}
	\fancyfoot[R]{}
	
	\begin{abstract}
		In the present paper, we examine a Crouzeix--Raviart approximation of the $p(\cdot)$-Dirichlet problem. We derive a \textit{medius} error estimate, \textit{i.e.}, a best-approximation result, which holds  for uniformly continuous  exponents and implies \textit{a priori} error estimates, which~apply~for~Hölder continuous exponents and are optimal for Lipschitz continuous exponents. 
        Numerical experiments are carried out to review the theoretical findings.
	\end{abstract}

	\keywords{\hspace{-0.1mm}$p(\cdot)$-Dirichlet \hspace{-0.1mm}problem; \hspace{-0.1mm}Crouzeix--Raviart \hspace{-0.1mm}element; \hspace{-0.1mm}\textit{a \hspace{-0.1mm}priori} \hspace{-0.1mm}error \hspace{-0.1mm}analysis;~\hspace{-0.1mm}\textit{medius}~\hspace{-0.1mm}error~\hspace{-0.1mm}\mbox{analysis}.}
	\MSC{\hspace{-0.1mm}49M29; 65N15; 65N30; 35J60; 46E30.}

	\section{Introduction}\label{sec:intro}
    \thispagestyle{empty}
  	
  	\hspace{5mm}We examine the numerical approximation of a non-linear system of \textit{$p(\cdot)$-Dirichlet type}, \textit{i.e.},\enlargethispage{7mm}
  	\begin{align}\label{eq:pDirichlet}
  	 \begin{aligned}
  	 -\mathrm{div}(\AAA(\cdot,\nabla u))&=f&&\quad\textup{ in }\Omega\,,\\
  	 u&= 0&&\quad\textup{ on }\Gamma_D\,,\\
  	  \AAA(\cdot,\nabla u)\cdot n &= 0&&\quad\textup{ on }\Gamma_N\,,
  	 \end{aligned}
  	\end{align}
 using the Crouzeix--Raviart element (\textit{cf.}\ \cite{CR73}). More precisely, for given $f\hspace{-0.15em}\in\hspace{-0.15em} L^{p'(\cdot)}(\Omega)$,~where~$p'(x)\hspace{-0.15em}\coloneqq\hspace{-0.15em} \smash{\frac{p(x)}{p(x)-1}}$ for all $x\hspace{-0.15em}\in\hspace{-0.15em} \Omega$, and $p\hspace{-0.15em}\in\hspace{-0.15em} C^0(\overline{\Omega})$ with $p^-\hspace{-0.15em}\coloneqq\hspace{-0.15em}\min_{x\in \overline{\Omega}}{p(x)}\hspace{-0.15em}>\hspace{-0.15em}1$, we seek $u\hspace{-0.15em} \in\hspace{-0.15em} {W^{1,p(\cdot)}_D(\Omega)}$~\mbox{solving}~\eqref{eq:pDirichlet}. Here, $\Omega \subseteq \mathbb{R}^d$, ${d \in \mathbb{N}}$, is a bounded simplicial Lipschitz domain whose topological boundary $\partial\Omega$ is disjointly  divided into a Dirichlet part $\Gamma_D$ and a Neumann part  $\Gamma_N$, and ${W^{1,p(\cdot)}_D(\Omega)}\coloneqq \{v\in \smash{W^{1,p(\cdot)}(\Omega)}\mid \textup{tr}\, v=0\textup{ a.e.\ on }\Gamma_D\}$. The  non-linear operator $\AAA\colon \Omega\times\mathbb{R}^d\to\mathbb{R}^d$, for every $(x,a)^\top\in \Omega\times\mathbb{R}^d$, is defined by
  	\begin{align}\label{example}
  	 \AAA(x,a)\coloneqq (\delta+\vert a\vert)^{p(x)-2}a\,.
  	\end{align}
   Since $\AAA\colon\Omega\times \mathbb{R}^d\to \mathbb{R}^d$ possesses a potential with respect~to~its~second~component, \textit{i.e.}, there exists a  function $\varphi\colon \Omega\times\mathbb{R}_{\ge 0}\to \mathbb{R}_{\ge 0}$, which is strictly convex with respect to its~second~component, such~that $(\partial_a\varphi)(x,a)=\AAA(x,a)$ 
   for all $(x,a)^\top\in\Omega\times \mathbb{R}^d$, each solution  $u\in  {W^{1,p(\cdot)}_D(\Omega)}$ of \eqref{eq:pDirichlet} is the unique minimizer of the functional ${I\colon {W^{1,p(\cdot)}_D(\Omega)}\to \mathbb{R}_{\ge 0}}$, for every $v\in{ W^{1,p(\cdot)}_D(\Omega)}$~defined~by\enlargethispage{5mm}
        \begin{align}
            I(v)\coloneqq \int_{\Omega}{\varphi(\cdot,\vert \nabla v\vert)\,\mathrm{d}x}-\int_{\Omega}{f\,v\,\mathrm{d}x}\,,\label{eq:pDirichletMin}
        \end{align}
        and vice-versa, leading to a primal and a dual formulation of \eqref{eq:pDirichlet}, and to convex~duality~relations.
        
The $p(\cdot)$-Dirichlet problem \eqref{eq:pDirichlet} is a prototypical example of a non-linear system with variable growth conditions. Regularity results for this models go back to the seminal papers by Acerbi and Mingione,\cite{AM01,AM011}. It  appears in physical models for \textit{smart fluids}, \textit{e.g.},~\mbox{electro-rheological~fluids} (\textit{cf.}\ \cite{RR96,Ru00}), micro-polar electro-rheological fluids (\textit{cf.}\ \cite{winr,eringenbook}), chemically reacting~fluids~(\textit{cf.}~\cite{LKM78,HMPR10}), and thermo-rheological fluids (\textit{cf.}\ \cite{Z97,AR06}).  In these models, the variable exponent depends on~physical quantities such as, \textit{e.g.}, an electric field, a concentration field or a~temperature~field. In addition, the $p(\cdot)$-Dirichlet problem \eqref{eq:pDirichlet} has applications in the field of image reconstruction (\textit{cf.}\ \cite{AMS08,CLR06,LLP10}). The main difficulty in the treatment of  models of type \eqref{eq:pDirichlet} is the non-autonomous structure (\textit{i.e.}, the dependency on the space variable).   In general,~without~additional assumptions on the space variable, there can be difficulties in the numerical treatment~of~such~problems.~In~this paper, we consider the situation with additional regularity assumption on the variable exponent, \textit{i.e.}, we assume that $p\in C^{0, \alpha}(\overline{\Omega})$ for some $\alpha\in (0,1]$. The obtained convergence rate result is new for  the non-conforming~\mbox{Crouzeix--Raviart}~method. 
\begin{remark}The essential feature of such models is the possibility of so-called energy gap (or Lavrentiev phenomenon), which could lead to the disconvergence of  conforming schemes to the global minimizer of the problem.  In comparison to the Lagrange elements, non-conforming methods, in particular, the Crouzeix--Raviart method, converges, see, e.g., \cite{BalDieSto22}. \end{remark}

   \subsection{Related contributions}\vspace{-1mm}

    \hspace{5mm}The numerical approximation of \eqref{eq:pDirichlet} in the case of a constant exponent $p\in (1,\infty)$ has already been subject of numerous contributions: 
    the best approximation property in terms of the natural distance 
  has first been established in \cite{BarLiu94}. This has been extended in \cite{DieRuz07} to the functionals with Orlicz growth. \textit{A priori} error estimates in terms of the mesh-size were obtained in~\cite{EbmLiu05,DieRuz07}. This was done in \cite{DieRuz07} by extending the approximation properties of standard interpolation operators to Orlicz spaces.  Results for generalized Newtonian~fluids~can~be~found~\cite{ST20}.    The~$p$-Dirichlet problem with Discontinuous Galerkin (DG) methods was studied in~\cite{BE08} and 
\cite{BO09}; for~the~Orlicz~problems~in~\cite{DKRT14}.~In~\cite{BelDieKre12},  it has been shown that the  adaptive finite element method with element-wise affine functions and D\"orfler marking converges with quasi-optimal rates. As the problem is non-linear, it has to be solved by iterative methods. A \hspace{-0.1mm}stable \hspace{-0.1mm}procedure, \hspace{-0.1mm}that \hspace{-0.1mm}is  \hspace{-0.1mm}efficient \hspace{-0.1mm}in \hspace{-0.1mm}the \hspace{-0.1mm}experiments,~\hspace{-0.1mm}was~\hspace{-0.1mm}introduced \hspace{-0.1mm}in \hspace{-0.1mm}\cite{DieForTomWan20}~\hspace{-0.1mm}for~\hspace{-0.1mm}the~\hspace{-0.1mm}case~\hspace{-0.1mm}${p\hspace{-0.1em}<\hspace{-0.1em}2}$ and in~\cite{BalDieSto22} for the case~$p>2$. In~\cite{K22CR}, an  error analysis for a Crouzeix--Raviart approximation of the $p$-Dirichlet problem was carried out, including optimal \textit{a priori}~error estimates,~a~\textit{medius}~error estimate, \textit{i.e.}, a best-approximation~result,~and~\textit{a~posteriori}~error~estimates. 
  
    However, there exist only a few contributions in the case of a variable exponent $p\in C^0(\overline{\Omega})$. The paper \cite{CHP10} is concerned with the (weak) convergence of a conforming, discretely inf-sup stable finite element approximation of the model for electro-rheological fluids. The first contribution addressing \textit{a priori} error estimates for finite element~approximations~of~\eqref{eq:pDirichlet}~can~be~found~in~\cite{BDS15}; see also \cite{BBD15}, for a extension to the model for electro-rheological fluids.
    
    In \cite{DPLM13}, the (weak) convergence of DG type methods is studied; the result~contains~no~convergence~rates.   In \cite{KPS18,KS19}, the (weak) convergence of a conforming, discretely inf-sup stable finite element approximation of the model for chemically reacting 
    fluids~is~proved.  In  \cite{BalOrtSto22}, the weak convergence of an approximation using the Crouzeix--Raviart element was established.
    To the best of the authors' knowledge, however, no \textit{a priori} error analysis for \eqref{eq:pDirichlet} for approximations using non-conforming ansatz~classes --in particular, the Crouzeix--Raviart element, which is usually the first step towards a fully-non-conforming \textit{a priori} error analysis-- has been carried~out~yet.\enlargethispage{7.5mm}

   \subsection{New contributions}\vspace{-1mm}

    \hspace{5mm}Deriving local efficiency estimates in terms of shifted $N$-functions and deploying the so-called node-averaging quasi-interpolation operator (\textit{cf.} \cite{Osw93,Sus96}), we generalize the \textit{medius} error analysis in 
    \cite{Bre15} from $p=2$ and $\delta=0$ in \eqref{example}, \textit{i.e.}, $\AAA=\textup{id}_{\mathbb{R}^d}\colon \mathbb{R}^d\to \mathbb{R}^d$, and of \cite{K22CR} from $p\in (1,\infty)$ and $\delta\ge 0$ in \eqref{example}
    to variable exponents $p\in C^0(\overline{\Omega})$ with $p^->1$.
    This \textit{medius} error analysis implies that the performances of the 
    conforming Lagrange finite element method applied to \eqref{eq:pDirichlet} and the non-conforming Crouzeix--Raviart finite element method applied~to~\eqref{eq:pDirichlet}~are~comparable.~As~a~result, we obtain \textit{a priori} error estimates for the approximation of \eqref{eq:pDirichlet}
    using the
    Crouzeix--Raviart element, which apply for Hölder continuous variable exponents 
    $p\in C^{0,\alpha}(\overline{\Omega})$, where $\alpha\in (0,1]$ and $p^->1$, and $\delta\ge  0$, and are quasi-optimal for Lipschitz continuous variable exponents $p\in C^{0,1}(\overline{\Omega})$~and~$\delta> 0$. Since $\AAA\colon\Omega\times \mathbb{R}^d\to \mathbb{R}^d$ has a potential and, therefore, \eqref{eq:pDirichlet} admits an equivalent formulation as a convex minimization problem  (\textit{cf}.\ \eqref{eq:pDirichletMin}),  we have access to a (discrete) convex duality theory (\textit{cf}.\ \cite{LLC18,Bar21,BK22B}) and \eqref{eq:pDirichlet} as well as the approximation of \eqref{eq:pDirichlet} employing the Crouzeix--Raviart  element admit dual formulations with a dual solution as well as a discrete~dual~solution,~respectively. We derive \textit{a priori} error estimates for the dual solution and the discrete~dual~solution.\vspace{-1mm}

  \subsection{Outline}\vspace{-1mm}

  \hspace{5mm}\textit{This article is organized as follows:} \!In Section \ref{sec:preliminaries}, we introduce the employed notation,~the~\mbox{basic} properties of non-linear operator $\AAA\colon\Omega\times\mathbb{R}^d\to \mathbb{R}^d$ and its  consequences, the relevant finite element spaces, and give brief review of the continuous and the discrete $p(\cdot)$-Dirichlet~problem. 
   In~Section~\ref{sec:medius}, we establish a \textit{medius} error analysis, \textit{i.e.}, best-approximation result, for the Crouzeix--Raviart finite element method applied to \eqref{eq:pDirichlet}.
   In Section \ref{sec:a_priori}, by~means~of~this~\textit{medius}~error~analysis, we derive \textit{a priori} error estimates for the  Crouzeix--Raviart finite element method applied~to~\eqref{eq:pDirichlet}, which are optimal for all Lipschitz continuous exponents $p\in C^{0,1}(\overline{\Omega})$ with $p^->1$ and all $\delta>0$. 
   In Section \ref{sec:experiments}, we review our theoretical findings via numerical experiments.

	\section{Preliminaries}\label{sec:preliminaries}\enlargethispage{7.5mm}
	
	\qquad Throughout the article, let ${\Omega\subseteq \mathbb{R}^d}$,~${d\in\mathbb{N}}$, always be a bounded simplicial Lipschitz domain, whose topological boundary $\partial\Omega$ is disjointly divided into a (relatively) closed Dirichlet part $\Gamma_D$, for which we  assume that $\vert \Gamma_D\vert>0$, and~a~Neumann~part~$\Gamma_N$,~\textit{i.e.}, ${\partial\Omega=\Gamma_D\cup\Gamma_N}$~and~${\emptyset=\Gamma_D\cap\Gamma_N}$. 
 The integral mean of an integrable function $f\colon \hspace*{-0.1em}M\hspace*{-0.1em}\to\hspace*{-0.1em}\mathbb{R}$ over a (Lebesgue)~\mbox{measurable}~set~${M\hspace*{-0.1em}\subseteq\hspace*{-0.1em} \mathbb{R}^d}$, ${d\hspace*{-0.1em}\in\hspace*{-0.1em} \mathbb{N}}$, with $\vert M\vert>0$, is denoted by $\langle f\rangle_M\coloneqq \smash{\frac 1 {|M|}\int_M f \,\textup{d}x}$. For (Lebesgue) measurable functions ${f,g\colon M\to \mathbb{R}}$ and a (Lebesgue) measurable set $ M\hspace*{-0.1em}\subseteq \hspace*{-0.1em}\mathbb{R}^d$, $d\hspace*{-0.1em}\in\hspace*{-0.1em} \mathbb{N}$,~we~write~${(f,g)_M\hspace*{-0.1em}\coloneqq \hspace*{-0.1em}\int_M f g\,\textup{d}x}$, whenever the  right-hand side is well-defined. We employ the notation~$\wedge$ and $\vee$, for the minimum and maximum, respectively.

 \subsection{Variable Lebesgue spaces and variable Sobolev spaces}

 \qquad Let $M\subseteq \mathbb{R}^d$, $d\in \mathbb{N}$, be a (Lebesgue) measurable set 
and $p\colon M\to [1,+\infty]$  (Lebesgue)~measurable; a \hspace{-0.1mm}\textit{variable
\hspace{-0.1mm}exponent}. \hspace{-0.1mm}By \hspace{-0.1mm}$\mathcal{P}(M)$, \hspace{-0.1mm}we \hspace{-0.1mm}denote \hspace{-0.1mm}the \hspace{-0.1mm}\textit{set \hspace{-0.1mm}of \hspace{-0.1mm}variable~exponents}.~\hspace{-0.1mm}Then,~\hspace{-0.1mm}for~\hspace{-0.1mm}${p\in \mathcal{P}(M)} $, we denote by
${p^+\hspace*{-0.1em}\coloneqq \hspace*{-0.1em}\textup{ess\,sup}_{x\in M}{p(x)}}$~and~${p^-\hspace*{-0.1em}\coloneqq \hspace*{-0.1em}\textup{ess\,inf}_{x\in
 M}{p(x)}}$ its constant~\textit{limit~exponents}.~Then,~by
$\mathcal{P}^{\infty}(M)\hspace*{-0.1em}\coloneqq\hspace*{-0.1em}  \{p\hspace*{-0.1em}\in\hspace*{-0.1em}\mathcal{P}(M)\mid
p^+\hspace*{-0.1em}<\hspace*{-0.1em}\infty\}$, we denote the \textit{set of bounded variable exponents}.~For~${p\hspace*{-0.1em}\in\hspace*{-0.1em}\mathcal{P}^\infty(M)}$ and $l\hspace*{-0.1em}\in\hspace*{-0.1em} \mathbb{N}$, \hspace{-0.1mm}we \hspace{-0.1mm}denote \hspace{-0.1mm}by \hspace{-0.1mm}$L^{p(\cdot)}(M;\mathbb{R}^l)$, \hspace{-0.1mm}the \textit{\hspace{-0.1mm}variable Lebesgue \hspace{-0.1mm}space},~\hspace{-0.1mm}\textit{i.e.},~\hspace{-0.1mm}the~\hspace{-0.1mm}vector~\hspace{-0.1mm}space~\hspace{-0.1mm}of~\hspace{-0.1mm}(Lebesgue) measurable functions $v\colon M\to \mathbb{R}^l$~for~which~the~\textit{modular}~${\rho_{p(\cdot),M}(v)\coloneqq \smash{\int_{M}{\vert v\vert^{p(\cdot)}\,\mathrm{d}x}}<\infty}$.
Then,~the \textit{Luxembourg norm} $\| v\|_{p(\cdot),M}\coloneqq \inf\{\lambda> 0\mid \rho_{p(\cdot),M}(\frac{v}{\lambda})\leq 1\}$ turns $L^{p(\cdot)}(M;\mathbb{R}^l)$~into~a~\mbox{Banach}~space. 

Moreover, for $p\in\mathcal{P}^\infty(\Omega)$ and $l\in \mathbb{N}$, we define the spaces\vspace{-1mm}
		\begin{alignat}{2}
		W^{1,p(\cdot)}_D(\Omega;\mathbb{R}^l)&\coloneqq \big\{v\in L^{p(\cdot)}(\Omega;\mathbb{R}^l)&& \mid \nabla v\in L^{p(\cdot)}(\Omega;\mathbb{R}^{l\times d})\,,\; \textup{tr}\,v=0\text{ in }L^{p^-}(\Gamma_D;\mathbb{R}^l)\big\}\,,\notag\\
		W^{p(\cdot)}_N(\textup{div};\Omega)&\coloneqq \big\{y\in L^{p(\cdot)}(\Omega;\mathbb{R}^d)&& \mid \textup{div}\,y\in L^{p(\cdot)}\label{eq:spaces}(\Omega)\,,\\[-1mm]&&&\;\;\langle\textup{tr}_n\,y,v\rangle_{W^{\smash{1-\frac{1}{(p^-)'},(p^-)'}}(\partial\Omega)}=0\text{ for all }v\in W_D^{1,(p^-)'}(\Omega)\big\}\,,\notag
	\end{alignat}
$W^{1,p(\cdot)}(\Omega;\mathbb{R}^l)\coloneqq W^{1,p(\cdot)}_D(\Omega;\mathbb{R}^l)$ if $\Gamma_D=\emptyset$, and $W^{p(\cdot)}(\textup{div};\Omega)\coloneqq W^{p(\cdot)}_N(\textup{div};\Omega)$ if $\Gamma_N=\emptyset$. Here,  $\textup{tr}\colon \smash{W^{1,p(\cdot)}(\Omega;\mathbb{R}^l)}\to\smash{L^{p^-}(\partial\Omega;\mathbb{R}^l)}$ and by $
	\textup{tr}_n\colon\smash{W^{p(\cdot)}(\textup{div};\Omega)}\to W^{\smash{-\frac{1}{p^-},p^-}}(\partial\Omega)$ denote the \textit{trace operator} and the \textit{normal trace operator}, respectively. In particular, we will always omit $\textup{tr}$ and $\textup{tr}_n$.
 We write $L^{p(\cdot)}(\Omega) \coloneqq L^{p(\cdot)}(\Omega;\mathbb{R}^1)$, ${W^{1,p(\cdot)}(\Omega)\coloneqq W^{1,p(\cdot)}(\Omega;\mathbb{R}^1)}$,~and~${W^{1,p(\cdot)}_D(\Omega)\coloneqq W^{1,p(\cdot)}_D(\Omega;\mathbb{R}^1)}$.

 \subsection{(Generalized) $N$-functions}
	
	\qquad A  convex function
 $\psi\colon\hspace{-0.1em}\mathbb{R}_{\geq 0} \to \mathbb{R}_{\geq 0}$ is called an
 \textit{$N$-function},~if~${\psi(0)=0}$,~${\psi(t)>0}$~for~all~${t>0}$,
 $\lim_{t\rightarrow0} \psi(t)/t=0$, and
 $\lim_{t\rightarrow\infty} \psi(t)/t=\infty$. If, in addition, $\psi\in C^1(\mathbb{R}_{\geq 0})\cap C^2(\mathbb{R}_{> 0})$~and~${\psi''(t)\hspace{-0.1em}>\hspace{-0.1em}0}$ for \hspace{-0.1mm}all \hspace{-0.1mm}$t>0$, \hspace{-0.1mm}we \hspace{-0.1mm}call \hspace{-0.1mm}$\psi$ \hspace{-0.1mm}a \hspace{-0.1mm}\textit{regular
  \hspace{-0.1mm}$N$-function}. \hspace{-0.1mm}For \hspace{-0.1mm}a \hspace{-0.1mm}regular \hspace{-0.1mm}$N$-function \hspace{-0.1mm}${\psi \colon \hspace{-0.1em}\mathbb{R}_{\geq 0}\to \mathbb{R}_{\geq 0}}$,~\hspace{-0.1mm}we~\hspace{-0.1mm}have~\hspace{-0.1mm}that $\psi (0)=\psi'(0)=0$,
 $\psi'\colon\hspace{-0.1em}\mathbb{R}_{\geq 0} \to \mathbb{R}_{\geq 0}$ is increasing and $\lim _{t\to \infty} \psi'(t)=\infty$.~For~a~given~\mbox{$N$-function} ${\psi \colon\mathbb{R}_{\geq 0} \to \mathbb{R}_{\geq 0}}$, we define the \textit{(Fenchel) conjugate \mbox{$N$-function}} $\psi^*\colon\mathbb{R}_{\geq 0} \to \mathbb{R}_{\geq 0}$,~for~every~$t\ge 0$,~by
 ${\psi^*(t)\coloneqq \sup_{s \geq 0} (st
  -\psi(s))}$, which satisfies $(\psi^*)' =
 (\psi')^{-1}$ in $\mathbb{R}_{\ge 0}$. An $N$-function $\psi\colon \mathbb{R}_{\geq 0} \to \mathbb{R}_{\geq 0}$ satisfies the \textit{$\Delta_2$-condition}
 (in short, $\psi \hspace*{-0.15em}\in\hspace*{-0.15em} \Delta_2$), if there exists $K\hspace*{-0.15em}>\hspace*{-0.15em} 2$ such that~for~all~${t \hspace*{-0.15em}\ge\hspace*{-0.15em} 0}$,~it~holds~that ${\psi(2\,t) \hspace*{-0.1em}\leq\hspace*{-0.1em} K\, \psi(t)}$. Then, \hspace*{-0.15mm}we \hspace*{-0.15mm}denote \hspace*{-0.15mm}the
 \hspace*{-0.1mm}smallest \hspace*{-0.15mm}such \hspace*{-0.15mm}constant \hspace*{-0.15mm}by \hspace*{-0.15mm}$\Delta_2(\psi)\hspace*{-0.15em}>\hspace*{-0.15em}0$. \hspace*{-0.15mm}We \hspace*{-0.15mm}say \hspace*{-0.15mm}that \hspace*{-0.15mm}an \hspace*{-0.15mm}$N$-function~\hspace*{-0.15mm}${\psi\colon\hspace*{-0.1em}\mathbb{R}_{\ge 0}\hspace*{-0.15em}\to \hspace*{-0.15em}\mathbb{R}_{\ge 0}}$ satisfies the \textit{$\nabla_2$-condition} (in short, $\psi\in \nabla_2$), if its (Fenchel) conjugate $\psi^*\colon\mathbb{R}_{\ge 0}\to \mathbb{R}_{\ge 0}$ is an $N$-function satisfying the $\Delta_2$-condition. 
 If $\psi\colon\mathbb{R}_{\ge 0}\to \mathbb{R}_{\ge 0}$ satisfies~the~$\Delta_2$-~and the $\nabla_2$-condition (in~short, $\psi\hspace*{-0.1em}\in\hspace*{-0.1em} \Delta_2\cap \nabla_2$), then, it holds  
 the~\textit{$\varepsilon$-Young}~\textit{inequality}:~for~\mbox{every}~${\varepsilon\hspace*{-0.1em}>\hspace*{-0.1em} 0}$, there exists a constant $c_\varepsilon>0 $, depending only on
 $\Delta_2(\psi),\Delta_2( \psi ^*)<\infty$, such that~for~every~$ s,t\geq0 $, it holds that
 \begin{align}
  \label{ineq:young}
 s\,t\leq  c_\varepsilon \,\psi^*(s)+\varepsilon \, \psi(t)\,.
 \end{align}
 For a (Lebesgue) measurable set $M\subseteq \mathbb{R}^d$, $d\in \mathbb{N}$,   $\psi \colon M \times \mathbb{R}_{\ge 0} \to \mathbb{R}_{\ge 0}$ is called \textit{generalized $N$-function} if it is Carath\'eodory mapping and $\psi(x,\cdot)\colon  \mathbb{R}_{\ge 0}\to \mathbb{R}_{\ge 0}$ is for a.e.\ $x \in M$~an~\mbox{$N$-function}. For a generalized $N$-function~$\psi \colon M \times \mathbb{R}_{\ge 0}\hspace{-0.1em} \to \hspace{-0.1em}\mathbb{R}_{\ge 0}$ and  a (Lebesgue)~\mbox{measurable}~function~${f\colon \hspace{-0.1em}M\hspace{-0.1em}\to \hspace{-0.1em}\mathbb{R}}$, we write 
 $$\rho_{\psi,M}(f)\coloneqq \int_M \psi(\cdot,\vert f\vert )\,\textup{d}x \,,$$
 whenever~the~\mbox{right-hand}~side~is~\mbox{well-defined}.
 
 \subsection{Basic properties of the non-linear operators}\label{sec:basic} 
 
 \qquad Throughout the entire article, we always assume that  $\AAA\colon\Omega\times \mathbb{R}^d\to \mathbb{R}^d$ has \textit{$(p(\cdot),\delta)$-structure}, where $p\in \mathcal{P}^{\infty}(\Omega)$ with $p^->1$ and $\delta\ge 0$, \textit{i.e.}, for every $ a\in \mathbb{R}^d$ and a.e.\ $x\in \Omega$,~it~holds~that
 \begin{align}\label{def:A}
 \AAA(x,a)\coloneqq (\delta+\vert a\vert )^{p(x)-2}a\,.
 \end{align}
 For given $p\in \mathcal{P}^{\infty}(\Omega)$ with $p^->1$ and $\delta\ge 0$, we introduce the \textit{special generalized $N$-function}
$\varphi\coloneqq \varphi_{p,\delta}\colon\Omega\times\mathbb{R}_{\ge 0}\to \mathbb{R}_{\ge 0}$, for every $t\ge 0$ and a.e.\ $x\in \Omega$, defined by
\begin{align} 
 \label{eq:def_phi} 
 \varphi(x,t)\coloneqq \int _0^t \varphi'(x,s)\, \mathrm ds\,,\quad\text{where}\quad
 \varphi'(x,t) \coloneqq (\delta +t)^{p(x)-2} t\,.
\end{align}
For (Lebesgue) measurable functions $f,g\colon\Omega \to \mathbb{R}_{\ge 0}$, we write
$f\sim g$ (or $f\lesssim g$) if~there~exists~a constant \hspace{-0.1mm}$c\hspace{-0.1em}>\hspace{-0.1em}0$ \hspace{-0.1mm}such \hspace{-0.1mm}that \hspace{-0.1mm}$c^{-1}g\hspace{-0.1em}\leq\hspace{-0.1em} f\hspace{-0.1em}\leq\hspace{-0.1em} c\,g$ \hspace{-0.1mm}(or \hspace{-0.1mm}$ f\hspace{-0.1em}\leq\hspace{-0.1em} c\,g$) \hspace{-0.1mm}a.e.\ \hspace{-0.1mm}in \hspace{-0.1mm}$\Omega$. \hspace{-0.1mm}In \hspace{-0.1mm}particular,~i\hspace{-0.1mm}f~\hspace{-0.1mm}not~\hspace{-0.1mm}otherwise~\hspace{-0.1mm}specified, we always assume that the hidden constant in $\sim$ and $\lesssim$ depends only on $p^-,p^+>1$ and $\delta\ge 0$. 

Then, $\varphi\colon \Omega\times\mathbb{R}_{\ge 0}\to \mathbb{R}_{\ge 0}$ satisfies, uniformly in $\delta\ge  0$ and a.e.\ $x\in \Omega$, the
$\Delta_2$-condition~with $\textup{ess\,sup}_{x\in \Omega}{\Delta_2(\varphi(x,\cdot))}\lesssim 2^{\smash{\max \{2,p^+\}}}$. In addition, 
the (Fenchel) conjugate function (with respect to the second argument) $\varphi^*\colon\Omega\times\mathbb{R}_{\ge 0}\to \mathbb{R}_{\ge 0}$ satisfies, uniformly~in~both~${t \ge 0}$,~${\delta \ge 0}$,~and~a.e.~$x\in \Omega$, $$\varphi^*(x,t) \sim
(\delta^{p(x)-1} + t)^{p'(x)-2} t^2\,,$$ and the $\Delta_2$-condition with
$\textup{ess\,sup}_{x\in \Omega}{\Delta_2(\varphi^*(x,\cdot))} \lesssim 2^{\smash{\max \{2,(p^-)'\}}}$.

For a generalized $N$-function $\psi\colon\Omega\times \mathbb{R}_{\ge 0}\to \mathbb{R}_{\ge 0}$, we introduce \textit{shifted generalized $N$-functions} $\psi_a\colon\Omega\times \mathbb{R}_{\ge 0}\to \mathbb{R}_{\ge 0}$, ${a\ge 0}$, for every $a,t\ge 0$ and a.e.\ $x\in \Omega$, defined by
\begin{align}
 \label{eq:phi_shifted}
 \psi_a(x,t)\coloneqq \int _0^t \psi_a'(x,s)\, \mathrm ds\,,\quad\text{where}\quad
 \psi'_a(x,t)\coloneqq \psi'(x,a+t)\frac {t}{a+t}\,.
\end{align}

\begin{remark} \label{rem:phi_a}
 For the above defined $N$-function $\varphi\colon \Omega\times \mathbb{R}_{\ge 0}\to \mathbb{R}_{\ge 0}$ (\textit{cf.}\ \eqref{eq:def_phi}) uniformly~in~${a,t\ge 0}$ and a.e.\ $x\hspace*{-0.1em}\in\hspace*{-0.1em} \Omega$, it holds that
 \begin{align*}
     \varphi_a(x,t) &\sim (\delta+a+t)^{p(x)-2} t^2\,,\\
     (\varphi_a)^*(x,t)&\sim ((\delta+a)^{p(x)-1} +t)^{\smash{p'(x)-2}} t^2\,.
 \end{align*}
 The \hspace{-0.1mm}families
 \hspace{-0.1mm}$\{\varphi_a\}_{\smash{a \ge 0}},\{(\varphi_a)^*\}_{\smash{a \ge 0}}\colon\hspace{-0.15em}\Omega\hspace{-0.05em}\times\hspace{-0.05em} \mathbb{R}_{\ge 0}\hspace{-0.15em}\to \hspace{-0.15em}\mathbb{R}_{\ge 0}$ \hspace{-0.1mm}satisfy, \hspace{-0.1mm}uniformly \hspace{-0.1mm}in \hspace{-0.1mm}$a\hspace{-0.15em}\ge\hspace{-0.15em} 0$,~\hspace{-0.1mm}the~\hspace{-0.1mm}\mbox{$\Delta_2$-condition}~\hspace{-0.1mm}with ${\textup{ess\,sup}_{x\in \Omega}{\Delta_2(\varphi_a(x,\cdot))} \lesssim 2^{\smash{\max \{2,p^+\}}}}\!$ and
 ${\textup{ess\,sup}_{x\in \Omega}{\Delta_2((\varphi_a)^*(x,\cdot))} \lesssim 2^{\smash{\max \{2,(p^-)'\}}}}\!$,~respectively.
\end{remark}

Closely related to non-linear operator $\AAA\colon\Omega\times\mathbb{R}^d\to \mathbb{R}^d$, defined by \eqref{def:A},~are~the~non-linear operators $F,F^*\colon\Omega\times\mathbb{R}^d\to \mathbb{R}^d$, for every $a\in \mathbb{R}^d$ and a.e.\ $x\in \Omega$ defined by
\begin{align}
 F(x,a)&\coloneqq(\delta+\vert a\vert)^{\frac{p(x)-2}{2}}a\\\,\qquad F^*(x,a)&\coloneqq (\delta^{p(x)-1}+\vert a\vert)^{\frac{p'(x)-2}{2}}a
 \,.\label{eq:def_F}
\end{align}

The relations between
$\AAA,F,F^*\colon\Omega\times\mathbb{R}^d
\to \mathbb{R}^d$ and
$\varphi_a,(\varphi^*)_a,(\varphi_a)^*\colon\Omega\times\mathbb{R}_{\ge
 0}\to \mathbb{R}_{\ge
 0}$,~${a\ge 0}$, are presented in
 the following proposition.

\begin{proposition}
 \label{lem:hammer}
 Uniformly in $t\ge 0$, 
 $a, b \in \mathbb{R}^d$, and a.e.\ $ x,y\in \Omega$, we have that
 \begin{align}\label{eq:hammera}
 \begin{aligned}
  \hspace{-8mm}(\AAA(x,a) - \AAA(x,b))
  \cdot(a-b ) &\sim \smash{\vert F(x,a) - F(x,b)\vert^2}
  \\&
  \sim \varphi_{\vert a\vert }(x,\vert a - b\vert ) 
  \,,\\[-3mm]
  \end{aligned}
  \end{align}
  \begin{align}
  \smash{\vert F^*(x,a) - F^*(x,b)\vert^2}
  \label{eq:hammerf}
  &\sim \smash{\smash{(\varphi^*)}_{\smash{\vert a\vert }}(x,\vert a - b\vert )}\,,\\
   \label{eq:hammerg}
 \smash{\smash{(\varphi^*)}_{\smash{\vert \AAA(x,a)\vert }}(x,t)}
   &\sim \smash{\smash{(\varphi}_{\smash{\vert a\vert }})^*(x,t)}\,,\\
     \label{eq:hammerh}
\smash{\vert F^*(x,\AAA(x,a)) - F^*(x,\AAA(y,b))\vert^2}
   &\sim \smash{\smash{(\varphi}_{\smash{\vert a\vert }})^*(x,\vert \AAA(x,a)-\AAA(y,b)\vert)}\,.
 \end{align}
\end{proposition} 

\begin{proof}
 For the equivalences \eqref{eq:hammera}--\eqref{eq:hammerg}, we refer to \cite[Rem.\  A.9]{BDS15}. The equivalence 
 \eqref{eq:hammerh} follows from the equivalences \eqref{eq:hammerf} and \eqref{eq:hammerg}.
 \enlargethispage{4mm}
\end{proof}

In addition, we need the following shift change result.

\begin{lemma}\label{lem:shift-change}
 For every $\varepsilon>0$, there exists $c_\varepsilon \geq 1$ (depending only
 on~$\varepsilon>0$, $p^-,p^+>1$, and $\delta\ge 0$) such that for every $t\geq 0$, $a,b\in\mathbb{R}^d$, and a.e.\ $x\in \Omega$, it holds that
 \begin{align}
 \varphi_{\vert a\vert}(x,t)&\leq c_\varepsilon\, \varphi_{\vert b\vert }(x,t)
 +\varepsilon\, \vert F(x,a) - F(x,b)\vert^2\,,\label{lem:shift-change.1}
 \\
 (\varphi_{\vert a\vert})^*(x,t)&\leq c_\varepsilon\, (\varphi_{\vert b\vert })^*(x,t)
 +\varepsilon\, \vert F(x,a) - F(x,b)\vert^2\,.\label{lem:shift-change.3}
 \end{align}
\end{lemma}

\begin{proof}
 See \cite[Rem.\  A.9]{BDS15}.
\end{proof}

\begin{remark}[Natural distance]
 \label{rem:natural_dist}
 Due to \eqref{eq:hammera}, uniformly in $u, v \in W^{1,p(\cdot)}(\Omega)$, it holds that
 \begin{align*}
  (\AAA(\cdot,\nabla u) -
  \AAA(\cdot,\nabla v),\nabla u - \nabla v)_\Omega
  &\sim
 \|F(\cdot,\nabla u)-F(\cdot,\nabla v)\|_{2,\Omega}^2 \\&\sim \rho_{\varphi_{\vert \nabla u\vert},\Omega}(\nabla u -
 \nabla v)\,.
 \end{align*}
 We refer to all three equivalent quantities as the \textup{natural distance}.  In particular, note that $\varphi_{\vert \nabla v\vert}\colon \Omega\times\mathbb{R}_{\ge 0}\to \mathbb{R}_{\ge 0}$ for every $v\in  W^{1,p(\cdot)}(\Omega)$ is a generalized $N$-function. 
\end{remark}

\begin{remark}[Conjugate natural distance]
 \label{rem:conjugate_natural_dist}
 For a.e.\ $x\in \Omega$, $\AAA(x,\cdot)\colon\mathbb{R}^d\to \mathbb{R}^d$ 
 is a continuous, strictly monotone and coercive operator, so that from the theory of monotone~operators~(\textit{cf.}~\cite{Zei90B}), it follows that $\AAA(x,\cdot)\colon\mathbb{R}^d\to \mathbb{R}^d$ for a.e.\ $x\in \Omega$ is bijective and its inverse $\AAA^{-1}(x,\cdot):\mathbb{R}^d\to \mathbb{R}^d$ continuous. Due to \eqref{eq:hammera}, uniformly in ${z, y \in L^{p'(\cdot)}(\Omega;\mathbb{R}^d)}$, it holds that
 \begin{align*}
  (\AAA^{-1}(\cdot,z)-\AAA^{-1}(\cdot,y),z-y)_\Omega
  &\sim
 \|F^*(\cdot,z)-F^*(\cdot,y)\|_{2,\Omega}^2 \\&\sim \rho_{(\varphi^*)_{\vert z\vert},\Omega}( z -
 y)\,.
 \end{align*}
 We refer to all
 three equivalent quantities as the \textup{conjugate natural distance}. In particular, note that $(\varphi^*)_{\vert y\vert}\colon \Omega\times\mathbb{R}_{\ge 0}\to \mathbb{R}_{\ge 0}$ for every $y\in L^{p'(\cdot)}(\Omega;\mathbb{R}^d)$ is a generalized $N$-function. 
\end{remark}\newpage
 	
 \subsection{Triangulations and standard finite element spaces}\label{subsec:finite_elements}

    \hspace{5mm}In what follows, let $\{\mathcal{T}_h\}_{h>0}$ be  family of  triangulations of $\Omega\subseteq \mathbb{R}^d$, $d\in\mathbb{N}$, (\textit{cf.}\ \cite{EG21}). Here,
	the parameter $h>0$ denotes the \textit{averaged mesh-size}, \textit{i.e.}, we have that
 \begin{align*}
     h\coloneqq \left(\frac{\vert \Omega\vert}{\textup{card}(\mathcal{N}_h)}\right)^{\frac{1}{d}}\,,
 \end{align*} 
 where $\mathcal{N}_h$ is the \textit{set of vertices}.
  We assume~there~exists a constant  $\omega_0\hspace{-0.1em}>\hspace{-0.1em}0$,  independent of $h>0$, such that $\max_{T\in \mathcal{T}_h}{h_T}{\smash{\rho_T^{-1}}}\hspace{-0.1em}\le\hspace{-0.1em}
    \omega_0$, where $h_T\hspace{-0.1em}\coloneqq \hspace{-0.1em}\textup{diam}(T)$ and ${\rho_T\hspace{-0.1em}\coloneqq \hspace{-0.1em}
    \sup\{r\hspace{-0.1em}>\hspace{-0.1em}0\mid \exists x\hspace{-0.1em}\in\hspace{-0.1em} T: B_r^d(x)\hspace{-0.1em}\subseteq \hspace{-0.1em}T\}}$ for all $T\hspace{-0.1em}\in\hspace{-0.1em} \mathcal{T}_h$. The smallest such constant is called the \textit{chunkiness} of  $\{\mathcal{T}_h\}_{h>0}$.   For every ${T \in \mathcal{T}_h}$, 
    let $\omega_T\coloneqq \bigcup\{T'\in \mathcal{T}_h\mid T'\cap T\neq \emptyset\} $ denote the \textit{element patch} of $T$. Then, 
    we  assume that $\textup{int}(\omega_T)$ is connected for all $T\in \mathcal{T}_h$, so that $\textup{card}(\bigcup\{T'\in\mathcal{T}_h\mid T'\subseteq \omega_T\})+\textup{card}(\bigcup\{T'\in\mathcal{T}_h\mid T\subseteq \omega_{T'}\})\leq c $, where $c>0$ depends only on the chunkiness $\omega_0>0$, and $\vert T\vert \sim
    \vert \omega_T\vert$ for all  $T\in \mathcal{T}_h$. Eventually, we define the \textit{maximum mesh-size} by $h_{\max}\coloneqq \max_{T\in \mathcal{T}_h}{h_T}>0$.

 We define interior and boundary sides of $\mathcal{T}_h$ in the following way: an \textit{interior side} is the closure of the 
    non-empty relative interior of $\partial T \cap \partial T'$,~where~${T, T'\in \mathcal{T}_h}$~are~two~adjacent~elements.
    For an interior side $S\coloneqq 
    \partial T \cap \partial T'\in \mathcal{S}_h$, where $T,T'\in \mathcal{T}_h$, the \textit{side patch} is~defined~by~$\omega_S\coloneqq  T \cup
    T'$. A \textit{boundary side} is the closure of the non-empty relative interior of
    $\partial T \cap \partial \Omega$, where $T\in \mathcal{T}_h$ denotes a boundary~element~of~$\mathcal{T}_h$.  For a boundary side $S\coloneqq  \partial T \cap \partial
    \Omega$, the side patch~is~defined~by~$\omega_S\coloneqq  T $. 
    Eventually, 
    by $\mathcal{S}_h^{i}$, we denote the \textit{set of 
 interior sides}, by $\mathcal{S}_h^{\partial}$, we denote the \textit{set of 
 boundary~sides},
    and by $\mathcal{S}_h$, we denote~the~\textit{set~of~all~sides}.

    For (Lebesgue) measurable~functions~${u,v\colon\mathcal{S}_h\to \mathbb{R}}$ and $\mathcal{M}_h\subseteq \mathcal{S}_h$, we write
    \begin{align*}
        (u,v)_{\mathcal{M}_h}\coloneqq \sum_{S\in \mathcal{M}_h}{(u,v)_S}\,,\quad\text{ where }(u,v)_S\coloneqq\int_S{uv\,\mathrm{d}s}\,,
    \end{align*}
    whenever all integrals are well-defined. Analogously, for  (Lebesgue) measurable vector fields $z,y\colon \mathcal{S}_h\hspace{-0.1em}\to\hspace{-0.1em} \mathbb{R}^d$ \hspace{-0.1mm}and \hspace{-0.1mm}$\mathcal{M}_h\subseteq \mathcal{S}_h$, we write ${(z,y)_{\mathcal{M}_h}\hspace{-0.15em}\coloneqq\hspace{-0.15em} \sum_{S\in \mathcal{M}_h}{(z,y)_S}}$,~where~${(z,y)_S\coloneqq\int_S{z\cdot y\,\mathrm{d}s}}$.\enlargethispage{3mm}
	
	For $k\in \mathbb{N}\cup\{0\}$ and $T\in \mathcal{T}_h$, let $\mathcal{P}_k(T)$ denote the set of polynomials of maximal~degree~$k$~on~$T$. Then, for $k\in \mathbb{N}\cup\{0\}$~and $l\in \mathbb{N}$, the sets of  
 continuous and~\mbox{element-wise}~polynomial functions or vector~fields,~respectively, are defined by
	\begin{align*}
	\begin{aligned}
	\mathcal{S}^k(\mathcal{T}_h)^l&\coloneqq 	\big\{v_h\in C^0(\overline{\Omega};\mathbb{R}^l)\hspace*{-3mm}&&\mid v_h|_T\in\mathcal{P}_k(T)^l\text{ for all }T\in \mathcal{T}_h\big\}\,,\\
	\mathcal{L}^k(\mathcal{T}_h)^l&\coloneqq \big\{v_h\in L^\infty(\Omega;\mathbb{R}^l)\hspace*{-3mm}&&\mid v_h|_T\in\mathcal{P}_k(T)^l\text{ for all }T\in \mathcal{T}_h\big\}\,.
	\end{aligned}
	\end{align*}
	The \textit{element-wise constant mesh-size function} $h_\mathcal{T}\in \mathcal{L}^0(\mathcal{T}_h)$ is defined~by~${h_\mathcal{T}|_T\coloneqq h_T}$~for~all~${T\in \mathcal{T}_h}$.
	The \textit{side-wise constant mesh-size function} $h_\mathcal{S}\in \mathcal{L}^0(\mathcal{S}_h)$ is defined~by~${h_\mathcal{S}|_S\coloneqq h_S}$~for~all~${S\in \mathcal{S}_h}$,  where $h_S\coloneqq \textup{diam}(S)$ for all $S\in \mathcal{S}_h$.
	For every $T\in \mathcal{T}_h$ and $S\in \mathcal{S}_h$, we denote by 
 \begin{align*}
     x_T\coloneqq \frac{1}{d+1}\sum_{\nu\in \mathcal{N}_h\cap T}{\nu }\quad\text{ and }\quad x_S\coloneqq \frac{1}{d}\sum_{\nu \in \mathcal{N}_h\cap S}{\nu}\,,
 \end{align*}
 the  barycenters~of~$T$ and $S$, respectively. The \hspace{-0.2mm}\textit{(local) \hspace{-0.4mm}$L^2$\hspace{-0.2mm}-projection \hspace{-0.2mm}operator}  \hspace{-0.2mm}onto \hspace{-0.2mm}\mbox{element-wise} \hspace{-0.2mm}constant \hspace{-0.2mm}functions \hspace{-0.2mm}or~\hspace{-0.2mm}vector \hspace{-0.2mm}fields, \mbox{respectively}, $\Pi_h\colon L^1(\Omega;\mathbb{R}^l)\to \mathcal{L}^0(\mathcal{T}_h)^l$,
 for every $v\in L^1(\Omega;\mathbb{R}^l)$ is defined by 
 \begin{align*}
     \Pi_hv\coloneqq \sum_{T\in \mathcal{T}_h}{ \langle v\rangle_T \chi_T}\,.
 \end{align*}
 The \textit{element-wise~gradient}  
 $\nabla_{\!h}\colon \mathcal{L}^1(\mathcal{T}_h)^l\to \mathcal{L}^0(\mathcal{T}_h)^{l\times d}$, for every $v_h\in \mathcal{L}^1(\mathcal{T}_h)^l$ is defined by 
 \begin{align*}
    \nabla_{\!h}v_h\coloneqq \sum_{T\in \mathcal{T}_h}{ \nabla(v_h|_T) \chi_T}\,.
 \end{align*} 
 	
  For  every $v_h\in \mathcal{L}^k(\mathcal{T}_h)$ and $S\in\mathcal{S}_h$, the \textit{jump across} $S$ is defined by 
 	\begin{align*}
 		\jump{v_h}_S\coloneqq\begin{cases}
 			v_h|_{T_+}-v_h|_{T_-}&\text{ if }S\in \mathcal{S}_h^{i}\,,\text{ where }T_+, T_-\in \mathcal{T}_h\text{ satisfy }\partial T_+\cap\partial T_-=S\,,\\
 			 v_h|_T&\text{ if }S\in\mathcal{S}_h^{\partial}\,,\text{ where }T\in \mathcal{T}_h\text{ satisfies }S\subseteq \partial T\,.
 		\end{cases}
 	\end{align*}
 
 	For  every $y_h\in (\mathcal{L}^k(\mathcal{T}_h))^d$ and $S\in\mathcal{S}_h$, the \textit{normal jump across} $S$ is defined by\vspace*{-0.5mm}
 			\begin{align*}
 			\jump{y_h\cdot n}_S\coloneqq\begin{cases}
 			y_h|_{T_+}\!\cdot n_{T_+}+y_h|_{T_-}\!\cdot n_{T_-}&\text{ if }S\in \mathcal{S}_h^{i}\,,\text{ where }T_+, T_-\in \mathcal{T}_h\text{ satisfy }\partial T_+\cap\partial T_-=S\,,\\
 			y_h|_T\cdot n&\text{ if }S\in\mathcal{S}_h^{\partial}\,,\text{ where }T\in \mathcal{T}_h\text{ satisfies }S\subseteq \partial T\,,
 			\end{cases}
 	\end{align*}
 where, for every $T\in \mathcal{T}_h$, $\smash{n_T\colon\partial T\to \mathbb{S}^{d-1}}$ denotes the outward unit normal vector field to $ T$.\enlargethispage{2mm}
	
	\subsubsection{Crouzeix--Raviart element} 
	
	\hspace{5mm}The \textit{Crouzeix--Raviart finite element space} (\textit{cf.}\ \cite{CR73}) is defined as the space of element-wise affine functions that are continuous in the barycenters of interior element sides, \textit{i.e.},
	\begin{align*}\mathcal{S}^{1,\textit{cr}}(\mathcal{T}_h)\coloneqq \big\{v_h\in \mathcal{L}^1(\mathcal{T}_h)\mid \jump{v_h}_S(x_S)=0\text{ for all }S\in \mathcal{S}_h^{i}\big\}\,.
	\end{align*}
	The Crouzeix--Raviart finite element space with homogeneous Dirichlet boundary condition is defined as the space of Crouzeix--Raviart functions that vanish in the barycenters of boundary sides that belong to  $\Gamma_D$, \textit{i.e.},
	\begin{align*}
			\mathcal{S}^{1,\textit{cr}}_D(\mathcal{T}_h)\coloneqq \big\{v_h\in\smash{\mathcal{S}^{1,\textit{cr}}(\mathcal{T}_h)}\mid v_h(x_S)=0\text{ for all }S\in \mathcal{S}_h\cap \Gamma_D\big\}\,.
	\end{align*}
	
	\subsubsection{Raviart--Thomas element}
 
	\hspace{5mm}The \textit{lowest order Raviart--Thomas finite element space} (\textit{cf.}\ \cite{RT75}) is defined as the space of element-wise  affine vector fields that have continuous constant normal components on  interior elements sides, \textit{i.e.},
	\begin{align*}
 \mathcal{R}T^0(\mathcal{T}_h)\coloneqq \big\{y_h\in \mathcal{L}^1(\mathcal{T}_h)^d\mid &\,\smash{y_h|_T\cdot  n_T=\textup{const}\text{ on }\partial T\text{ for all }T\in \mathcal{T}_h\,,}\\ 
 &\smash{	\jump{y_h\cdot n}_S=0\text{ on }S\text{ for all }S\in \mathcal{S}_h^{i}\big\}\,.}
	\end{align*}
    The	Raviart--Thomas finite element space with homogeneous slip boundary condition is defined as the space of Raviart--Thomas functions that have vanishing normal components~on~$\Gamma_N$, \textit{i.e.},
	\begin{align*}
		\mathcal{R}T^{0}_N(\mathcal{T}_h)\coloneqq \big\{y_h\in	\mathcal{R}T^0(\mathcal{T}_h)\mid y_h\cdot n=0\text{ on }\Gamma_N\big\}\,.
	\end{align*} 
	
	\subsubsection{Discrete integration-by-parts formula} 
 
 	\hspace{5mm}For every $v_h\in \mathcal{S}^{1,\textit{cr}}(\mathcal{T}_h)$ and ${y_h\in \mathcal{R}T^0(\mathcal{T}_h)}$, we have the \textit{discrete integration-by-parts
	formula}
	\begin{align}
	(\nabla_{\!h}v_h,\Pi_h y_h)_\Omega+(\Pi_h v_h,\,\textup{div}\,y_h)_\Omega=(v_h,y_h\cdot n)_{\partial\Omega}\,,\label{eq:pi}
	\end{align}
	which follows from the fact that  for every $y_h\in \mathcal{R}T^0(\mathcal{T}_h)$, it holds that $y_h|_T\cdot n_T=\textrm{const}$~on~$\partial T$ for all $T\in \mathcal{T}_h$ and	$\jump{y_h\cdot n}_S=0$ on $S$ for all $S\in \mathcal{S}_h^{i}$, and for every~${v_h\in \mathcal{S}^{1,\textit{\textrm{cr}}}(\mathcal{T}_h)}$,~it~holds~that $\int_{S}{\jump{v_h}_S\,\textup{d}s}\hspace{-0.15em}=\hspace{-0.15em}\jump{v_h}_S(x_S)\hspace{-0.15em}=\hspace{-0.15em}0$ for all $S\hspace{-0.15em}\in\hspace{-0.15em} \mathcal{S}_h^{i}$.
 For~every~$v_h\hspace{-0.15em}\in\hspace{-0.15em} \smash{\mathcal{S}^{1,\textit{\textrm{cr}}}_D(\mathcal{T}_h)}$~and~${y_h\hspace{-0.15em}\in \hspace{-0.15em}\smash{\mathcal{R}T^0_N(\mathcal{T}_h)}}$,~\eqref{eq:pi}~reads
	\begin{align}
		(\nabla_{\!h}v_h,\Pi_h y_h)_\Omega=-(\Pi_h v_h,\,\textup{div}\,y_h)_\Omega\,.\label{eq:pi0}
	\end{align} 
    In~addition, appealing to \cite[Sec.\ 2.4]{BW21}, there holds the \textit{discrete Helmholtz decomposition}
    \begin{align}
        \mathcal{L}^0(\mathcal{T}_h)^d=\textup{ker}(\textup{div}|_{\smash{\mathcal{R}T^0_N(\mathcal{T}_h)}})\oplus \nabla_{\!h}(\smash{\mathcal{S}^{1,\textit{\textrm{cr}}}_D(\mathcal{T}_h)})\,,\label{eq:decomposition}
    \end{align}
    which shows that for every  $y_h\in \mathcal{L}^0(\mathcal{T}_h)^d$, the following implication applies:
    \begin{align}
        (y_h,\nabla_{\!h} v_h)_{\Omega}=0\quad\text{ for all }v_h\in \mathcal{S}^{1,\textit{\textrm{cr}}}_D(\mathcal{T}_h)\qquad\Rightarrow \qquad y_h\in \mathcal{R}T^0_N(\mathcal{T}_h)\cap \mathcal{L}^0(\mathcal{T}_h)^d\,.\label{eq:decomposition.1}
    \end{align}
    The implication \eqref{eq:decomposition.1} is of crucial importance in the derivation of discrete strong duality relations and a discrete recontruction formula in Proposition \ref{prop:discrete_convex_duality}.\newpage
	 
 \subsection{$p(\cdot)$-Dirichlet problem}
	
	\qquad In this section, we briefly review  the variational formulation, the primal formulation, and the dual formulation of the $p(\cdot)$-Dirichlet problem \eqref{eq:pDirichlet}.\vspace{-2mm}
	
	\subsubsection{Variational problem}\vspace{-0.5mm}
		
		\qquad Given 
		\hspace{-0.1mm}a \hspace{-0.1mm}right-hand \hspace{-0.1mm}side \hspace{-0.1mm}$f\hspace{-0.1em}\in \hspace{-0.1em} L^{p'(\cdot)}(\Omega)$, \hspace{-0.1mm}where \hspace{-0.1mm}$p\hspace{-0.1em}\in\hspace{-0.1em} \mathcal{P}^{\infty}(\Omega)$ \hspace{-0.1mm}with \hspace{-0.1mm}$p^-\hspace{-0.1em}>\hspace{-0.1em}1$, \hspace{-0.1mm}the \hspace{-0.1mm}\textit{$p(\cdot)$-Dirichlet~\hspace{-0.1mm}\mbox{problem}} seeks for a function ${u\in {W^{1,p(\cdot)}_D(\Omega)}}$ such that for every $v\in {W^{1,p(\cdot)}_D(\Omega)}$, it holds that
		\begin{align}
		 (\AAA(\cdot,\nabla u),\nabla v)_\Omega=(f,v)_\Omega\,.\label{eq:pDirichletW1p}
		\end{align}
		The theory of monotone operators (\textit{cf.}\ \cite{Zei90B}) proves the existence of a  unique solution to \eqref{eq:pDirichletW1p}. In what follows, we~reserve~the~notation~$u\in {W^{1,p(\cdot)}_D(\Omega)}$~for~this~solution.\vspace{-1mm}
		
	\subsubsection{Minimization problem and convex duality relations}
	
	\hspace{5mm}The variational problem \eqref{eq:pDirichletW1p} emerges as an optimality condition of an equivalent convex minimization problem, leading to a primal and a dual formulation, and convex duality relations.\smallskip
	
	\hspace{5mm}\textit{Primal problem.} The problem \eqref{eq:pDirichletW1p} is equivalent to the minimization of the \textit{$p(\cdot)$-Dirichlet energy}    $I\colon {W^{1,p(\cdot)}_D(\Omega)}\to \mathbb{R}$, for every $v\in {W^{1,p(\cdot)}_D(\Omega)}$ defined by
	\begin{align}\label{eq:pDirichletPrimal}
	 I(v)\coloneqq \rho_{\varphi,\Omega}( \nabla v)-(f,v)_\Omega\,.
	\end{align}
	We will always refer to the minimization of the $p(\cdot)$-Dirichlet energy \eqref{eq:pDirichletPrimal}~as~the~\textit{primal~\mbox{problem}}.
	Since the $p(\cdot)$-Dirichlet energy \eqref{eq:pDirichletPrimal} is proper, 
	strictly convex, weakly coercive, 
	and~lower~semi-continuous, the existence of a unique minimizer of \eqref{eq:pDirichletPrimal}, called  \textit{primal solution}, follows using the \hspace*{-0.1mm}direct \hspace*{-0.1mm}method \hspace*{-0.1mm}in \hspace*{-0.1mm}the \hspace*{-0.1mm}calculus \hspace*{-0.1mm}of \hspace*{-0.1mm}variations \hspace*{-0.1mm}(\textit{cf.}\ \hspace*{-0.1mm}\cite{Dac08}).
  \hspace*{-0.1mm}In \hspace*{-0.1mm}particular,
	\hspace*{-0.1mm}since~\hspace*{-0.1mm}the~\hspace*{-0.1mm}\mbox{$p(\cdot)$-Dirichlet~\hspace*{-0.1mm}energy} is Fr\'echet~differentiable~and for every $v,w\in W^{1,p(\cdot)}_D(\Omega)$, it holds that
	\begin{align*}
	 \langle DI(v),w\rangle_{{W^{1,p(\cdot)}_D(\Omega)}}=(\AAA(\cdot,\nabla v),\nabla w)_\Omega\,,
	\end{align*}
	the optimality condition of the primal problem and the strict convexity of the $p(\cdot)$-Dirichlet~energy imply that $u\in W^{1,p(\cdot)}_D(\Omega)$ is the unique minimizer of the $p(\cdot)$-Dirichlet energy.\smallskip
	
	\hspace{5mm}\textit{Dual problem.} \hspace{-0.1mm}Generalizing  \hspace{-0.1mm}in \hspace{-0.1mm}\mbox{\cite[p.\ \hspace{-0.1mm}113 \hspace{-0.1mm}ff.]{ET99}} \hspace{-0.1mm}to \hspace{-0.1mm}the \hspace{-0.1mm}spaces \hspace{-0.1mm}\eqref{eq:spaces}, \hspace{-0.1mm}one~\hspace{-0.1mm}finds~\hspace{-0.1mm}that
		the (Fenchel) dual problem consists in the maximization of the functional $D\colon W^{p'(\cdot)}_{N}(\textup{div};\Omega)\to \mathbb{R}\cup\{-\infty\}$, for every $y\in W^{p'(\cdot)}_N(\textup{div};\Omega)$ defined by
		\begin{align}
			D(y)\coloneqq -\rho_{\varphi^*,\Omega}(y) -I_{\{-f\}}(\textup{div}\,y)\,,\label{eq:pDirichletDual}
		\end{align}
		where $I_{\{-f\}}\colon L^{p'(\cdot)}(\Omega)\to  \mathbb{R}\cup\{+\infty\}$ for every $g\in L^{p'(\cdot)}(\Omega) $ is defined by 
 \begin{align}
  I_{\{-f\}}(g)\coloneqq\begin{cases}
  0&\text{ if }g=-f\text{ a.e. in }\Omega\,,\\+\infty&\text{ else}\,.
  \end{cases}
 \end{align}
		Due to \cite[Rem.\  4.1 (4.21), p.\ 60]{ET99}, the dual problem admits a unique solution $\smash{z\in W^{p'(\cdot)}_N(\textup{div};\Omega)}$, \textit{i.e.}, a maximizer of \eqref{eq:pDirichletDual}, called  \textit{dual solution}, and a \textit{strong duality relation} applies, \textit{i.e.},\enlargethispage{2mm}
 \begin{align*}
  I(u)=D(z)\,.
 \end{align*}
  In addition, there hold the \textit{convex optimality relations} (\textit{cf.} \cite[Rem.\  4.1 (4.22)--(4.25), p.\ 60]{ET99})
		\begin{alignat}{2}
				\textup{div}\,z &=-f&&\quad\text{ a.e.\ in }\Omega\label{eq:pDirichletOptimality1.1}\,,\\ 
 z&=\AAA(\cdot,\nabla u)&&\quad\text{ a.e.\ in } \Omega\,.\label{eq:pDirichletOptimality1.2}
		\end{alignat}
		By the Fenchel--Young identity (\textit{cf.}\ \cite[Prop.\ 5.1, p.\ 21]{ET99}), the relation \eqref{eq:pDirichletOptimality1.2} is equivalent to 
		\begin{align}
			 z\cdot\nabla u=\varphi^*(\cdot,\vert z\vert )+\varphi(\cdot,\vert \nabla u\vert )\quad\textup{ a.e.\ in }\Omega\,.\label{eq:pDirichletOptimality2}
		\end{align}

 \subsection{Discrete $p_h(\cdot)$-Dirichlet problem}

 \hspace{5mm}One important aspect in the numerical approximation of the $p(\cdot)$-Dirichlet~problem~\eqref{eq:pDirichletW1p} consists in the discretization of the $x$-dependent non-linearity. Here, it is convenient to use a simple one-point quadrature rule. More precisely, if $p\in C^0(\overline{\Omega})$ with $p^->1$,~then~we~define~the~element-wise constant variable exponent $p_h\in \mathcal{L}^0(\mathcal{T}_h)$, the generalized $N$-function $\varphi_h\colon\Omega\times\mathbb{R}_{\ge 0}\to \mathbb{R}_{\ge 0}$, and the non-linear~operators ${\AAA_h,F_h,F_h^*\colon\Omega\times\mathbb{R}^d\hspace{-0.05em}\to\hspace{-0.05em} \mathbb{R}^d}$ for every  ${a\in \mathbb{R}^d}$, $T\in  \mathcal{T}_h$, and a.e. $x\in T$~by 
 \begin{align}
  \begin{aligned}
  p_h(x)&\coloneqq p(\xi_T)\,,\qquad
  &&\hspace{-2mm}\varphi_h(x,\vert a\vert )\coloneqq \varphi(\xi_T,\vert a\vert)\,,\\
  \AAA_h(x,a)&\coloneqq \AAA(\xi_T,a)\,,\qquad
  &&F_h(x,a)\coloneqq F(\xi_T,a)\,,\qquad
  F_h^*(x,a)\coloneqq F^*(\xi_T,a)\,,
  \end{aligned}\label{def:A_h}
 \end{align}
 where $\xi_T\in T$ is an arbitrary quadrature point, \textit{e.g.}, the barycenter of the element $T$.

 \begin{remark}\label{rem:uniform}
    Since the hidden constants in all equivalences in Section \ref{sec:basic} depend only on $p^-,p^+>1$ and $\delta\ge 0$ and since $p^-\leq p_h^-\le p^+_h\leq p^+$ a.e. in $\Omega$ for all $h>0$,     
    the same equivalences apply to the discretizations \eqref{def:A_h} with the hidden constants depending~only~on~$p^-,p^+\in (1,\infty)$.
 \end{remark}
 
 \subsubsection{$\mathcal{S}^1_D(\mathcal{T}_h)$-approximation of the $p(\cdot)$-Dirichlet problem}
		
		\qquad Given a right-hand side $f\hspace{-0.1em}\in\hspace{-0.1em} L^{p'(\cdot)}(\Omega)$, where $p\hspace{-0.1em}\in\hspace{-0.1em} C^0(\overline{\Omega})$ with $p^-\hspace{-0.1em}>\hspace{-0.1em}1$,~the~\mbox{$\smash{\mathcal{S}^1_D(\mathcal{T}_h)}$-approximation} of the $p(\cdot)$-Dirichlet problem, where $\mathcal{S}^1_D(\mathcal{T}_h)\vcentcolon=\mathcal{S}^1(\mathcal{T}_h)\cap {W^{1,p(\cdot)}_D(\Omega)}$, seeks for $u_h^{\textit{c}}\in \mathcal{S}^1_D(\mathcal{T}_h)$~such~that for every $v_h\in \mathcal{S}^1_D(\mathcal{T}_h)$,~it~holds~that
		\begin{align}
		 (\AAA_h(\cdot,\nabla u_h^{\textit{c}}),\nabla v_h)_\Omega=(f,v_h)_\Omega\,.\label{eq:pDirichletS1D}
		\end{align}
		The theory of monotone operators (\textit{cf.}\ \cite{Zei90B}) 
        proves the existence of a unique~solution~to~\eqref{eq:pDirichletS1D}. In what follows, we reserve~the~notation~$\smash{u_h^{\textit{c}}\in \mathcal{S}^1_D(\mathcal{T}_h)}$~for~this~solution. 
 
  In \cite{BDS15}, the following best-approximation result was derived:
		
		\begin{theorem}[Best-approximation]\label{P1_best-approx}
  Let $p\in C^{0,\alpha}(\overline{\Omega})$ with $\alpha\in (0,1]$ and $p^->1$~and~let~$\delta\ge 0$. 
  Then, there exists some $s>1$, which can chosen to be close to $1$ if $h_{\max}>0$ is close~to~$0$,~such~that if 
    $u\in W^{1,p(\cdot)s}(\Omega)$, then 
		 \begin{align*}
		 \|F_h(\cdot, \nabla u_h^{\textit{c}})- F_h(\cdot, \nabla u)\|_{2,\Omega}^2\lesssim \inf_{v_h\in \mathcal{S}^1_D(\mathcal{T}_h)}{\|F_h(\cdot, \nabla v_h)-F_h(\cdot,\nabla u)\|_{2,\Omega}^2}+h_{\max}^{2\alpha}\,\big(1+\rho_{p(\cdot)s,\Omega}(\nabla u)\big)\,,
		 \end{align*}
         where the hidden constant in $\lesssim$ also depends on $s>1$ and the chunkiness $\omega_0>0$. 
		\end{theorem}

		\begin{proof} In \cite[Lem.\  4.7]{BDS15}, only the case $\xi_T\coloneqq \textrm{arg\,min}_{x\in T}{p(x)}$ for all $T\in \mathcal{T}_h$ has been considered. However, an  analysis of the proof of \cite[Lem.\  4.7]{BDS15} reveals that this particular quadrature rule is not needed there.
		\end{proof}

        Resorting \hspace{-0.1mm}in \hspace{-0.1mm}Theorem \hspace{-0.1mm}\ref{P1_best-approx} \hspace{-0.1mm}to \hspace{-0.1mm}the \hspace{-0.1mm}approximation \hspace{-0.1mm}properties \hspace{-0.1mm}of \hspace{-0.1mm}the \hspace{-0.1mm}\textit{Scott--Zhang~\hspace{-0.1mm}quasi-interpolation \hspace{-0.1mm}operator}  $\Pi_h^{\textit{sz}}\colon {W^{1,p(\cdot)}_D(\Omega)}\hspace*{-0.15em}\to\hspace*{-0.15em}\mathcal{S}^1_D(\mathcal{T}_h)$ (\textit{cf.}\ \cite{SZ90}), one arrives at  the following \textit{a priori} error estimate~for the $\mathcal{S}^1_D(\mathcal{T}_h)$-approximation \eqref{eq:pDirichletS1D} of the $p(\cdot)$-Dirichlet problem \eqref{eq:pDirichletW1p}.
		
		\begin{theorem}[\textit{a priori} error estimate]\label{P1_apriori}
		 Let $p\in C^{0,\alpha}(\overline{\Omega})$ with $\alpha\in (0,1]$ and $p^->1$~and~let~$\delta\ge 0$.  Moreover, let $F(\cdot,\nabla u)\in W^{1,2}(\Omega;\mathbb{R}^d)$.  
        Then, there exists some $s>1$, which can chosen to be close to $1$ if $h_{\max}>0$ is close to $0$, such that 
		 \begin{align*}
		 \|F_h(\cdot,\nabla u_h^{\textit{c}})-F_h(\cdot,\nabla u)\|_{2,\Omega}^2\lesssim h_{\max}^2\,\|\nabla F(\cdot,\nabla u)\|_{2,\Omega}^2+h_{\max}^{2\alpha}\,\big(1+\rho_{p(\cdot)s,\Omega}(\nabla u)\big)\,,
		 \end{align*}
        where the hidden constant in $\lesssim$ also depends on $s>1$ and the chunkiness $\omega_0>0$.\enlargethispage{1mm}
		\end{theorem}
		
		\begin{proof}
		 See \cite[Thm.\ 4.8]{BDS15}.
		\end{proof}

        \begin{remark}\label{rem:ps}
            If $F(\cdot,\nabla u)\in W^{1,2}(\Omega;\mathbb{R}^d)$, then  $F(\cdot,\nabla u)\in L^{2^*}(\Omega;\mathbb{R}^d)$, where $2^*\coloneqq \frac{2d}{d-2}$ if $2<d$ and $2^*\in [1,\infty)$ if $2\ge d$. As a result, for $s>1$ close to $1$, \textit{i.e.}, $h_{\max}>0$  close to $0$,~it~holds~that $\nabla u\in L^{p(\cdot)s}(\Omega;\mathbb{R}^d)$. 
            Similarly, if $F^*(\cdot,z)\in W^{1,2}(\Omega;\mathbb{R}^d)$, then for $s>1$ close to $1$, \textit{i.e.}, $h_{\max}>0$  close to $0$, it holds that $z \in L^{p'(\cdot)s}(\Omega;\mathbb{R}^d)$. 
        \end{remark}

 \subsection{$\smash{\mathcal{S}^{1,\textit{cr}}_D(\mathcal{T}_h)}$-approximation of the $p(\cdot)$-Dirichlet problem}

		 \hspace{5mm}Given a right-hand side $f\hspace{-0.05em}\in \hspace{-0.05em}L^{p'(\cdot)}(\Omega)$, $p\hspace{-0.05em}\in\hspace{-0.05em} C^0(\overline{\Omega})$ with $p^-\hspace{-0.05em}>\hspace{-0.05em}1$, and setting ${f_h\hspace{-0.05em}\coloneqq \hspace{-0.05em}\Pi_h f\hspace{-0.05em}\in \hspace{-0.05em}\mathcal{L}^0(\mathcal{T}_h)}$,
  the $\mathcal{S}^{1,\textit{cr}}_D(\mathcal{T}_h)$-approximation of the $p(\cdot)$-Dirichlet problem seeks for $u_h^{\textit{cr}}\in \mathcal{S}^{1,\textit{cr}}_D(\mathcal{T}_h)$ such that for every $v_h\in \mathcal{S}^{1,\textit{cr}}_D(\mathcal{T}_h)$, it holds that
		\begin{align}
		 (\AAA_h(\cdot,\nabla_{\!h}u_h^{\textit{cr}}),\nabla_{\! h} v_h)_\Omega=(f_h,\Pi_h v_h)_\Omega\,.\label{eq:pDirichletS1crD}
		\end{align}
		The theory of monotone operators (\textit{cf.}\ \cite{Zei90B}) 
        proves the existence of a unique~solution~to~\eqref{eq:pDirichletS1crD}. In what follows, we reserve~the~notation~$u_h^{\textit{cr}}\in \smash{\mathcal{S}^{1,\textit{cr}}_D(\mathcal{T}_h)}$~for~this~solution. 
		\subsubsection{Discrete minimization problem and discrete convex duality relations}\enlargethispage{5mm}
		
		\hspace{5mm}The discrete variational problem \eqref{eq:pDirichletS1crD} emerges as an optimality condition of an equivalent convex minimization problem.
		
		\hspace{5mm}\textit{Discrete primal problem.} The problem \eqref{eq:pDirichletS1crD} is equivalent to the minimization of the \textit{discrete $p_h(\cdot)$-Dirichlet energy} ${I_h^{\textit{cr}} \colon \mathcal{S}^{1,\textit{cr}}_D(\mathcal{T}_h) \to \mathbb{R}}$, for every $v_h\in \smash{\mathcal{S}^{1,\textit{cr}}_D(\mathcal{T}_h)}$ defined by
	 \begin{align}\label{eq:pDirichletPrimalCR}
		 I_h^{\textit{cr}}(v_h)\coloneqq \rho_{\varphi_h,\Omega}(\nabla_{\!h} v_h)-(f_h,\Pi_hv_h)_\Omega\,.
		\end{align}
		Hereinafter, we refer to the minimization of the discrete $p_h(\cdot)$-Dirichlet energy \eqref{eq:pDirichletPrimalCR}~as~the~\textit{discrete primal problem}.
		Since the discrete $p_h(\cdot)$-Dirichlet energy \eqref{eq:pDirichletPrimalCR} is proper, 
		strictly convex, weakly coercive, 
		and lower semi-continuous, the existence of a unique minimizer of \eqref{eq:pDirichletPrimalCR}, called \textit{discrete primal solution}, follows using 
  the direct method in the calculus of variations (\textit{cf.}\ \cite{Dac08}). More precisely,
	since the discrete $p_h(\cdot)$-Dirichlet energy \eqref{eq:pDirichletPrimalCR} is Fr\'echet differentiable and for every $v_h,w_h\in \smash{\mathcal{S}^{1,\textit{cr}}_D(\mathcal{T}_h)}$,~it~holds~that
	\begin{align*}
	 \langle DI_h^{\textit{cr}}(v_h),w_h\rangle_{\smash{\mathcal{S}^{1,\textit{cr}}_D(\mathcal{T}_h)}}=(\AAA_h(\cdot,\nabla_{\!h} v_h),\nabla_{\! h} w_h)_\Omega\,,
	\end{align*}
	the \hspace{-0.2mm}optimality \hspace{-0.2mm}condition \hspace{-0.2mm}of \hspace{-0.2mm}the \hspace{-0.2mm}discrete \hspace{-0.2mm}primal \hspace{-0.2mm}problem \hspace{-0.2mm}and \hspace{-0.2mm}convexity \hspace{-0.2mm}of \hspace{-0.2mm}the \hspace{-0.2mm}discrete~\hspace{-0.2mm}\mbox{$p_h(\cdot)$-Dirichlet} energy~\eqref{eq:pDirichletPrimalCR} imply that $\smash{u_h^{\textit{cr}}\in \smash{\mathcal{S}^{1,\textit{cr}}_D(\mathcal{T}_h)}}$ solves the discrete primal problem, \textit{i.e.}, is the unique minimizer of the discrete $p_h(\cdot)$-Dirichlet energy~\eqref{eq:pDirichletPrimalCR}.
	
		\hspace{5mm}\textit{Discrete dual problem.} The discrete dual problem consists in the maximization of the functional $D_h^{\textit{rt}}:\mathcal{R}T^0_N(\mathcal{T}_h)\to \mathbb{R}\cup\{-\infty\}$, for every $y_h\in \mathcal{R}T^0_N(\mathcal{T}_h)$ defined by
	\begin{align}
		D_h^{\textit{rt}}(y_h)\coloneqq -\rho_{\varphi_h^*,\Omega}(\Pi_hy_h)-I_{\{-f_h\}}(\textup{div}\,y_h)\,.\label{eq:pDirichletDualCR}
	\end{align}
 The following proposition establishes the well-posedness of the discrete dual problem, \textit{i.e.}, the existence of a maximizer, called \textit{discrete dual solution},  and
 discrete strong duality.~In~addition, it provides a reconstruction formula for~this~maximizer from the discrete primal solution.

 \begin{proposition}\label{prop:discrete_convex_duality}
 The following statements apply:
 \begin{itemize}[noitemsep,topsep=1pt,labelwidth=\widthof{(ii)},leftmargin=!]
 \item[(i)] There holds a \textup{discrete weak duality relation}, \textit{i.e.}, it holds that
 \begin{align}
  \inf_{v_h\in \mathcal{S}^{1,\textit{cr}}_D(\mathcal{T}_h)}{I_h^{\textit{cr}}(v_h)}\ge \sup_{y_h\in \mathcal{R}T^0_N(\mathcal{T}_h)}{D_h^{\textit{rt}}(y_h)}\,.\label{eq:discrete_weak_duality}
 \end{align}
  \item[(ii)] The discrete flux $z_h^{\textit{rt}}\in \mathcal{L}^1(\mathcal{T}_h)^d$, defined via the \textup{generalized Marini formula}
  \begin{align}
		z_h^{\textit{rt}}= \AAA_h(\cdot,\nabla_{\! h}u_h^{\textit{cr}})-\frac{f_h}{d}\big(\textup{id}_{\mathbb{R}^d}-\Pi_h\textup{id}_{\mathbb{R}^d}\big)\quad\text{ in } \mathcal{R}T^0_N(\mathcal{T}_h)\,,\label{eq:gen_marini}
	\end{align}
 satisfies $z_h^{\textit{rt}}\in \mathcal{R}T^0_N(\mathcal{T}_h)$ and the \textup{discrete convex optimality relations}
 \begin{alignat}{2}
		\textup{div}\,z_h^{\textit{rt}}&=-f_h&&\quad\text{ a.e.\ in }\Omega\,,\label{eq:pDirichletOptimalityCR1.1}\\
 \Pi_hz_h^{\textit{rt}}&=\AAA_h(\cdot,\nabla_{\! h}u_h^{\textit{cr}})&&\quad\text{ a.e.\ in }\Omega\,.\label{eq:pDirichletOptimalityCR1.2}
	\end{alignat}
 \item[(iii)] The discrete flux $z_h^{\textit{rt}}\in \mathcal{R}T^0_N(\mathcal{T}_h)$ is the unique maximizer of \eqref{eq:pDirichletDualCR} and  a \textup{discrete strong duality relation} applies, \textit{i.e.}, it holds that
 \begin{align}\label{eq:discrete_strong_duality}
     I_h^{\textit{cr}}(u_h^{\textit{cr}})=D_h^{\textit{rt}}(z_h^{\textit{rt}})\,.
 \end{align}
 \end{itemize}
 \end{proposition}

By the Fenchel--Young identity (\textit{cf.}\ \cite[Prop.\ 5.1, p. 21]{ET99}), the relation \eqref{eq:pDirichletOptimalityCR1.2}~is~equivalent~to
	\begin{align}\label{eq:pDirichletOptimalityCR2}
	  	\Pi_hz_h^{\textit{rt}}\cdot\nabla_{\! h} u_h^{\textit{cr}}=\varphi_h^*(\cdot,\vert\Pi_hz_h^{\textit{rt}}\vert)+\varphi_h(\cdot,\vert\nabla_{\! h}u_h^{\textit{cr}}\vert)\quad\text{ a.e.\ in }\Omega\,.
	\end{align}
 
 \begin{proof}
 \textit{ad (i).} Using element-wise for each $T\in \mathcal{T}_h$ that $\varphi(\xi_T,\cdot)=\varphi^{**}(\xi_T,\cdot)$, the definition of the convex conjugate, and the discrete integration-by-parts formula \eqref{eq:pi0}, we find that
 \begin{align*}
  \inf_{v_h\in \mathcal{S}^{1,\textit{cr}}_D(\mathcal{T}_h)}{I_h^{\textit{cr}}(v_h)}&=\inf_{v_h\in \mathcal{S}^{1,\textit{cr}}_D(\mathcal{T}_h)}{\rho_{\varphi_h^{**},\Omega}(\nabla_{\! h} v_h)-(f_h,\Pi_h v_h)_\Omega}
 \\
 &=
  \inf_{v_h\in \mathcal{S}^{1,\textit{cr}}_D(\mathcal{T}_h)}{ \sup_{\overline{y}_h\in \mathcal{L}^0(\mathcal{T}_h)^d}{(\overline{y}_h,\nabla_{\! h} v_h)_\Omega-\rho_{\varphi_h^*,\Omega}(\overline{y}_h)-(f_h,\Pi_h v_h)_\Omega}}
 \\&\ge 
  \inf_{v_h\in \mathcal{S}^{1,\textit{cr}}_D(\mathcal{T}_h)}{ \sup_{y_h\in \mathcal{R}T^0_N(\mathcal{T}_h)}{-\rho_{\varphi_h^*,\Omega}(\Pi_h y_h)+(\Pi_h y_h,\nabla_{\! h} v_h)_\Omega-(f_h,\Pi_h v_h)_\Omega}}
  \\&=
  \inf_{v_h\in \mathcal{S}^{1,\textit{cr}}_D(\mathcal{T}_h)}{ \sup_{y_h\in \mathcal{R}T^0_N(\mathcal{T}_h)}{-\rho_{\varphi_h^*,\Omega}(\Pi_h y_h)-(\textup{div}\,y_h+f_h,\Pi_h v_h)_\Omega}}
  \\&\ge 
  \inf_{\overline{v}_h\in \mathcal{L}^0(\mathcal{T}_h)}{ \sup_{y_h\in \mathcal{R}T^0_N(\mathcal{T}_h)}{-\rho_{\varphi_h^*,\Omega}(\Pi_h y_h)-(\textup{div}\,y_h+f_h,\overline{v}_h)_\Omega}}
 \\&\ge 
  \sup_{y_h\in \mathcal{R}T^0_N(\mathcal{T}_h)}{\inf_{\overline{v}_h\in \mathcal{L}^0(\mathcal{T}_h)}{ -\rho_{\varphi_h^*,\Omega}(\Pi_h y_h)-(\textup{div}\,y_h+f_h,\overline{v}_h)_\Omega}}
  \\&=
  \sup_{y_h\in \mathcal{R}T^0_N(\mathcal{T}_h)}{- \rho_{\varphi_h^*,\Omega}(\Pi_h y_h)-\sup_{\overline{v}_h\in \mathcal{L}^0(\mathcal{T}_h)}{(\textup{div}\,y_h+f_h,\overline{v}_h)_\Omega}}
 \\&=
  \sup_{y_h\in \mathcal{R}T^0_N(\mathcal{T}_h)}{-\rho_{\varphi_h^*,\Omega}(\Pi_h y_h)-I_{\{-f_h\}}(\textup{div}\,y_h)}
 \\&=
  \sup_{y_h\in \mathcal{R}T^0_N(\mathcal{T}_h)}{D_h^{\textit{rt}}(y_h)}\,.
 \end{align*}

 \textit{ad (ii).} By definition, the discrete flux $z_h^{\textit{rt}}\in \mathcal{L}^1(\mathcal{T}_h)^d$, defined by \eqref{eq:gen_marini}, satisfies the discrete convex optimality condition \eqref{eq:pDirichletOptimalityCR1.2} and $\textup{div}\,(z_h^{\textit{rt}}|_T)=-f_h|_T$ in $T$ for all $T\in \mathcal{T}_h$. Due to $\vert \Gamma_D\vert >0$, the divergence operator $\textup{div}\colon 
 \mathcal{R}T^0_N(\mathcal{T}_h)\to \mathcal{L}^0(\mathcal{T}_h)$ is surjective. Hence, there exists
  $y_h\in \mathcal{R}T^0_N(\mathcal{T}_h)$ such that $\textup{div}\, y_h=-f_h$ in $\mathcal{L}^0(\mathcal{T}_h)$. Then, we have that $\textup{div}\,((z_h^{\textit{rt}}-y_h)|_T)=0$ in $T$ for all $T\in \mathcal{T}_h$, \textit{i.e.}, $z_h^{\textit{rt}}-y_h\in \mathcal{L}^0(\mathcal{T}_h)^d$. In addition, for every $v_h\in \mathcal{S}^{1,\textit{cr}}_D(\mathcal{T}_h)$, it holds that
 \begin{align*}
  \begin{aligned}
  (\Pi_h y_h,\nabla_{\! h} v_h)_\Omega&=-(\textup{div}\, y_h,\Pi_h v_h)_\Omega\\&=(f_h,\Pi_h v_h)_\Omega\\&=(\AAA_h(\cdot,\nabla_{\! h} u_h^{\textit{cr}}),\nabla_{\! h} v_h)_\Omega
 \\& =(\Pi_h z_h^{\textit{rt}},\nabla_{\! h} v_h)_\Omega\,.
 \end{aligned}
 \end{align*}
 In other words, for every $v_h\in \mathcal{S}^{1,\textit{cr}}_D(\mathcal{T}_h)$, it holds that
 \begin{align*}
  (y_h-z_h^{\textit{rt}},\nabla_{\! h} v_h)_\Omega=(\Pi_h y_h-\Pi_h z_h^{\textit{rt}},\nabla_{\! h} v_h)_\Omega=0\,,
 \end{align*}
 \textit{i.e.}, $y_h-z_h^{\textit{rt}}\in \nabla_{\! h}(\mathcal{S}^{1,\textit{cr}}_D(\mathcal{T}_h))^{\perp}$. By the decomposition \eqref{eq:decomposition}, we have that $\nabla_{\! h}(\mathcal{S}^{1,\textit{cr}}_D(\mathcal{T}_h))^{\perp}=\textup{ker}(\textup{div}|_{\mathcal{R}T^0_N(\mathcal{T}_h)})\subseteq \mathcal{R}T^0_N(\mathcal{T}_h)$. 
 As a result, we have that $y_h-z_h^{\textit{rt}}\in \mathcal{R}T^0_N(\mathcal{T}_h)$.~Due~to~${y_h\in \mathcal{R}T^0_N(\mathcal{T}_h)}$, we conclude that $z_h^{\textit{rt}}\in \mathcal{R}T^0_N(\mathcal{T}_h)$. In particular, now from
 $\textup{div}\,(z_h^{\textit{rt}}|_T)=-f_h|_T$ in $T$ for all $T\in \mathcal{T}_h$, it follows the discrete optimality condition
  \eqref{eq:pDirichletOptimalityCR1.1}.

 \textit{ad (iii).} Using \eqref{eq:pDirichletOptimalityCR2}, \eqref{eq:pDirichletOptimality1.1}, and the discrete integration-by-parts formula \eqref{eq:pi0},~we~find~that\enlargethispage{7mm}
 \begin{align*}
  I_h^{\textit{cr}}(u_h^{\textit{cr}})&=\rho_{\varphi_h,\Omega}(\nabla_{\! h} u_h^{\textit{cr}})-(f_h,\Pi_h u_h^{\textit{cr}})_\Omega
 \\& =-\rho_{\varphi_h^*,\Omega}(\Pi_h z_h^{\textit{rt}})+(\Pi_h z_h^{\textit{rt}},\nabla_{\! h} u_h^{\textit{cr}})_\Omega+(\textup{div}\,z_h^{\textit{rt}},\Pi_h u_h^{\textit{cr}})_\Omega
  \\&=-\rho_{\varphi_h^*,\Omega}(\Pi_h z_h^{\textit{rt}})-I_{\{-f_h\}}(\textup{div}\,z_h^{\textit{rt}})
  \\&=D_h^{\textit{rt}}(z_h^{\textit{rt}})\,,
 \end{align*}
 \textit{i.e.}, the discrete strong duality relation applies, which in conjunction with the discrete~weak~duality relation \eqref{eq:discrete_weak_duality} implies the maximality of $z_h^{\textit{rt}}\in \mathcal{R}T^0_N(\mathcal{T}_h)$ for \eqref{eq:pDirichletDualCR}. Since \eqref{eq:pDirichletDualCR} is strictly convex, $z_h^{\textit{rt}}\in \mathcal{R}T^0_N(\mathcal{T}_h)$ is unique.
 \end{proof}
	
 \newpage

 \subsection{Natural regularity assumption in the case $p\in C^{0,1}(\overline{\Omega})$}\enlargethispage{5mm}
	
	\qquad In this section, for $p\in C^{0,1}(\overline{\Omega})$, we briefly examine 
 the natural regularity assumption 
    \begin{align}\label{natural_regularity}
        F(\cdot,\nabla u)\in W^{1,2}(\Omega;\mathbb{R}^d)\,,
    \end{align}   
    on the solution $u\in W^{1,p(\cdot)}_D(\Omega)$ of \eqref{eq:pDirichletW1p}, which is satisfied under mild assumptions on the domain $\Omega\subseteq\mathbb{R}^d$, $d\in \mathbb{N}$, the exponent $p\hspace*{-0.1em}\in\hspace*{-0.1em} C^{0,1}(\overline{\Omega})$ with $p^-\hspace*{-0.1em}>\hspace*{-0.1em}1$ and the right-hand~side~${f\hspace*{-0.1em}\in\hspace*{-0.1em} \smash{L^{p'(\cdot)}(\Omega)}}$~(\textit{cf.}~\cite{ELM04}; see also
    \cite[Rem.\  4.5]{BDS15}).\enlargethispage{2mm}
	
	\begin{lemma}\label{lem:reg_primal_source}
            Let $p\in C^{0,1}(\overline{\Omega})$ with $p^->1$ and $\delta>0$. Then, for every  $v\in W^{1,p(\cdot)}(\Omega)$ with $F(\cdot,\nabla v)\in W^{1,2}(\Omega;\mathbb{R}^d)$, for $\mu(v)\coloneqq\vert \ln(\delta +\vert \nabla v\vert)\vert^2(\delta +\vert \nabla v\vert)^{p(\cdot)-2}\vert\nabla p\otimes \nabla v\vert^2\in L^1(\Omega)$,~it~holds~that
                \begin{align*}
                   (\delta+\vert \nabla v\vert )^{p(\cdot)-2}\vert \nabla^2 v\vert^2+\mu(v)\sim  \vert \nabla F(\cdot,\nabla v)\vert^2+\mu(v)\quad\text{ a.e.\ in }\Omega\,.
                \end{align*}
		\end{lemma}
		
		\begin{proof}
            Since the claimed equivalence reads $0\sim 0$ on $\{\vert\nabla v\vert =0\}$, we restrict to the case $\vert\nabla v\vert>0$. Here, by Rademacher's theorem, the product rule, and the chain rule, we find that
            \begin{align*}
                \nabla F(\cdot,\nabla v)&= (\delta+\vert \nabla v\vert)^{\frac{p-2}{2}}\big(\tfrac{\nabla p\otimes \nabla v}{2}\ln(\delta +\vert \nabla v\vert)+\tfrac{p-2}{2}\tfrac{\nabla\vert\nabla v\vert\otimes \nabla v}{\delta+\vert \nabla v\vert }+\nabla^2 v
                \big)\quad\text{ a.e.\ in }\{\vert\nabla v\vert>0\}\,.
            \end{align*}
            Since $\nabla p\in L^\infty(\Omega;\mathbb{R}^d)$ as well as $ (\delta+\vert \nabla v\vert)^{p-2}\vert\nabla^2 v\vert^2\sim (\delta+\vert \nabla v\vert)^{p-4}\vert\nabla\vert\nabla v\vert\otimes \nabla v\vert^2$ a.e.\ in $\Omega$ (\textit{cf.}\ \cite[Lem.\  2.8]{K22CR}),
            we conclude the claimed equivalence on $\{\vert\nabla v\vert>0\}$.
		\end{proof}

  The following lemma  
  translates the natural regularity assumption \eqref{natural_regularity} to the flux  $z\coloneqq \AAA(\cdot,\nabla u)\in W^{p'(\cdot)}_N(\textup{div};\Omega)$ and vice versa.  
		
		\begin{lemma}\label{lem:reg_equiv} Let $p\in C^{0,1}(\overline{\Omega})$ with $p^->1$ and $\delta\ge 0$. Then, for every $v\in W^{1,p(\cdot)}(\Omega)$ and $y\coloneqq \AAA(\cdot,\nabla v)\in L^{p'(\cdot)}(\Omega;\mathbb{R}^d)$, it holds  that $F(\cdot,\nabla v)\in W^{1,2}(\Omega;\mathbb{R}^d)$ if and only if $F^*(\cdot,y)\in W^{1,2}(\Omega;\mathbb{R}^d)$. In particular, it holds that $\vert F(\cdot,\nabla v)\vert^2 \sim \vert  F^*(\cdot,y)\vert^2$ a.e.\ in $\Omega$ and
        $\vert \nabla F(\cdot,\nabla v)\vert^2+(1+\vert \nabla v\vert^{p(\cdot)s}) \sim \vert \nabla F^*(\cdot,y)\vert^2+(1+\vert y\vert^{p'(\cdot)s})$ a.e.\ in $\Omega$ for some $s>1$ with can chosen to be~close~to~$1$.
		\end{lemma}
		
		\begin{proof}
                The first equivalence is evident.  For the second equivalence, we denote by
          $\tau_h f \coloneqq \vert h\vert^{-1}(f (\cdot + h) - f)$ the difference quotient and exploit that, by Lemma \ref{lem:Ax-Axh}, for all $h\in \mathbb{R}^d$~small~enough
          \begin{align*}
             \vert \tau_h[F(\cdot,\nabla v)]\vert^2
             &\lesssim \vert F(\cdot\hspace{-0.15em}+\hspace{-0.15em}h,(\nabla v)(\cdot\hspace{-0.15em}+\hspace{-0.15em}h))\hspace{-0.15em}-\hspace{-0.15em}F(\cdot\hspace{-0.15em}+\hspace{-0.15em}h,\nabla v)\vert^2
             +\vert F(\cdot\hspace{-0.15em}+\hspace{-0.15em}h,\nabla v)\hspace{-0.15em}-\hspace{-0.15em}F(\cdot,\nabla v)\vert^2
              \\&\lesssim \vert F^*(\cdot\hspace{-0.15em}+\hspace{-0.15em}h,\AAA(\cdot\hspace{-0.15em}+\hspace{-0.15em}h,(\nabla v)(\cdot+h))\hspace{-0.15em}-\hspace{-0.15em}F^*(\cdot\hspace{-0.15em}+\hspace{-0.15em}h,\AAA(\cdot\hspace{-0.15em}+\hspace{-0.15em}h,\nabla v))\vert^2
           \hspace{-0.15em}+\hspace{-0.15em}\vert \tau_h p\vert (1\hspace{-0.15em}+\hspace{-0.15em}\vert \nabla v\vert^{p(\cdot)s})
             \\&\lesssim \vert \tau_h[F^*(\cdot,\AAA(\cdot,\nabla v))]\vert^2\hspace{-0.15em}+\hspace{-0.15em} \vert \tau_h p\vert (1\hspace{-0.05em}+\hspace{-0.05em}\vert\nabla v\vert^{p(\cdot)s})
            \\&\quad+\vert F^*(\cdot\hspace{-0.15em}+\hspace{-0.15em}h,\AAA(\cdot,\nabla v))\hspace{-0.15em}-\hspace{-0.15em}F^*(\cdot\hspace{-0.15em}+\hspace{-0.15em}h,\AAA(\cdot\hspace{-0.15em}+\hspace{-0.15em}h,\nabla v))\vert^2
             \\&\quad+\vert F^*(\cdot\hspace{-0.15em}+\hspace{-0.15em}h,\AAA(\cdot,\nabla v))\hspace{-0.15em}-\hspace{-0.15em}F^*(\cdot,\AAA(\cdot,\nabla v))\vert^2
            \\&\lesssim \vert \tau_h[F^*(\cdot,y)]\vert^2\hspace{-0.15em}+\hspace{-0.15em}\vert \tau_h p\vert (1+\vert y\vert^{p'(\cdot)s})\quad\textup{ a.e.\ in }\{x\in \Omega\mid \textup{dist}(x,\partial\Omega)>\vert h\vert\}\,.
          \end{align*} 
        Similarly, \hspace{-0.1mm}we \hspace{-0.1mm}find \hspace{-0.1mm}that \hspace{-0.1mm}$\vert \tau_h[F^*(\cdot,y)]\vert\hspace{-0.15em}\lesssim\hspace{-0.15em}\vert \tau_h[F(\cdot,\nabla v)]\vert+(1+\vert \nabla v\vert^{p(\cdot)s})$ \hspace{-0.1mm}a.e.\ \hspace{-0.1mm}in \hspace{-0.1mm}${\{x\hspace{-0.15em}\in\hspace{-0.15em} \Omega\mid \textup{dist}(x,\partial\Omega)\hspace{-0.15em}>\hspace{-0.15em}\vert h\vert\}}$. Passing to the limit $\vert h\vert \to 0$ proves the claim.
		\end{proof}

        Lemma \ref{lem:reg_equiv}, in turn, motivates to prove the following dual counterparts of Lemma \ref{lem:reg_primal_source}.
		
		\begin{lemma}\label{lem:reg_dual_source}
        Let $p\in C^{0,1}(\overline{\Omega})$ with $p^->1$ and $\delta> 0$. Then, for every  $y\in L^{p'(\cdot)}(\Omega;\mathbb{R}^d)$~with $F^*(\cdot,y)\hspace{-0.15em}\in \hspace{-0.15em} W^{1,2}(\Omega;\mathbb{R}^d)$, for  $\mu^*(y)\hspace{-0.15em}\coloneqq\hspace{-0.15em}\vert \ln(\delta^{p(\cdot)-1} +\vert y\vert)\vert^2(\delta^{p(\cdot)-1} +\vert y\vert)^{p'(\cdot)-2}\vert \nabla p'\otimes y\vert^2\hspace{-0.15em}\in\hspace{-0.15em} L^1(\Omega)$,~it~holds that
                \begin{align*}
                 (\delta^{p(\cdot)-1}+\vert y\vert )^{p'(\cdot)-2}\vert \nabla y\vert^2
                +\mu^*(y)\sim  \vert \nabla F^*(\cdot, y)\vert^2+\mu^*(y)\quad\text{ a.e.\ in }\Omega\,.
                \end{align*}
		\end{lemma}
		
		\begin{proof}
        Since the claimed equivalence reads $0\sim 0$ on $\{\vert y\vert =0\}$, we restrict to the case $\vert y\vert>0$. Here, by Rademacher's theorem, the product rule, and the chain rule, we find that
            \begin{align*}
                \nabla F^*(\cdot,y)&= (\delta^{p-1}\!\!+\!\vert y\vert)^{\frac{p'-2}{2}}\big(\tfrac{\nabla p'\otimes y}{2}\ln(\delta^{p-1}\! \!+\!\vert y\vert)\!+\!\tfrac{p'-2}{2}\tfrac{(\nabla p \ln(\delta)\delta^{p-1}+\nabla\vert y\vert)\otimes y}{\delta^{p-1}+\vert y\vert }\!+\!
                \nabla y\big)\text{ a.e.\ in }\{\vert y\vert \!>\!0\}\,.
            \end{align*}
        Since $\nabla p'\in L^\infty(\Omega;\mathbb{R}^d)$ as well as $ (\delta^{p-1}+\vert y\vert)^{p'-2}\vert\nabla y\vert^2\sim (\delta^{p-1}+\vert y\vert)^{p'-4}\vert  \nabla\vert y\vert\otimes y\vert^2$ a.e.\ in $\Omega$ (\textit{cf.}\ \cite[Lem.\  2.11]{K22CR}), 
            we conclude the claimed equivalence on $\{\vert y\vert >0\}$.
		\end{proof}\newpage
	\section{\textit{Medius} error analysis}\label{sec:medius}

 	\qquad In this section, we prove a best-approximation result for the $\mathcal{S}^{1,cr}_D(\mathcal{T}_h)$-approximation \eqref{eq:pDirichletS1crD} of \eqref{eq:pDirichletW1p}.\enlargethispage{5mm}
  
 \begin{theorem}\label{thm:best-approx} 
 Let $p\in C^0(\overline{\Omega})$ with $p^->1$ and $\delta\ge 0$ and let $f\in L^{p'(\cdot)}(\Omega)\cap \bigcap_{h\in (0,h_0]}{L^{p_h'(\cdot)}(\Omega)}$ for some $h_0\hspace{-0.15em}>\hspace{-0.15em}0$. \hspace{-0.2em}Then, there exists some $s\hspace{-0.15em}>\hspace{-0.15em}1$, which can chosen to be close~to~$1$~if~${h\hspace{-0.15em}>\hspace{-0.15em}0}$~is~close~to~$0$, such that if $u\in W^{1,p(\cdot)s}_D(\Omega)$, then for every $h\in (0,h_0]$, it holds that
		\begin{align*}
    \begin{aligned}
		\|F_h(\cdot,\nabla_{\!h} u_h^{\textit{cr}})-F_h(\cdot,\nabla u)\|_{2,\Omega}^2&\lesssim \inf_{v_h\in\mathcal{S}^1_D(\mathcal{T}_h)}{\big[\|F_h(\cdot,\nabla v_h)-F_h(\cdot,\nabla u)\|_{2,\Omega}^2+\mathrm{osc}_h^2(f,v_h)}\\&\;\;+\|\omega_p(h_{\mathcal{T}})^2\,(1+\vert \nabla u\vert^{p(\cdot)s}+(\vert \nabla v_h\vert+\vert \nabla v_h-\nabla_{\!h} u_h^{\textit{cr}}\vert)^{p_h(\cdot)s})\|_{1,\Omega}\big]\,,
 \end{aligned}  
		\end{align*}
		where the hidden constant in $\lesssim$ also depends on $s\hspace{-0.17em}>\hspace{-0.17em}1$ and the chunkiness $\omega_0\hspace{-0.17em}>\hspace{-0.17em}0$, and~for~${v_h\hspace{-0.17em}\in\hspace{-0.17em} \smash{\mathcal{S}^1_D(\mathcal{T}_h)}}$ and $\mathcal{M}_h\subseteq \mathcal{T}_h $, we define $\mathrm{osc}_h^2(f,v_h,\mathcal{M}_h)\coloneqq \sum_{T\in \mathcal{M}_h}{\mathrm{osc}_h^2(f,v_h,T)}$,  
		where we define $\mathrm{osc}_h^2(f,v_h,T)$ $\coloneqq \rho_{((\varphi_h)_{\vert \nabla v_h\vert})^*,T}(h_T (f-f_h))$
 for all $T\in \mathcal{T}_h$ 
 and $\mathrm{osc}_h^2(f,v_h)\coloneqq\mathrm{osc}_h^2(f,v_h,\mathcal{T}_h)$.
	\end{theorem}

     The proof of Theorem \ref{thm:best-approx}
     involves three tools.\vspace{-1mm} 
		 
		 \subsection{Node-averaging quasi-interpolation operator}\label{subsec:node-averaging}

        \hspace{5mm}The \hspace{-0.1mm}first \hspace{-0.1mm}tool \hspace{-0.1mm}is \hspace{-0.1mm}the \hspace{-0.1mm}\textit{node-averaging \hspace{-0.1mm}quasi-interpolation \hspace{-0.1mm}operator} $\Pi_h^{\textit{av}}\colon \mathcal{L}^1(\mathcal{T}_h)\to \mathcal{S}^1_D(\mathcal{T}_h)$,~that, 
		  denoting~for~every~${\nu\in \mathcal{N}_h}$, 
        by ${\mathcal{T}_h(\nu)\coloneqq \{T\in \mathcal{T}_h\mid \nu \in T\}}$, the \textit{set of elements sharing~$\nu$}, for every ${v_h\in \mathcal{L}^1(\mathcal{T}_h)}$,~is~defined~by
	\begin{align*}
		\Pi_h^{\textit{av}}v_h\coloneqq \sum_{\nu\in \smash{\mathcal{N}_h}}{\langle v_h\rangle_\nu \varphi_\nu}\,,\qquad \langle v_h\rangle_\nu\coloneqq \begin{cases}
			\frac{1}{\textup{card}(\mathcal{T}_h(\nu ))}\sum_{T\in \mathcal{T}_h(\nu)}{(v_h|_T)(\nu)}&\;\text{ if }\nu\in \Omega\cup \Gamma_N\,,\\
			0&\;\text{ if }\nu\in \Gamma_D\,,
		\end{cases}
	\end{align*}
	where we denote by $(\varphi_\nu )_{\smash{\nu\in \mathcal{N}_h}}$, the nodal basis of $\mathcal{S}^1(\mathcal{T}_h)$.
	If $p\in C^0(\overline{\Omega})$ with $p^->1$ and $\delta\ge 0$, then there exists some $s>1$, which can chosen to be close to $1$ if $h_T>0$ is close to $1$,  such that 
 for every $a\ge 0$, $v_h\in \smash{\smash{\mathcal{S}^{1,\textit{cr}}_D(\mathcal{T}_h)}}$, $T\in \mathcal{T}_h$, and $m\in \{0,1\}$, it holds that (\textit{cf.}\ \cite[Cor.\ A.2]{BK22B})
	\begin{align}\label{eq:a6} 
\begin{aligned}&\int_T{\varphi_a(\xi_T,h_T^m\vert\nabla^m_{\!h}(v_h-\Pi_h^{\textit{av}}v_h)\vert)\,\mathrm{d}x}
	\lesssim \int_{\omega_T}{\varphi_a(\xi_T,h_T\vert\nabla_{\!h}v_h\vert)\,\mathrm{d}x}\\
  &\quad\lesssim \int_{\omega_T}{(\varphi_h)_a(\cdot,h_T\vert\nabla_{\!h}v_h\vert)\,\mathrm{d}x}+h_T^{\min\{2,p^-\}}\,\omega_{p,\omega_T}(h_T)\,\|1+a^{p_h(\cdot)s}+\vert\nabla_{\!h}v_h\vert^{p_h(\cdot)s}\|_{1,\omega_T} \,,
  \end{aligned}
	\end{align}
    where we use that $\varphi_a(\xi_T,h_T\,t)\lesssim \varphi_a(\xi_{T'},h_T\,t) +\smash{h_T^{\min\{2,p^-\}}}\,\vert p(\xi_T)-p(\xi_{T'})\vert\,(1+a^{p(\xi_{T'})s}+t^{p(\xi_{T'})s})$ for all $T'\in \mathcal{T}_h$ with $T'\subseteq \omega_T$ for 
    the second inequality, which follows analogously~to~\eqref{eq:phixh-phix}.
	
	\subsection{Local efficiency estimates}

    \qquad The second tool are local efficiency estimates that are  based on standard bubble function techniques.\enlargethispage{2mm} 
		
		\begin{lemma}\label{lem:efficiency}
		Let $p\in C^0(\overline{\Omega})$ with $p^->1$ and $\delta\ge 0$ and let $f\in L^{p'(\cdot)}(\Omega)\cap \bigcap_{h\in (0,h_0]}{L^{p_h'(\cdot)}(\Omega)}$ for some $h_0>0$. Then, there exists some $s>1$, which can chosen to be close to $1$ if $h_T>0$ is close to $0$, such that if $u\in W^{1,p(\cdot)s}_D(\Omega)$, then for~every~${h\in (0,h_0]}$,~${v_h\in \mathcal{S}^1_D(\mathcal{T}_h)}$, $T\in \mathcal{T}_h$,~and~$S\in \mathcal{S}_h^{i}$, it holds that
		 \hspace{-3mm}\begin{align}
		 \rho_{((\varphi_h)_{\vert \nabla v_h\vert})^*,T}(h_Tf_h)&\lesssim \|F_h(\cdot,\nabla v_h)-F_h(\cdot,\nabla u)\|_{2,T}^2,\label{lem:efficiency.1}\\&\quad+\|\omega_{p,T}(h_T)^2\,(1+\vert \nabla u\vert^{p(\cdot)s})\|_{1,T}+\mathrm{osc}_h^2(f,v_h,T)\, \notag\\
		 h_S\|\jump{F_h(\cdot,\nabla v_h)}_S\|_{2,S}^2&\lesssim \|F_h(\cdot,\nabla v_h)-F_h(\cdot,\nabla u)\|_{2,\omega_S}^2\label{lem:efficiency.2}\\&\quad+\|\omega_{p,\omega_S}(h_S)^2\,(1+\vert \nabla u\vert^{p(\cdot)s}+\vert \nabla v_h\vert^{p_h(\cdot)s})\|_{1,\omega_S}+\mathrm{osc}_h^2(f,v_h,\omega_S) \notag\,,\hspace{-1mm}
		 \end{align}
        where the hidden constants in \eqref{lem:efficiency.1} and \eqref{lem:efficiency.2} also depend on $s>1$ and the chunkiness $\omega_0>0$, and for every $\mathcal{M}_h\subseteq\mathcal{T}_h$, we define $\omega_{p,\cup\mathcal{M}_h}(t)|_T\coloneqq \omega_{p,T}(t)$~for~all~$t\ge 0$ and $T\in \mathcal{T}_h$~and~${\omega_p\coloneqq \omega_{p,\mathcal{T}_h}}$.
		\end{lemma}
		
		\begin{proof}
		 We generalize the procedure in the proof of \cite[Lem.\  3.2]{K22CR}.
		 
		 \textit{ad \eqref{lem:efficiency.1}.} Let $T \hspace{-0.1em}\in \hspace{-0.1em} \mathcal{T}_h$ be fixed, but arbitrary. Then, there exists a bubble~function~${b_T \hspace{-0.1em}\in  \hspace{-0.1em}W^{1,\infty}_0(T)}$ such that $0 \leq b_T   \lesssim   1 $ in $T$, $\vert \nabla b_T\vert  \lesssim \smash{h^{-1}_T}$ in $T$, and $\smash{\langle b_T\rangle_T} = 1$, where~the~hidden~\mbox{constant}~in~$\lesssim$ depends only on the chunkiness $\omega_0>0$.
		 Using \eqref{eq:pDirichletW1p} and integration-by-parts, taking~into~account that $\AAA_h(\cdot,\nabla v_h)\in \mathcal{L}^0(\mathcal{T}_h)^d$ and ${b_T\in W^{1,\infty}_0(T)}$, for every $\lambda\in \mathbb{R}$, we find that
		 \begin{align}
		 (\AAA(\cdot,\nabla u)-\AAA_h(\cdot,\nabla v_h), \nabla(\lambda b_T))_T=(f,\lambda b_T)_T\,.\label{lem:efficiency.5}
		 \end{align}
		 For the special choice $\lambda_T\coloneqq\textup{sgn}(f_h)\partial_a((\varphi_{\vert\nabla v_h\vert})^*)(\xi_T,h_T\vert f_h\vert )\in\mathbb{R}$, 
		 by the Fenchel--Young~identity (\textit{cf.}\ \cite[Prop.\ 5.1,~p.~21]{ET99}), we obtain
		 \begin{align}
		  (h_T f_h)\lambda_T=(\varphi_{\vert\nabla v_h\vert})^*(\xi_T,h_T\vert f_h\vert)+\varphi_{\vert\nabla v_h\vert}(\xi_T,\vert\lambda_T\vert)\,.\label{lem:efficiency.6}
		 \end{align}
		 Then, choosing $\lambda=h_T\lambda_T\in \mathbb{R}$ (\textit{cf.}\ \eqref{lem:efficiency.6}) in \eqref{lem:efficiency.5}, we observe that
		 \begin{align}\label{lem:efficiency.7}
		 \begin{aligned}
		  \rho_{((\varphi_h)_{\vert\nabla v_h\vert})^*,T}(h_Tf_h)+\rho_{(\varphi_h)_{\vert\nabla v_h\vert},T}(\lambda_T)&=(f,h_T\lambda_T b_T)_T\\&\quad+( f_h-f,h_T\lambda_T b_T)_T
		  \\&= (\AAA_h(\cdot,\nabla u)-\AAA_h(\cdot,\nabla v_h), \nabla(h_T\lambda_T b_T))_T\\&\quad+
 (\AAA(\cdot,\nabla u)-\AAA_h(\cdot,\nabla u), \nabla(h_T\lambda_T b_T))_T\\&\quad +(f_h-f,h_T\lambda_T b_T)_T
  \\&\eqqcolon I_h^1+I_h^2+I_h^3\,.
		  \end{aligned}
		 \end{align}
		 Applying the $\varepsilon$-Young inequality \eqref{ineq:young} with $\psi=\varphi_{\vert\nabla v_h\vert}(\xi_T,\cdot)$ together~with~\eqref{eq:hammera}, also using that $\vert b_T\vert+h_T\vert \nabla b_T\vert\lesssim 1$ in $T$, we find that
		 \begin{align}\label{lem:efficiency.8}
		 \begin{aligned}
		  I_h^1&\leq c_\varepsilon\,\|F_h(\cdot,\nabla v_h)-F_h(\cdot,\nabla u)\|_{2,T}^2+\varepsilon\,\rho_{(\varphi_h)_{\vert\nabla v_h\vert},T}(\lambda_T)\,,\\
		  I_h^3&\leq c_\varepsilon\,\mathrm{osc}_h^2(f,v_h,T)+\varepsilon\,\rho_{(\varphi_h)_{\vert\nabla v_h\vert},T}(\lambda_T)\,.
		 \end{aligned}
		 \end{align}
  Applying the $\varepsilon$-Young inequality \eqref{ineq:young} with $\psi\hspace*{-0.15em}=\hspace*{-0.15em}\varphi_{\vert\nabla u\vert}(\xi_T,\cdot)$ together~with~\eqref{eq:hammerh},~${\vert b_T\vert\hspace*{-0.15em}+\hspace*{-0.15em}h_T\vert \nabla b_T\vert\hspace*{-0.15em}\lesssim\hspace*{-0.15em} 1}$ in $T$, the shift change \eqref{lem:shift-change.1}, and Lemma~\ref{lem:A-Ah}\eqref{eq:Ah-A}, we obtain
		 \begin{align}\label{lem:efficiency.8.1}
   \begin{aligned}
 \hspace{-2mm}I_h^2&
 \leq c_\varepsilon\,\|F_h^*(\cdot,\AAA_h(\cdot,\nabla u))-F_h^*(\cdot,\AAA(\cdot,\nabla u))\|_{2,T}^2
 +\varepsilon\,\rho_{(\varphi_h)_{\vert\nabla u\vert},T}(\lambda_T)
 \\&\lesssim c_\varepsilon\,\|\omega_{p,T}(h_T)^2\, (1+\vert \nabla u\vert)^{p(\cdot)s}\|_{1,T}
 +\varepsilon\,\big[\rho_{(\varphi_h)_{\vert\nabla v_h\vert},T}(\lambda_T)+\|F_h(\cdot,\nabla v_h)-F_h(\cdot,\nabla u)\|_{2,T}^2\big]\,.\hspace{-2mm}
 \end{aligned}
		 \end{align}
  Taking into account  \eqref{lem:efficiency.8} and \eqref{lem:efficiency.8.1} in \eqref{lem:efficiency.7}, for sufficiently small $\varepsilon>0$,~we~conclude~that
		 \begin{align}\label{lem:efficiency.9}
		 \begin{aligned}
		  \rho_{((\varphi_h)_{\vert\nabla v_h\vert})^*,T}(h_T f_h)&\lesssim \|F_h(\cdot,\nabla v_h)-F_h(\cdot,\nabla u)\|_{2,T}^2\\&\quad+\|\omega_{p,T}(h_T)^2 (1+\vert \nabla u\vert^{p(\cdot)s})\|_{1,T}+\mathrm{osc}_h^2(f,v_h,T)\,.
		  \end{aligned}
		 \end{align}
		 
		 \textit{ad \eqref{lem:efficiency.2}.} Let $S \hspace{-0.17em}\in \hspace{-0.17em} \mathcal{S}_h^{i}$ be fixed, but arbitrary. \!Then, there exists a bubble~function~${b_S \hspace{-0.17em}\in\hspace{-0.17em} W^{1,\infty}_0(\omega_S)}$ such that $0  \leq b_S \lesssim  1 $ in $\omega_S$, $\vert \nabla b_S\vert  \lesssim   \smash{h^{-1}_S}$ in $\omega_S$, and $\smash{\langle b_S\rangle_S} = 1$, where~the~hidden~\mbox{constant}~in~$\lesssim$ depends only on the chunkiness $\omega_0>0$.
		 Using \eqref{eq:pDirichletW1p} and integration-by-parts, taking into account that $\AAA_h(\cdot,\nabla v_h)\in \mathcal{L}^0(\mathcal{T}_h)^d$ and $b_S\in W^{1,\infty}_0(\omega_S)$ with $\smash{\langle b_S\rangle_S}=1$,~for~every~${\lambda\in \mathbb{R}}$,~we~find~that\enlargethispage{3mm}
		 \begin{align}\label{lem:efficiency.10}
		 (\AAA(\cdot,\nabla u)-\AAA_h(\cdot,\nabla v_h), \nabla(\lambda b_S))_{\omega_S}=(f,\lambda b_S)_{\omega_S}- \vert S\vert \jump{\AAA_h(\cdot,\nabla v_h)\cdot n}_S\lambda\,.
		 \end{align}
		 Let $T\in \mathcal{T}_h$ be such that $T\subseteq \omega_S$. Then, using the notation $\nabla v_h(T)\coloneqq\nabla v_h|_T$\footnote{In what follows, we employ this notation to indicate that the value of $\nabla v_h(T)\in \mathbb{R}^d$ depends only on the value of $\nabla v_h\in \mathcal{L}^0(\mathcal{T}_h)$ on $T\in \mathcal{T}_h$.}, for the choice
  \begin{align*}
  \lambda_{S,T}\coloneqq \textup{sgn}(\jump{\AAA_h(\cdot,\nabla v_h)\cdot n}_S)\partial_a((\varphi_{\vert\nabla v_h(T)\vert})^*)(\xi_T,\vert \jump{\AAA_h(\cdot,\nabla v_h)\cdot n}_S\vert)\in \mathbb{R}\,,
  \end{align*}
		 by the Fenchel--Young identity (\textit{cf.}\ \cite[Prop.\ 5.1,~p.~21]{ET99}),~it~holds~that
		 \begin{align}\label{lem:efficiency.11}
		 \jump{\AAA_h(\cdot,\nabla v_h)\cdot n}_S\lambda_{S,T}=(\varphi_{\vert\nabla v_h(T)\vert})^*(\xi_T,\vert\jump{\AAA_h(\cdot,\nabla v_h)\cdot n}_S\vert)+\varphi_{\vert\nabla v_h(T)\vert}(\xi_T,\vert\lambda_{S,T}\vert)\,.
		 \end{align}
		 Next, let $T'\in \mathcal{T}_h\setminus\{T\}$ be such that $T'\subseteq \omega_S$. Then, due to the convexity of $(\varphi_{\vert\nabla v_h(T)\vert})^*(\xi_T,\cdot)$~and $\Delta_2((\varphi_{\vert\nabla v_h(T)\vert})^*(\xi_T,\cdot)) \lesssim 2^{\smash{\max\{2,(p^-)'\}}}$, also using the shift change \eqref{lem:shift-change.3} and \eqref{eq:hammerh},~we~have~that
		 \begin{align}\label{lem:efficiency.12.0}
  \begin{aligned}
		 (\varphi_{\vert\nabla v_h(T)\vert})^*&(\xi_T,\vert\jump{\AAA_h(\cdot,\nabla v_h)\cdot n}_S\vert) 
  \\&\gtrsim (\varphi_{\vert\nabla v_h(T)\vert})^*(\xi_T,\vert \jump{\AAA(\xi_T,\nabla v_h)}_S\cdot n_{T}\vert)\\&\quad-
  (\varphi_{\vert\nabla v_h(T)\vert})^*(\xi_T,\vert (\AAA(\xi_T,\nabla v_h(T'))-\AAA(\xi_{T'},\nabla v_h(T')))\cdot n_{T}\vert)
  \\&\gtrsim (\varphi_{\vert\nabla v_h(T)\vert})^*(\xi_T,\vert \jump{\AAA(\xi_T,\nabla v_h)}_S\cdot n_{T}\vert)\\&\quad-
  c_\varepsilon\,(\varphi_{\vert\nabla v_h(T')\vert})^*(\xi_T,\vert \AAA(\xi_T,\nabla v_h(T'))-\AAA(\xi_{T'},\nabla v_h(T'))\vert)
  \\&\quad-
  \varepsilon\,\vert F(\xi_T,\nabla v_h(T'))-F(\xi_T,\nabla v_h(T))\vert^2
  \\&\gtrsim (\varphi_{\vert\nabla v_h(T)\vert})^*(\xi_T,\vert \jump{\AAA(\xi_T,\nabla v_h)}_S\cdot n_{T}\vert)
 \\&\quad-c_\varepsilon\,\vert F^*(\xi_T,\AAA(\xi_T,\nabla v_h(T')))-F^*(\xi_T,\AAA(\xi_{T'},\nabla v_h(T')))\vert^2 
 \\&\quad-
  \varepsilon\,\vert\jump{F(\xi_T,\nabla v_h)}_S\vert^2
  \eqqcolon I_h^1-c_\varepsilon\, I_h^2-\varepsilon\, I_h^3\,.
  \end{aligned}
		 \end{align}
 Resorting to Lemma \ref{lem:Ax-Axh}\eqref{eq:Axh-Ax}, we deduce that
 \begin{align}\label{lem:efficiency.12.1}
 \smash{ I_h^2\lesssim \omega_{p,\omega_S}(h_S)^2\, (1+\vert \nabla_{\! h} v_h(T')\vert^{p(\xi_{T'})s})\,.}
 \end{align}
 Using that $n_T=\pm\frac{\jump{\nabla v_h}_S}{\vert \jump{\nabla v_h}_S\vert}$ since $v_h\in \mathcal{S}^1_D(\mathcal{T}_h)$ (\textit{cf.}\ \cite[p.\ 12]{DK08}), and \eqref{eq:hammera}, we find that
 \begin{align*}
 \smash{\vert \jump{\AAA(\xi_T,\nabla v_h)}_S\cdot n_{T}\vert  
 \sim \tfrac{\varphi_{\vert \nabla v_h(T)\vert }(\xi_T,\vert \jump{\nabla v_h}_S\vert)}{\vert \jump{\nabla v_h}_S\vert}\sim\partial_a(\varphi_{\vert \nabla v_h(T)\vert })(\xi_T,\vert \jump{\nabla v_h}_S\vert) \,,}
 \end{align*}
 and, thus, using $(\varphi_{\vert\nabla v_h(T)\vert})^*(\xi_T,\cdot)\circ \partial_a(\varphi_{\vert\nabla v_h(T)\vert})(\xi_T,\cdot)\hspace*{-0.1em}\sim\hspace*{-0.1em} \vert F(\xi_T,\cdot)\vert^2$ (\textit{cf.}\ \cite[(2.6)]{DK08}),~we~get~${I_h^1\hspace*{-0.1em}\sim\hspace*{-0.1em} I_h^3}$.
 Therefore, for sufficiently small $\varepsilon>0$, resorting to Lemma \ref{lem:Ax-Axh}\eqref{eq:Fxh-Fx}, we deduce that
 \begin{align}\label{lem:efficiency.12.2}
\begin{aligned}
 I_h^1-\varepsilon\, I_h^3&\gtrsim\vert\jump{F_h(\cdot,\nabla v_h)}_S\vert^2- \vert F(\xi_{T'},\nabla v_h(T'))-F(\xi_T,\nabla v_h(T'))\vert^2  \\&\gtrsim\vert\jump{F_h(\cdot,\nabla v_h)}_S\vert^2-\omega_{p,\omega_S}(h_S)^2\, (1+\vert \nabla v_h(T')\vert^{p(\xi_{T'})s})
  \,.\end{aligned}
 \end{align}
 Combining \eqref{lem:efficiency.12.1} and \eqref{lem:efficiency.12.2} in \eqref{lem:efficiency.12.0}, we arrive at
 \begin{align}\label{lem:efficiency.12}
 \hspace{-2mm}\vert\jump{F_h(\cdot,\hspace{-0.1em}\nabla v_h)}_S\vert^2\hspace{-0.1em}\lesssim \hspace{-0.1em}(\varphi_{\vert\nabla v_h(T)\vert})^*(\xi_T,\hspace{-0.1em}\vert\jump{\AAA_h(\cdot,\hspace{-0.1em}\nabla v_h)\cdot n}_S\vert)\hspace{-0.1em}+\hspace{-0.1em}\omega_{p,\omega_S}(h_S)^2\, (1\hspace{-0.1em}+\hspace{-0.1em}\vert \nabla v_h(T')\vert^{p(\xi_{T'})s})\,.\hspace{-1mm}
 \end{align}
		 For $\smash{\lambda=\frac{\vert \omega_S\vert}{\vert S\vert}\lambda_{S,T}\in \mathbb{R}}$ (\textit{cf.}~\eqref{lem:efficiency.11}) in \eqref{lem:efficiency.10}, also using \eqref{lem:efficiency.12}, we observe that
		 \begin{align}\label{lem:efficiency.12.2.1}
		 \begin{aligned}
		  h_S\|\jump{F_h(\cdot,\nabla v_h)}_S\|_{2,S}^2
		  &+\rho_{(\varphi_h)_{\vert\nabla v_h(T)\vert},\omega_S}(\lambda_{S,T})-\| \omega_{p,\omega_S}(h_S)^2(1+\vert \nabla v_h\vert^{p_h(\cdot)s})\|_{1,\omega_S}
		  \\&\lesssim \vert \omega_S\vert \jump{\AAA_h(\cdot,\nabla v_h)\cdot n}_S\lambda_{S,T}
 \\&= \tfrac{\vert \omega_S\vert}{\vert S\vert}(\AAA_h(\cdot,\nabla v_h)-\AAA_h(\cdot,\nabla u), \nabla(\lambda_{S,T} b_S))_{\omega_S}\\&\quad+\tfrac{\vert \omega_S\vert}{\vert S\vert}(\AAA_h(\cdot,\nabla u)-\AAA(\cdot,\nabla u), \nabla(\lambda_{S,T} b_S))_{\omega_S}
		  \\&\quad+\tfrac{\vert \omega_S\vert}{\vert S\vert}(f_h,\lambda_{S,T} b_S)_{\omega_S}+\tfrac{\vert \omega_S\vert}{\vert S\vert}(f-f_h,\lambda_{S,T} b_S)_{\omega_S}
  \\&\eqqcolon I_h^1+I_h^2+I_h^3+I_h^4\,.
		 \end{aligned}
		 \end{align}
		 Applying the $\varepsilon$-Young inequality \eqref{ineq:young} with $\psi=\varphi_{\vert\nabla v_h(T')\vert }(\xi_{T'},\cdot)$ together with \eqref{eq:hammera} for all $T'\in \mathcal{T}_h$ with $T'\subseteq \omega_S$, and $\vert b_S\vert+h_S\vert \nabla b_S\vert\lesssim 1$ in $\omega_S$, we obtain\enlargethispage{3mm}
		 \begin{align}\label{lem:efficiency.13}
		 \begin{aligned}
		 I_h^1 &\leq \smash{c_\varepsilon\,\|F_h(\cdot,\nabla v_h)-F_h(\cdot,\nabla u)\|_{2,\omega_S}^2+\varepsilon\,\rho_{(\varphi_h)_{\vert\nabla v_h\vert},\omega_S}(\lambda_{S,T})}\,,\\[0.5mm]
		  I_h^3&\leq 
		 \smash{ c_\varepsilon\,\rho_{((\varphi_h)_{\vert\nabla v_h\vert})^*,\omega_S}(h_{\mathcal{T}}f)+\varepsilon\,\rho_{(\varphi_h)_{\vert\nabla v_h\vert},\omega_S}(\lambda_{S,T})}\,,\\[0.5mm]
     I_h^4&\leq 
		  \smash{c_\varepsilon\,\textup{osc}(f,v_h,\omega_S)+\varepsilon\,\rho_{(\varphi_h)_{\vert\nabla v_h\vert},\omega_S}(\lambda_{S,T})}\,.
		  \end{aligned}
		 \end{align}
  Applying the $\varepsilon$-Young inequality \eqref{ineq:young} with $\psi=\varphi_{\vert\nabla u\vert }(\xi_{T'},\cdot)$ together with \eqref{eq:hammerh} for all $T'\in \mathcal{T}_h$ with $T'\subseteq \omega_S$, $\vert b_S\vert+h_S\vert \nabla b_S\vert\lesssim 1$ in $\omega_S$, the shift change \eqref{lem:shift-change.1},  and Lemma~\ref{lem:A-Ah}\eqref{eq:Ah-A},~we~obtain
 \begin{align}\label{lem:efficiency.13.2}
  I_h^2& 
  \leq c_\varepsilon\,\|F_h^*(\cdot,\AAA_h(\cdot, \nabla u))-F_h^*(\cdot,\AAA(\cdot,\nabla u))\|_{2,\omega_S}^2
+\varepsilon\,\rho_{(\varphi_h)_{\vert\nabla u\vert},\omega_S}(\lambda_{S,T})
  \\&\lesssim c_\varepsilon\,\|\omega_p(h_{\mathcal{T}})^2\,(1+\vert \nabla u\vert^{p(\cdot)s})\|_{1,\omega_S}
  +\varepsilon\,\big[\rho_{(\varphi_h)_{\vert\nabla v_h\vert},\omega_S}(\lambda_{S,T})+\|F_h(\cdot,\nabla v_h)-F_h(\cdot,\nabla u)\|_{2,\omega_S}^2\big]\,.\notag
		 \end{align}
    The shift change \eqref{lem:shift-change.1} on $T'\in \mathcal{T}_h\setminus\{T\}$ with $T'\subseteq \omega_S$ further yields that
        \begin{align}\label{lem:efficiency.14}
            \begin{aligned}
                \rho_{(\varphi_h)_{\vert\nabla v_h\vert},\omega_S}(\lambda_T^S)&\lesssim \rho_{(\varphi_h)_{\vert\nabla v_h(T)\vert},\omega_S}(\lambda_T^S)+h_S\|\jump{F(\xi_{T'},\nabla v_h)}_S\|_{2,S}^2
                \\&\lesssim \rho_{(\varphi_h)_{\vert\nabla v_h(T)\vert},\omega_S}(\lambda_T^S)+ h_S\|\jump{F_h(\cdot,\nabla v_h)}_S\|_{2,S}^2\\&\quad+\|\omega_{p,\omega_S}(h_S)^2\,(1+\vert \nabla v_h\vert^{p_h(\cdot)s})\|_{1,\omega_S}\,. \end{aligned}
        \end{align}
        Combining \eqref{lem:efficiency.12.2.1}--\eqref{lem:efficiency.14}, 
		 for sufficiently small $\varepsilon>0$,  we conclude that \eqref{lem:efficiency.2} applies.\enlargethispage{5mm}
		\end{proof}

		\subsection{Patch-shift-to-element-shift inequality}

  \qquad The third tool is an estimate that enables to pass from element-patch-shifts to element-shifts. This is essential in the application of  quasi-interpolation operators that are only element-to-patch stable, \textit{e.g.}, the node-averaging quasi-interpolation operator $\Pi_h^{\textit{av}}\colon\mathcal{S}^{1,\textit{cr}}_D(\mathcal{T}_h)\to \mathcal{S}^1_D(\mathcal{T}_h)$~(\textit{cf.}~\eqref{eq:a6}).
		
		\begin{lemma}\label{lem:patch_to_element}
		 Let $p\in C^0(\overline{\Omega})$ with $p^->1$ and let $\delta\ge 0$.
		 Then, there exists some $s>1$, which can chosen to be close to $1$ if $h_T>0$ is close to $0$, such that for every $y_h\in L^{p_h(\cdot)}(\Omega;\mathbb{R}^d)$,~${v_h\in \mathcal{L}^1(\mathcal{T}_h)}$, and $T\in \mathcal{T}_h$, it holds that
		 \begin{align*}
  \rho_{(\varphi_h)_{\vert \nabla_{\! h} v_h(T)\vert},\omega_T}(y_h)&\lesssim \rho_{(\varphi_h)_{\vert \nabla_{\! h} v_h\vert},\omega_T}(y_h)+\|\smash{h_{\mathcal{S}}^{1/2}}\jump{F_h(\cdot,\nabla_{\! h} v_h)}\|_{2,\mathcal{S}_h^{i}(T)}^2\\&\quad+\|\omega_{p,\omega_T}(h_T)^2\,(1+\vert \nabla_{\!h}v_h\vert^{p_h(\cdot)s})\|_{1,T}\,.
		 \end{align*}
		 where $\mathcal{S}_h^{i}(T)\coloneqq \mathcal{S}_h(T)\cap \mathcal{S}_h^{i}$ and $\mathcal{S}_h(T)\coloneqq \{S\in \mathcal{S}_h\mid S\cap T\neq \emptyset\}$.
		\end{lemma}
		
		\begin{proof}
        Applying for every $T'\in \mathcal{T}_h$ with $T'\subseteq \omega_T$, the shift change \eqref{lem:shift-change.1}, we arrive at
		 \begin{align}\label{lem:patch_to_element.1}
		  \smash{\rho_{(\varphi_h)_{\vert \nabla_{\! h} v_h(T)\vert},\omega_T}(y_h)\lesssim \rho_{(\varphi_h)_{\vert \nabla_{\! h} v_h\vert},\omega_T}(y_h)+\|F_h(\cdot,\nabla_{\! h} v_h(T))-F_h(\cdot,\nabla_{\! h} v_h)\|_{2,\omega_T}^2\,.}
		 \end{align}
		 Since each $T'\in \mathcal{T}_h$ with $T'\subseteq \omega_T$ can be reached by passing through a uniformly bounded~number (depending on $\omega_0\hspace*{-0.1em} > \hspace*{-0.1em}0$) of sides ${S\hspace*{-0.1em}\in\hspace*{-0.1em} \mathcal{S}_h^{i}(T)}$, for every $T'\hspace*{-0.1em} \in \hspace*{-0.1em}\mathcal{T}_h$ with $T' \hspace*{-0.1em}\subseteq\hspace*{-0.1em} \omega_T$,~using~\eqref{eq:Fh-F.1},~it~holds~that
     \begin{align}\label{lem:patch_to_element.2}
  \begin{aligned}
  &\vert F(\xi_{T'},\nabla_{\! h} v_h(T))-F(\xi_{T'},\nabla_{\! h} v_h(T'))\vert^2\\&\quad\lesssim \vert F(\xi_T,\nabla_{\! h} v_h(T))-F(\xi_{T'},\nabla_{\! h} v_h(T'))\vert^2+\vert F(\xi_{T'},\nabla_{\! h} v_h(T))-F(\xi_T,\nabla_{\! h} v_h(T))\vert^2 
 \\&\quad\lesssim \smash{\vert T\vert^{-1}\| \smash{h_{\mathcal{S}}^{1/2}}\jump{F_h(\cdot,\nabla_{\! h} v_h)}_S\|^2_{2,\mathcal{S}_h^{i}(T)}}
  +\vert T\vert^{-1}\|\omega_{p,\omega_T}(h_T)^2\,(1+\vert \nabla_{\!h}v_h\vert^{p_h(\cdot)s})\|_{1,T}\,.
  \end{aligned}
  \end{align}
  Eventually, multiplying \eqref{lem:patch_to_element.2} by $\vert T'\vert$ for all $T'\in \mathcal{T}_h$ with $T'\subseteq \omega_T$,
  due to $\vert T'\vert \sim \vert T\vert$, 
  we~arrive~at the claimed estimate.
	\end{proof}
	 
 \begin{proof}[Proof (of Theorem \ref{thm:best-approx})]
				Abbreviating $e_h\coloneqq v_h-u_h^{\textit{cr}}\in \smash{\mathcal{S}^{1,\textit{cr}}_D(\mathcal{T}_h)}$ and resorting~to~\eqref{eq:pDirichletW1p}~and~\eqref{eq:pDirichletS1crD} as well as that $f-f_h\perp \Pi_he_h$ in $L^2(\Omega)$, we arrive at
				\begin{align}
				\begin{aligned}
					( \AAA_h(\cdot,\nabla v_h)-\AAA_h(\cdot,\nabla_{\!h} u_h^{\textit{cr}}),\nabla_{\!h} e_h )_\Omega&=
					(\AAA_h(\cdot,\nabla v_h),\nabla_{\!h}( e_h- \Pi_h^{\textit{av}} e_h) )_\Omega	\\&\quad+( f,\Pi_h^{\textit{av}} e_h-e_h)_\Omega
					\\&\quad+( \AAA_h(\cdot,\nabla v_h)-\AAA_h(\cdot,\nabla u) ,\nabla \Pi_h^{\textit{av}} e_h)_\Omega
					\\&\quad+(f-f_h,e_h-\Pi_he_h)_\Omega
			 \\&
    \eqqcolon I_h^1+I_h^2+I_h^3+I_h^4\,.
			 \end{aligned}\label{thm:best-approx.1}
				\end{align}
				
				\textit{ad $I_h^1$.} An element-wise integration-by-parts, 
             $\jump{\AAA_h(\cdot,\nabla v_h)\hspace{-0.1em}\cdot \hspace{-0.1em} n( e_h- \Pi_h^{\textit{av}} e_h)}_S\hspace{-0.1em}=\hspace{-0.1em}\jump{\AAA_h(\cdot,\nabla v_h)\hspace{-0.1em}\cdot\hspace{-0.1em} n }_S $ $\{e_h-\Pi_h^{\textit{av}} e_h\}_S +\{\AAA_h(\cdot,\nabla v_h)\cdot n\}_S\jump{e_h\hspace{-0.1em}-\hspace{-0.1em}\Pi_h^{\textit{av}} e_h}_S$ on $S$, $\int_S{\jump{e_h\hspace*{-0.1em}-\hspace*{-0.1em}\Pi_h^{\textit{av}} e_h}_S\,\textup{d}s}\hspace*{-0.1em}=\hspace*{-0.1em}0$ and ${\{\AAA_h(\cdot,\nabla v_h)\hspace*{-0.1em}\cdot\hspace*{-0.1em} n\}_S}$ $=\textup{const}$ on $S$ for all ${S\in \mathcal{S}_h^{i}}$, the discrete trace inequality \cite[Lem.\ 12.8]{EG21}, and \eqref{eq:a6} with $\psi=\vert\cdot\vert$ and $a=0$ yield\vspace{-0.5mm}
				\begin{align}
					\begin{aligned}
					I_h^1&=\sum_{S\in \mathcal{S}_h^{i}}{\jump{\AAA_h(\cdot,\nabla v_h)\cdot n}_S\int_S{\{e_h- \Pi_h^{\textit{av}} e_h\}_S\,\textup{d}s}}
      \\&
					\lesssim \sum_{S\in \mathcal{S}_h^{i}}{ \sum_{T\in \mathcal{T}_h;T\subseteq \omega_S}{\int_{\omega_T}{\vert \jump{\AAA_h(\cdot,\nabla v_h)\cdot n}_S\vert\vert \nabla_{\!h}e_h\vert\,\textup{d}x}}}\,.
				\end{aligned}	\label{thm:best-approx.2}
				\end{align}
 Next, let $T'\in \mathcal{T}_h\setminus\{T\}$ with $T'\subseteq \omega_S$. Then, resorting to the convexity of $(\varphi_{\vert\nabla v_h(T)\vert})^*(\xi_T,\cdot)$ and $\Delta_2((\varphi_{\vert\nabla v_h(T)\vert})^*(\xi_T,\cdot)) \lesssim 2^{\max\{2,(p^-)'\}}$, also using the shift change \eqref{lem:shift-change.3} and \eqref{eq:hammerh},~we~have~that
		 \begin{align}\label{thm:best-approx.3.0}
  \begin{aligned}
		 (\varphi_{\vert\nabla v_h(T)\vert})^*&(\xi_T,\vert\jump{\AAA_h(\cdot,\nabla v_h)\cdot n}_S\vert)
  \\&\lesssim (\varphi_{\vert\nabla v_h(T)\vert})^*(\xi_T,\vert \jump{\AAA(\xi_T,\nabla v_h)\cdot n_T}_S\vert)
  \\&\quad+(\varphi_{\vert\nabla v_h(T)\vert})^*(\xi_T,\vert (\AAA(\xi_T,\nabla v_h(T'))-\AAA(\xi_{T'},\nabla v_h(T')))\cdot n_{T}\vert)
  \\&\lesssim 
  (\varphi_{\vert\nabla v_h(T)\vert})^*(\xi_T,\vert \jump{\AAA(\xi_T,\nabla v_h)\cdot n_T}_S\vert)
  \\&\quad+\vert F(\xi_T,\nabla v_h(T'))-F(\xi_T,\nabla v_h(T))\vert^2
  \\&\quad+
  \vert F^*(\xi_T,\AAA(\xi_T,\nabla v_h(T')))-F^*(\xi_T,\AAA(\xi_{T'},\nabla v_h(T')))\vert^2 
  \\&\lesssim \vert\jump{F(\xi_T,\nabla v_h)}_S\vert^2+\omega_{p,\omega_S}(h_S)^2\, (1+\vert \nabla v_h(T')\vert^{p(\xi_{T'})s})
  \\& \lesssim\vert\jump{F_h(\cdot,\nabla v_h)}_S\vert^2+ \vert F(\xi_{T'},\nabla v_h(T'))-F(\xi_T,\nabla v_h(T'))\vert^2 \\&\quad + \omega_{p,\omega_S}(h_S)^2\, (1+\vert \nabla v_h(T')\vert^{p(\xi_{T'})s})
  \\&\lesssim\vert\jump{F_h(\cdot,\nabla v_h)}_S\vert^2+\omega_{p,\omega_S}^2\, (1+\vert \nabla v_h(T')\vert^{p(\xi_{T'})s})\,.
  \end{aligned}
		 \end{align}
				Applying for every $T'\in \mathcal{T}_h$ with $T'\subseteq \omega_T$, the $\varepsilon$-Young inequality~\eqref{ineq:young}~with~${\psi=\smash{\varphi_{\vert\nabla v_h(T)\vert}}(\xi_{T'},\cdot)}$, \eqref{thm:best-approx.3.0},
				and the finite overlapping of the element~patches~$\omega_T$,~${T\in \mathcal{T}_h}$, in \eqref{thm:best-approx.2}, for every $\varepsilon>0$, we conclude that
				\begin{align}\label{thm:best-approx.3}
				 \begin{aligned}
					 I_h^1&\lesssim c_{\varepsilon}\sum_{S\in \mathcal{S}_h^{i}}{\sum_{T\in \mathcal{T}_h;T\subseteq \omega_S}{\int_{\omega_T}{((\varphi_h)_{\vert \nabla v_h(T)\vert})^*(\cdot,\vert \jump{\AAA_h(\cdot,\nabla v_h)\cdot n}_S\vert)\,\textup{d}x}}}
  \\&\quad +\varepsilon\sum_{S\in \mathcal{S}_h^{i}}{\sum_{T\in \mathcal{T}_h;T\subseteq \omega_S}{\int_{\omega_T}{(\varphi_h)_{\vert\nabla v_h(T)\vert}(\cdot,\vert \nabla_{\!h} e_h\vert)\,\textup{d}x}}}
					 	\\
					 	&
					 	\lesssim  c_\varepsilon\,\big[\|\smash{h_{\mathcal{S}}^{1/2}}\jump{F_h(\cdot,\nabla v_h)}\|_{2,\mathcal{S}_h^{i}}^2 +\|\omega_p(h_{\mathcal{T}})^2\,(1+\vert\nabla_{\! h} v_h\vert^{p_h(\cdot)s})\|_{1,\Omega} \big]\\&\quad + \varepsilon\,  \sum_{T\in \mathcal{T}_h}{ \rho_{(\varphi_h)_{\vert\nabla v_h(T)\vert},\omega_T}( \nabla_{\!h} e_h)}
					 	\,.
					\end{aligned}
				\end{align}
  Appealing to Lemma \ref{lem:patch_to_element} with $y_h= \nabla_{\!h} e_h\in L^{p_h(\cdot)}(\Omega;\mathbb{R}^d)$, we have that
				\begin{align}\label{thm:best-approx.6}
 \begin{aligned}
				 \sum_{T\in \mathcal{T}_h}{\rho_{(\varphi_h)_{\vert\nabla v_h(T)\vert},\omega_T}( \nabla_{\!h} e_h)}&\lesssim \rho_{(\varphi_h)_{\vert\nabla v_h\vert},\Omega}(\nabla_{\!h} e_h)
     +\|\smash{h_{\mathcal{S}}^{1/2}}\jump{F_h(\cdot,\nabla v_h)}\|_{2,\mathcal{S}_h^{i}}^2
 \\[-2mm]&\quad+\|\omega_p(h_{\mathcal{T}})^2\,(1+\vert \nabla v_h\vert^{p_h(\cdot)s})\|_{1,\Omega}\,.
					\end{aligned}
				\end{align}
				Thus, resorting in \eqref{thm:best-approx.3} to \eqref{lem:efficiency.2} and \eqref{eq:hammera}, we deduce that
				\begin{align}\label{thm:best-approx.3.1}
				  \begin{aligned}
					 I_h^1&
					 \lesssim  c_\varepsilon\,\big[\|F_h(\cdot,\nabla v_h)-F_h(\cdot,\nabla u)\|_{2,\Omega}^2\big]\\&\quad\quad+\|\omega_p(h_{\mathcal{T}})^2\,(1+\vert \nabla u\vert^{p(\cdot)s}+\vert \nabla v_h\vert^{p_h(\cdot)s})\|_{1,\Omega}+\mathrm{osc}_h^2(f,v_h) \big]
 \\&\quad+ \varepsilon\,  
 \|F_h(\cdot,\nabla v_h)-F_h(\cdot,\nabla_{\!h} u_h^{\textit{cr}})\|_{2,\Omega}^2
 \,.
					\end{aligned}
				\end{align}
				
				\textit{ad $I_h^2$.}
				 Applying for every $T\in \mathcal{T}_h$ the $\varepsilon$-Young inequality \eqref{ineq:young} with ${\psi=\varphi_{\vert\nabla v_h(T)\vert}(\xi_T,\cdot)}$, for every ${\varepsilon>0}$, we obtain
				\begin{align}
					\begin{aligned}
						I_h^2&\leq c_\varepsilon\,\rho_{((\varphi_h)_{\vert \nabla v_h\vert})^*,\Omega}(h_{\mathcal{T}}
  f)+
						\varepsilon\,\rho_{(\varphi_h)_{\vert \nabla v_h\vert},\Omega}\big(h_{\mathcal{T}}^{-1}( e_h-\Pi_h^{\textit{av}}e_h)\big)\,.
					\end{aligned}\label{thm:best-approx.4}
				\end{align}
				Then, \hspace{-0.1mm}using \hspace{-0.1mm}for \hspace{-0.1mm}every \hspace{-0.1mm}$T\hspace{-0.1em}\in \hspace{-0.1em}\mathcal{T}_h$ \hspace{-0.1mm}the \hspace{-0.1mm}Orlicz-approximation \hspace{-0.1mm}property \hspace{-0.1mm}of \hspace{-0.1mm}$\Pi_h^{\textit{av}}\colon\smash{\mathcal{S}^{1,\textit{cr}}_D(\mathcal{T}_h)}\hspace{-0.1em}\to \hspace{-0.1em}\mathcal{S}^1_D(\mathcal{T}_h)$ (\textit{cf.}\ \eqref{eq:a6} with $\psi=\varphi_h(\xi_T,\cdot)$ and $a=\vert \nabla v_h(T)\vert$), 
    we find that
				\begin{align}\label{thm:best-approx.5}
                    \begin{aligned}
				 \rho_{(\varphi_h)_{\vert \nabla v_h\vert},\Omega}\big(h_{\mathcal{T}}^{-1}( e_h-\Pi_h^{\textit{av}}e_h)\big)&\lesssim\,\sum_{T\in \mathcal{T}_h}{\rho_{(\varphi_h)_{\vert\nabla v_h(T)\vert},\omega_T}( \nabla_{\!h} e_h)}\\&\quad+\|\omega_p(h_{\mathcal{T}})^2\,(1+\vert \nabla v_h\vert^{p_h(\cdot)s}+\vert\nabla_{\! h} e_h\vert^{p_h(\cdot)s})\|_{1,\Omega}\,.
                    \end{aligned}
		  		\end{align}
				Thus, using \eqref{thm:best-approx.5} and \eqref{thm:best-approx.6} in conjunction with \eqref{lem:efficiency.1} and \eqref{lem:efficiency.2} in \eqref{thm:best-approx.4}, we arrive at
				\begin{align}\label{thm:best-approx.6.1}
				 \begin{aligned}
				 	I_h^2&
					 	\lesssim  c_\varepsilon\,\big[\|F_h(\cdot,\nabla v_h)-F_h(\cdot,\nabla u)\|_{2,\Omega}^2 \\&\quad \quad+\|\omega_p(h_{\mathcal{T}})^2\,(1+\vert \nabla u\vert^{p(\cdot)s}+\vert \nabla v_h\vert^{p_h(\cdot)s}+\vert \nabla_{\!h } e_h\vert^{p_h(\cdot)s})\|_{1,\Omega}+\mathrm{osc}_h^2(f,v_h)\big]
					 	\\&\quad+ \varepsilon\,  
					 	\|F_h(\cdot,\nabla v_h)-F_h(\cdot,\nabla_{\!h} u_h^{\textit{cr}})\|_{2,\Omega}^2
					 	\,.
					 \end{aligned}
				\end{align}
				
				\textit{ad $I_h^3$.}
				Applying for every $T\in \mathcal{T}_h$ the $\varepsilon$-Young inequality \eqref{ineq:young} with $\psi=\varphi_{\vert\nabla v_h(T)\vert }(\xi_T,\cdot)$ together with \eqref{eq:hammera}, for~every~${\varepsilon>0}$, we obtain
				\begin{align}
					\begin{aligned}
					I_h^3&\leq c_\varepsilon\,\|F_h(\cdot,\nabla v_h)-F_h(\cdot,\nabla u)\|_{2,\Omega}^2+\varepsilon\,\rho_{(\varphi_h)_{\vert\nabla v_h\vert },\Omega}(\nabla \Pi_h^{\textit{av}}e_h )\,.
				\end{aligned}\label{thm:best-approx.7}
				\end{align}
				Then, using for every $T\in \mathcal{T}_h$ the Orlicz-stability property of $\Pi_h^{\textit{av}}\colon\smash{\mathcal{S}^{1,\textit{cr}}_D(\mathcal{T}_h)}\to \mathcal{S}^1_D(\mathcal{T}_h)$ (\textit{cf.}\ \eqref{eq:a6} with $\psi=\varphi(\xi_T,\cdot)$ and $a=\vert \nabla v_h(T)\vert$),
    we~find~that
				\begin{align}\label{thm:best-approx.8}
                    \begin{aligned}
				        \rho_{(\varphi_h)_{\vert\nabla v_h\vert },\Omega}( \nabla \Pi_h^{\textit{av}}e_h)&\lesssim\sum_{T\in \mathcal{T}_h}{\rho_{(\varphi_h)_{\vert\nabla v_h(T)\vert},\omega_T}( \nabla_{\!h} e_h)}\\&\quad +\|\omega_p(h_{\mathcal{T}})^2\,(1+\vert \nabla v_h\vert^{p_h(\cdot)s}+\vert \nabla_{\! h} e_h\vert^{p_h(\cdot)s})\|_{1,\Omega}\,.
                    \end{aligned}
				\end{align}
				Thus, using \eqref{thm:best-approx.8} and \eqref{thm:best-approx.6} in conjunction with \eqref{lem:efficiency.2} in \eqref{thm:best-approx.7}, we arrive at
				\begin{align}\label{thm:best-approx.8.1}
				 \begin{aligned}
				 	I_h^3&
					 	\lesssim c_\varepsilon\,\big[\|F_h(\cdot,\nabla v_h)-F_h(\cdot,\nabla u)\|_{2,\Omega}^2 \\&\quad \quad+\|\omega_p(h_{\mathcal{T}})^2\,(1+\vert \nabla v_h\vert^{p_h(\cdot)s}+\vert \nabla_{\! h} e_h\vert^{p_h(\cdot)s})\|_{1,\Omega}+\mathrm{osc}_h^2(f,v_h)\big]\\&\quad+\varepsilon\,  
					 	\|F_h(\cdot,\nabla v_h)-F_h(\cdot,\nabla_{\!h} u_h^{\textit{cr}})\|_{2,\Omega}^2
					 	\,.
					 \end{aligned}
				\end{align}
				
				\textit{ad $I_h^4$.}
				Applying for every $T\in \mathcal{T}_h$ the $\varepsilon$-Young inequality \eqref{ineq:young} with $\psi=\smash{\varphi_{\vert\nabla v_h\vert}}$,~for~every~${\varepsilon>0}$, we obtain
				\begin{align}
					\begin{aligned}
					I_h^4&\leq c_\varepsilon\,\mathrm{osc}_h^2(f,v_h) +\varepsilon\,\rho_{\varphi_{\vert\nabla v_h\vert },\Omega}\big(h_{\mathcal{T}}^{-1} (e_h- \Pi_h^{\textit{av}}e_h) \big)\,.
				\end{aligned}	\label{thm:best-approx.9}
				\end{align}
				Thus, using \eqref{thm:best-approx.5} and \eqref{thm:best-approx.6} in conjunction with \eqref{lem:efficiency.2} in \eqref{thm:best-approx.9}, we arrive at
				\begin{align}\label{thm:best-approx.9.1}
				 \begin{aligned}
				 	I_h^4&
					 	\lesssim c_\varepsilon\,\big[\|F_h(\cdot,\nabla v_h)-F_h(\cdot,\nabla u)\|_{2,\Omega}^2 \\&\quad \quad+\|\omega_p(h_{\mathcal{T}})^2\,(1+\vert \nabla v_h\vert^{p_h(\cdot)s}+\vert \nabla_{\! h} e_h\vert^{p_h(\cdot)s})\|_{1,\Omega}+\mathrm{osc}_h^2(f,v_h)\big]\\&\quad + \varepsilon\, 
					 	\|F_h(\cdot,\nabla v_h)-F_h(\cdot,\nabla_{\!h} u_h^{\textit{cr}})\|_{2,\Omega}^2
					 	\,.
					 \end{aligned}
				\end{align}
				Then, combining \eqref{eq:hammera}, \eqref{thm:best-approx.3.1}, \eqref{thm:best-approx.6.1}, \eqref{thm:best-approx.8.1}, and \eqref{thm:best-approx.9.1} 
				in \eqref{thm:best-approx.1}, for every $\varepsilon>0$, we arrive at
				\begin{align*}
				  \|F_h(\cdot,\nabla v_h)-F_h(\cdot,\nabla_{\!h}u_h^{\textit{cr}})\|_{2,\Omega}^2
      &
					 	\lesssim  c_\varepsilon\,\big[\|F_h(\cdot,\nabla v_h)-F_h(\cdot,\nabla u)\|_{2,\Omega}^2+\mathrm{osc}_h^2(f,v_h) \\&\quad \quad+\|\omega_p(h_{\mathcal{T}})^2\,(1+\vert \nabla u\vert^{p(\cdot)s}+(\vert \nabla v_h\vert+\vert \nabla_{\! h} e_h\vert)^{p_h(\cdot)s})\|_{1,\Omega}
                        \big]
					 	\\&\quad+ \varepsilon\,  
					 	\|F_h(\cdot,\nabla v_h)-F_h(\cdot,\nabla_{\!h} u_h^{\textit{cr}})\|_{2,\Omega}^2
					 	\,.
				\end{align*}
    Next, choosing $\varepsilon>0$ sufficiently small, for every $v_h\in \smash{\mathcal{S}^{1,\textit{cr}}_D(\mathcal{T}_h)}$, we obtain
				\begin{align}\label{thm:best-approx.12}
					\begin{aligned}
						\|F_h(\cdot,\nabla v_h)-F_h(\cdot,\nabla_{\!h}u_h^{\textit{cr}})\|_{2,\Omega}^2&
					 	\lesssim  \|F_h(\cdot,\nabla v_h)-F_h(\cdot,\nabla u)\|_{2,\Omega}^2+\mathrm{osc}_h^2(f,v_h) \\&\quad +\|\omega_p(h_{\mathcal{T}})^2\,(1+\vert \nabla u\vert^{p(\cdot)s}+(\vert \nabla v_h\vert+\vert \nabla_{\! h} e_h\vert)^{p_h(\cdot)s})\|_{1,\Omega}
       \,.
					\end{aligned}\hspace{-2mm}
				\end{align}	
				From \eqref{thm:best-approx.12}, in turn, we deduce that\enlargethispage{6mm}
				\begin{align}\label{thm:best-approx.13}
				 \begin{aligned}
				 \|F_h(\cdot,\nabla_{\!h}u_h^{\textit{cr}})-F_h(\cdot,\nabla u)\|_{2,\Omega}^2&\lesssim \|F_h(\cdot,\nabla v_h)-F_h(\cdot,\nabla_{\!h}u_h^{\textit{cr}})\|_{2,\Omega}^2\\&\quad+
				  \|F_h(\cdot,\nabla v_h)-F_h(\cdot,\nabla u)\|_{2,\Omega}^2\\&
				 \lesssim \|F_h(\cdot,\nabla v_h)-F_h(\cdot,\nabla u)\|_{2,\Omega}^2 +\mathrm{osc}_h^2(f,v_h)\\&\quad +\|\omega_p(h_{\mathcal{T}})^2\,(1+\vert \nabla u\vert^{p(\cdot)s}+(\vert \nabla v_h\vert+\vert \nabla_{\! h} e_h\vert)^{p_h(\cdot)s})\|_{1,\Omega} \,.
				 \end{aligned}\hspace{-5mm}
				\end{align}
				Taking in \eqref{thm:best-approx.13} the infimum with respect to $v_h\in \smash{\mathcal{S}^1_D(\mathcal{T}_h)}$, we conclude~the~claimed~estimate.
			\end{proof}
			
			\hspace*{-1mm}Adding \hspace*{-0.1mm}oscillation \hspace*{-0.1mm}terms \hspace*{-0.1mm}on \hspace*{-0.1mm}the \hspace*{-0.1mm}right-hand \hspace*{-0.1mm}side \hspace*{-0.1mm}measuring \hspace*{-0.1mm}the \hspace*{-0.1mm}regularity \hspace*{-0.1mm}of \hspace*{-0.1mm}${F(\cdot,\hspace*{-0.5mm}\nabla u)\hspace*{-0.15em}\in\hspace*{-0.15em} L^{p(\cdot)}(\Omega;\mathbb{R}^d)}$, it is possible to extend the best-approximation result in Theorem \ref{thm:best-approx2} to $\mathcal{S}^{1,cr}_D(\mathcal{T}_h)$.
			
			\begin{theorem}\label{thm:best-approx2} 
				Let $p\in C^0(\overline{\Omega})$ with $p^->1$ and $\delta\ge 0$ and let $f\in L^{p'(\cdot)}(\Omega)\cap \bigcap_{h\in (0,h_0]}{L^{p_h'(\cdot)}(\Omega)}$ for some $h_0\hspace{-0.15em}>\hspace{-0.15em}0$. \hspace{-0.2em}Then, there exists some $s\hspace{-0.15em}>\hspace{-0.15em}1$, which can chosen to be close~to~$1$~if~${h\hspace{-0.15em}>\hspace{-0.15em}0}$~is~close~to~$0$, such that if $u\in W^{1,p(\cdot)s}_D(\Omega)$, then for every $h\in (0,h_0]$, it holds that
				\begin{align*}
					\begin{aligned}
						\|F_h(\cdot,\nabla_{\!h} u_h^{\textit{cr}})-F_h(\cdot,\nabla u)\|_{2,\Omega}^2&\lesssim \inf_{v_h\in\mathcal{S}^{1,cr}_D(\mathcal{T}_h)}{\big[\|F_h(\cdot,\nabla_{\!h} v_h)-F_h(\cdot,\nabla u)\|_{2,\Omega}^2+\mathrm{osc}_h^2(f,v_h)}\\&\;\;+
						\sum_{T\in \mathcal{T}_h}{\|F(\cdot,\nabla u)-\langle F(\cdot,\nabla u)\rangle_{\omega_T}\|_{2,\omega_T}^2}
						\\&\;\;+\|\omega_p(h_{\mathcal{T}})^2\,(1+\vert \nabla u\vert^{p(\cdot)s}+(\vert \nabla_{\!h} v_h\vert+\vert \nabla_{\!h} v_h-\nabla_{\!h} u_h^{\textit{cr}}\vert)^{p_h(\cdot)s})\|_{1,\Omega}\big]\,,
					\end{aligned}  
				\end{align*}
				where the hidden constant  also depends on $s>1$ and the chunkiness $\omega_0>0$.\enlargethispage{6mm}
			\end{theorem}
			
			\begin{proof}
				Using Theorem \ref{thm:best-approx2}, for every $v_h\in\mathcal{S}^{1,cr}_D(\mathcal{T}_h)$,  we find that
				\begin{align}\label{thm:best-approx2.1} 
					\begin{aligned}
					\|F_h(\cdot,\nabla_{\!h} u_h^{\textit{cr}})-F_h(\cdot,\nabla u)\|_{2,\Omega}^2&\lesssim \|F_h(\cdot,\nabla \Pi_h^{\textit{av}}v_h)-F_h(\cdot,\nabla u)\|_{2,\Omega}^2+\mathrm{osc}_h^2(f,\Pi_h^{\textit{av}}v_h)
					\\&\quad+\|\omega_p(h_{\mathcal{T}})^2\,(1+\vert \nabla u\vert^{p(\cdot)s})\|_{1,\Omega}\\&\quad+\|\omega_p(h_{\mathcal{T}})^2(\vert \nabla_{\!h} \Pi_h^{\textit{av}} v_h\vert+\vert \nabla_{\!h} \Pi_h^{\textit{av}} v_h-\nabla_{\!h} u_h^{\textit{cr}}\vert)^{p_h(\cdot)s})\|_{1,\Omega}
					\\&\eqqcolon I_h^1+I_h^2+I_h^3+I_h^4\,,
				\end{aligned}
				\end{align}
				so it is left to estimate $I_h^1$, $I_h^2$, and $I_h^4$:
				
			\textit{ad $I_h^1$.} Using \cite[Lem.\  3.8]{K22CR} and Lemma \ref{lem:A-Ah}\eqref{eq:Fh-F} (or Lemma \ref{lem:Ax-Axh}\eqref{eq:Fxh-Fx}), for every $T\in \mathcal{T}_h$, we find that
			\begin{align}\label{thm:best-approx2.2} 
				\begin{aligned}
				 \|F_h(\cdot,\nabla_{\! h} v_h)-F_h(\cdot,\nabla\Pi_h^{\textit{av}}v_h)\|_{2,T}^2&\lesssim \|F(\xi_T,\nabla_{\! h} v_h)-F(\xi_T,\nabla u)\|_{2,\omega_T}^2\\&\quad+\| h_{\mathcal{S}}^{1/2}\jump{F(\xi_T,\nabla_{\!h}v_h)}\|_{2,\mathcal{S}_h^{i}(T)}^2
				 \\&\lesssim\|F_h(\cdot,\nabla_{\! h} v_h)-F_h(\cdot,\nabla u)\|_{2,\omega_T}^2\\&\quad+\| h_{\mathcal{S}}^{1/2}\jump{F_h(\cdot,\nabla_{\!h}v_h)-\langle F(\cdot,\nabla u )\rangle_{\omega_T}}\|_{2,\mathcal{S}_h^{i}(T)}^2\\&\quad+\|\omega_{p,\omega_T}(h_{\mathcal{T}})^2\,(1+\vert \nabla u\vert^{p(\cdot)s}+\vert \nabla_{\!h} v_h\vert^{p_h(\cdot)s})\|_{1,\omega_T}\,, 
				 	\end{aligned}
			\end{align}
			where, due to the discrete trace inequality (\textit{cf.}\ \cite[Lem.\ 12.8]{EG21}), it holds that
			\begin{align}\label{thm:best-approx2.3}
				\begin{aligned}
				\| F_h(\cdot,\nabla_{\!h}v_h)-\langle F(\cdot,\nabla u )\rangle_{\omega_T}\|_{2,\mathcal{S}_h^{i}(T)}^2&\lesssim \| F_h(\cdot,\nabla_{\!h}v_h)-\langle F(\cdot,\nabla u )\rangle_{\omega_T}\|_{2,\omega_T}^2
				\\&\lesssim \| F_h(\cdot,\nabla_{\!h}v_h)-F(\cdot,\nabla u )\|_{2,\omega_T}^2\\&\quad + \| F(\cdot,\nabla u )-\langle F(\cdot,\nabla u )\rangle_{\omega_T}\|_{2,\omega_T}^2\,.
			\end{aligned} 
			\end{align}
			Therefore, using \eqref{thm:best-approx2.3} in \eqref{thm:best-approx2.2} and  subsequent summation with respect to $T\in \mathcal{T}_h$ yield that
			\begin{align}\label{thm:best-approx2.4}
			\begin{aligned}
				I_h^1&\lesssim \|F_h(\cdot,\nabla_{\! h} v_h)-F_h(\cdot,\nabla u)\|_{2,\Omega}^2
				+\sum_{T\in \mathcal{T}_h}{\|F(\cdot,\nabla u)-\langle F(\cdot,\nabla u)\rangle_{\omega_T}\|_{2,\omega_T}^2}
				\\&\quad+ \|\omega_p(h_{\mathcal{T}})^2\,(1+\vert \nabla u\vert^{p(\cdot)s}+\vert \nabla_{\!h} v_h\vert^{p_h(\cdot)s})\|_{1,\Omega}\,.
					\end{aligned} 
			\end{align}
			
			\textit{ad $I_h^2$.} Using the shift-change \eqref{lem:shift-change.3}, we find that
			\begin{align}\label{thm:best-approx2.5}
					I_h^2\lesssim \mathrm{osc}_h^2(f,v_h)+ \|F_h(\cdot,\nabla_{\! h} v_h)-F_h(\cdot,\nabla u)\|_{2,\Omega}^2\,.
			\end{align}
			
			\textit{ad $I_h^4$.} Using stability properties of $\Pi_h^{av}\colon \mathcal{S}^{1,cr}_D(\mathcal{T}_h)\to \mathcal{S}^1_D(\mathcal{T}_h)$ (\textit{cf.}\ \eqref{eq:a6} with $\psi= \vert \cdot\vert^{p'(\xi_T)s}$ and $a=0$)  and Lemma \ref{lem:A-Ah}\eqref{eq:Fh-F} (or Lemma \ref{lem:Ax-Axh}\eqref{eq:Fxh-Fx}), for every $T\in \mathcal{T}_h$, we find that
			\begin{align}\label{thm:best-approx2.6}
				\rho_{p'(\xi_T)s,T}(\nabla_{\!h}\Pi_h^{\textit{av}}v_h)&\lesssim \rho_{p'_h(\cdot)s,\omega_T}(\nabla_{\!h}v_h)+\|\omega_{p,\omega_T}(h_{\mathcal{T}})^2\,(1+\vert\nabla_{\!h}v_h\vert^{p'_h(\cdot)s})\|_{1,\omega_T}\,.
			\end{align}
			Therefore, summation of \eqref{thm:best-approx2.6} with respect to   $T\in \mathcal{T}_h$ yields that
			\begin{align}\label{thm:best-approx2.7}
				I_h^4\lesssim \|\omega_{p,\omega_T}(h_{\mathcal{T}})^2\,(1+\vert\nabla_{\!h}v_h\vert^{p'_h(\cdot)s})\|_{1,\omega_T}\,.
			\end{align}
			Combining \eqref{thm:best-approx2.2}, \eqref{thm:best-approx2.5}, and \eqref{thm:best-approx2.7} in \eqref{thm:best-approx2.1}, we arrive at the claimed best-approximation~result.
			\end{proof}

		\newpage
 \section{\textit{A priori} error analysis}\label{sec:a_priori}
	
		\qquad In this section, we establish \textit{a priori} error estimates for the $S^{1,\textit{\textrm{cr}}}_D(\mathcal{T}_h)$-approximation \eqref{eq:pDirichletS1crD} of the $p(\cdot)$-Dirichlet problem \eqref{eq:pDirichletW1p}. To this end, we resort to the \textit{medius} error analysis of Section~\ref{sec:medius}, which allows us to tranfer the approximation rate capabilities of the $S^1_D(\mathcal{T}_h)$-approximation \eqref{eq:pDirichletS1D} (\textit{cf.}\ Theorem \ref{P1_apriori}) to the $S^{1,\textit{\textrm{cr}}}_D(\mathcal{T}_h)$-approximation \eqref{eq:pDirichletS1crD}.
 
		\begin{theorem}\label{thm:rate_u}
		  Let $p\in C^{0,\alpha}(\overline{\Omega})$ with $\alpha\in (0,1]$ and $p^->1$ and let $\delta\ge 0$.~Moreover,~let~$F(\cdot,\nabla u)\in W^{1,2}(\Omega;\mathbb{R}^d)$, $(\delta^{p(\cdot)-1}+\vert z\vert )^{p'(\cdot)-2}\vert f\vert^2\in L^1(\{p> 2\})$, and 
 $ f\in L^{(p^-)'}(\Omega)$ or $\Gamma_D=\partial \Omega$.
		 Then, assuming that $h_{\max}\sim h_{\mathcal{T}}$, there exists some $s\hspace{-0.1em}>\hspace{-0.1em}1$, which can chosen to be close to $1$ if $h_{\max}>0$ is close to $0$, such that if~$f\hspace{-0.1em}\in\hspace{-0.1em} L^{p'(\cdot)s}(\Omega)$, then
		 \begin{align*}
		 \|F_h(\cdot,\nabla_{\!h} u_h^{\textit{cr}})-F_h(\cdot,\nabla u)\|_{2,\Omega}^2&\lesssim h_{\max}^{2\alpha}\,\big(1+\|\nabla F(\cdot,\nabla u)\|_{2,\Omega}^{2}+\|(\delta^{p(\cdot)-1}+\vert z\vert)^{p'(\cdot)-2}\vert f\vert^2\|_{1,\{p>2\}}\\&\quad+
  \sigma(f,s)+\rho_{p'(\cdot)s,\Omega}(f)+\rho_{p(\cdot)s,\Omega}( \nabla u)\big)^s\,,
		 \end{align*}
        where the hidden constants also depend on $s>1$ and $\sigma(f,s)\coloneqq 1+\rho_{(p^-)',\Omega}(f)$ if $f\in L^{(p^-)'}(\Omega)$ and $\sigma(f,s)\coloneqq 1+\rho_{p'(\cdot)s,\Omega}(f)^{\smash{(p^-)'/(p^+)'}}$ if $\Gamma_D=\partial \Omega$.
		\end{theorem}

 \begin{remark}[Comments on the regularity assumptions in Theorem \ref{thm:rate_u}]\hphantom{                   }
 \begin{itemize}[noitemsep,topsep=2pt,labelwidth=\widthof{(iii)},leftmargin=!]
    \item[(i)] If $p\in C^{0,\alpha}(\overline{\Omega})$ with $\alpha\in (0,1)$, one cannot expect that $F(\cdot,\nabla u)\in W^{1,2}(\Omega;\mathbb{R}^d)$, even locally.
    However, appealing to \cite[Rem.\  4.5]{BDS15}, one can expect that $F(\cdot,\nabla u)\in \mathcal{N}^{\alpha,2}(\Omega)^d$, where $\mathcal{N}^{\alpha,2}(\Omega)$ is the Nikolski\u{\i} space with order of differentiability $\alpha\in (0,1)$, which should still be enough the justify the arguments below, but is beyond the scope of this article.
  \item[(ii)] If $p\leq 2$ in $\overline{\Omega}$ in Theorem \ref{thm:rate_u}, then the assumption  $(\delta^{p(\cdot)-1}+\vert z\vert )^{p'(\cdot)-2}\vert f\vert^2\in L^1(\{p> 2\})$ is trivially satisfied.

  \item[(iii)]  
  If $(\delta^{p(\cdot)-1}+\vert z\vert )^{p'(\cdot)-2}\vert \nabla z\vert^2\hspace{-0.1em}\in L^1(\{p\hspace{-0.1em}>\hspace{-0.1em} 2\})$ in Theorem \ref{thm:rate_u}, then, due to $f\hspace{-0.1em}=\hspace{-0.1em}-\textup{div}\,z$~in~$L^{p'(\cdot)}(\Omega)$, it holds that   $(\delta^{p(\cdot)-1}+\vert z\vert )^{p'(\cdot)-2}\vert f\vert^2\in L^1(\{p> 2\})$.

  \item[(iv)] If $f\in L^2(\{p> 2\})$ and $\delta>0$, then  $(\delta^{p(\cdot)-1}+\vert z\vert )^{p'(\cdot)-2}\vert f\vert^2\leq \delta^{2-p(\cdot)}\vert f\vert^2$ a.e.\ in $\{p> 2\}$, \textit{i.e.}, it holds that $(\delta^{p(\cdot)-1}+\vert z\vert )^{p'(\cdot)-2}\vert f\vert^2\in L^1(\{p> 2\})$.

  \item[(v)] If $p \in  (1,\infty)$, then it holds that $F(\nabla u)\in W^{1,2}(\Omega;\mathbb{R}^d)$ if and~only~if~$F^*(z)\in\ W^{1,2}(\Omega;\mathbb{R}^d)$ with $\vert F(\nabla u)\vert \sim \vert F^*(z)\vert$ a.e.\ in $\Omega$ and $\vert \nabla F(\nabla u)\vert \sim  \vert \nabla F^*(z)\vert$~a.e.~in~$\Omega$ (\textit{cf}.\ \cite[Lem.\  2.3]{DKRI14}). In addition, due to \cite[Lem.\  2.11]{K22CR}, it holds that $\vert \nabla F^*(z)\vert\sim (\delta^{p-1}+\vert z\vert )^{(p'-2)/2}\vert \nabla z\vert$~a.e.~in~$\Omega$.
  Thus, Theorem \ref{thm:rate_u} extends the \textit{a priori} error analysis in \cite{K22CR}.
 \end{itemize}
 
 \end{remark}

 Since the right-hand side in Theorem \ref{thm:best-approx} still involves the discrete primal solution, before proving Theorem \ref{thm:rate_u}, we first derive the following \textit{a priori} estimate.

 \begin{lemma}\label{lem:a_priori}
     Let the assumptions of Theorem  \ref{thm:rate_u} be satisfied. Then, there exists  some $s>1$, which can chosen to be close to $1$ if $h_{\max}>0$ is close to $0$, such that if $f\in L^{p'(\cdot)s}(\Omega)$, then 
     \begin{align*}
         \rho_{\varphi_h,\Omega}(\nabla_{\! h}u_h^{\textit{\textrm{cr}}})\lesssim  \sigma(f,s)\,,
     \end{align*}
      where the hidden constants also depend on $s>1$ and $\sigma(f,s)$ is defined as in Theorem  \ref{thm:rate_u}.
 \end{lemma}

\begin{proof}
    We distinguish the cases $f\in L^{(p^-)'}(\Omega)$ and $\Gamma_D=\partial \Omega$:

    \textit{Case $f\hspace{-0.15em}\in\hspace{-0.15em} L^{(p^-)'}(\Omega)$.}
    Since $I_h^{\textit{\textrm{cr}}}(u_h^{\textit{\textrm{cr}}})\hspace{-0.15em}\leq\hspace{-0.15em} I_h^{\textit{\textrm{cr}}}(0)$, applying  the $\varepsilon$-Young inequality \eqref{ineq:young}~with~${\psi\hspace{-0.15em}=\hspace{-0.15em}\vert\hspace{-0.15em} \cdot\hspace{-0.15em}\vert^{p^-}}\!\!$,  the 
    $L^{p^-}$-stability of $\Pi_h$, the discrete Poincar\'e inequality in $\mathcal{S}^{1,\textit{\textrm{cr}}}_D(\mathcal{T}_h)$ (\textit{cf.} \cite[Prop.\ I.4.13]{Tem77}), and $p^-\lesssim p_h$ a.e.\ in $\Omega$,
    we find that
    \begin{align*}
        \rho_{\varphi_h,\Omega}(\nabla_{\! h}u_h^{\textit{\textrm{cr}}})&\lesssim (f_h,\Pi_h  u_h^{\textit{\textrm{cr}}})_{\Omega}
        \\&\lesssim c_\varepsilon\,\rho_{(p^-)',\Omega}(f_h)+\varepsilon\,
        \rho_{p^-,\Omega}(\nabla_{\! h} u_h^{\textit{\textrm{cr}}})
        \\&\lesssim c_\varepsilon\,\rho_{(p^-)',\Omega}(f)+\varepsilon\,
        (1+\rho_{p_h(\cdot),\Omega}(\nabla_{\! h} u_h^{\textit{\textrm{cr}}}))
        \,.
    \end{align*}
    For $\varepsilon>0$ sufficiently small, using $p_h\lesssim 1+\varphi_h$, we conclude that $\rho_{\varphi_h,\Omega}(\nabla_{\! h}u_h^{\textit{\textrm{cr}}})\lesssim 1+\rho_{(p^-)',\Omega}(f)$. 

    \textit{Case $\Gamma_D=\partial\Omega$.} Let $R>0$ be such that $\Omega\subseteq B_R^d(0)$. If $f\in L^{p'(\cdot)s}(\Omega)$,
    by \cite[Thm.\  14.1.2]{DHHR11},  the vector field $G \coloneqq (-\nabla u)|_{\Omega} $, where $u\in W^{2,p(\cdot)}(B_R^d(0))\cap W^{1,p(\cdot)}_0(B_R^d(0))$ denotes the unique solution of  $-\Delta u =\overline{f}$ a.e.\ in $B_R^d(0)$, where $\overline{f}\coloneqq f$ a.e.\ in $\Omega$ and $\overline{f}\coloneqq 0$ a.e.\ in $B_R^d(0)\setminus\Omega$, satisfies 
     $G\in W^{1,p'(\cdot)s}(\Omega;\mathbb{R}^d)$, $\textup{div}\,G=f$ a.e.\ in $\Omega$,  and 
    $\|G\|_{p'(\cdot)s,\Omega}+\|\nabla G\|_{p'(\cdot)s,\Omega}\lesssim \|f\|_{p'(\cdot)s,\Omega}$, with a constant depending~on~$s$,~$p$, and $\Omega$. Using \cite[Lem.\  3.2.5]{DHHR11}, we find that
    \begin{align}
        \rho_{p'(\cdot)s,\Omega}(G)+\rho_{p'(\cdot)s,\Omega}(\nabla G)\lesssim \rho_{p'(\cdot)s,\Omega}(f)^{\smash{(p^-)'/(p^+)'}}\,.\label{lem:a_priori.1}
    \end{align}    
    Denote by $\Pi_h^{\textit{\textrm{rt}}}\colon W^{1,1}(\Omega;\mathbb{R}^d)\to \mathcal{R}T^0(\mathcal{T}_h)$, the Raviart--Thomas quasi-iterpolation operator. Then, $\textup{div}\,\Pi_h^{\textit{\textrm{rt}}} G=\Pi_h \textup{div}\, G=f_h$ a.e.\ in $\Omega$. Appealing to  \cite[Thm.\ 16.4]{EG21},~for~every~$T\in \mathcal{T}_h$, we have that
    \begin{align}
        \rho_{p'(\xi_T),T}(\Pi_h^{\textit{\textrm{rt}}}G)\lesssim \rho_{p'(\xi_T),T}(G)+\rho_{p'(\xi_T),T}(\nabla G)\,.\label{lem:a_priori.2}
    \end{align}
    Since $\Pi_h^{\textit{\textrm{cr}}}(u_h^{\textit{\textrm{cr}}})\leq I_h^{\textit{\textrm{cr}}}(0)$, using $\textup{div}\,\Pi_h^{\textit{\textrm{rt}}} G=f_h$  a.e.\ in $\Omega$, the discrete integration-by-parts formula \eqref{eq:pi0}, for every $T\in \mathcal{T}_h$ the $\varepsilon$-Young inequality \eqref{ineq:young} with $\psi=\vert \cdot\vert^{p(\xi_T)}$, \eqref{lem:a_priori.2},~and~\eqref{lem:a_priori.1},~we~find~that
    \begin{align*}
        \rho_{\varphi_h,\Omega}(\nabla_{\! h}u_h^{\textit{\textrm{cr}}})&\lesssim (f_h,\Pi_h  u_h^{\textit{\textrm{cr}}})_{\Omega}
        =(\textup{div}\,\Pi_h^{\textit{\textrm{rt}}} G,\Pi_h  u_h^{\textit{\textrm{cr}}})_{\Omega}
        =-(\Pi_h\Pi_h^{\textit{\textrm{rt}}} G,\nabla_{\! h}  u_h^{\textit{\textrm{cr}}})_{\Omega}
        \\&\lesssim c_\varepsilon\,\rho_{p_h'(\cdot),\Omega}(\Pi_h^{\textit{\textrm{rt}}} G)+\varepsilon\,
        \rho_{p_h(\cdot),\Omega}(\nabla_{\! h} u_h^{\textit{\textrm{cr}}})\,,
        \\&\lesssim c_\varepsilon\,\big(\rho_{p_h'(\cdot)}(G)+\rho_{p_h'(\cdot),\Omega}(\nabla G)\big)+\varepsilon\,
        \rho_{p_h(\cdot),\Omega}(\nabla_{\! h} u_h^{\textit{\textrm{cr}}})
        \\&\lesssim c_\varepsilon\,\big(1+\rho_{p'(\cdot)s,\Omega}(G)+\rho_{p'(\cdot)s,\Omega}(\nabla G)\big)+\varepsilon\,
        \rho_{p_h(\cdot),\Omega}(\nabla_{\! h} u_h^{\textit{\textrm{cr}}})
        \\&\lesssim c_\varepsilon\,\big(1+\rho_{p'(\cdot)s,\Omega}(f)^{\smash{(p^-)'/(p^+)'}}\big)+\varepsilon\,
        \rho_{p_h(\cdot),\Omega}(\nabla_{\! h} u_h^{\textit{\textrm{cr}}})\,.
    \end{align*}
    For $\varepsilon\hspace{-0.1em}>\hspace{-0.1em}0$ sufficiently small, using $p_h\hspace{-0.1em}\lesssim \hspace{-0.1em}1+\varphi_h$, we conclude that $ \rho_{\varphi_h}(\nabla_{\! h}u_h^{\textit{\textrm{cr}}})\hspace{-0.1em}\lesssim\hspace{-0.1em}  1+\rho_{p'(\cdot)s}(f)^{\smash{(p^-)'/(p^+)'}}$. 
\end{proof}
		
		\begin{proof}[Proof (of Theorem \ref{thm:rate_u}).]
		  The convexity of $(\varphi_{\vert \nabla_{\!h} v_h(T)\vert})^*(\xi_T,\cdot)$ and that $\Delta_2((\varphi_{\vert \nabla_{\!h} v_h(T)\vert})^*(\xi_T,\cdot))\lesssim 2^{\smash{\max\{2,(p^-)'\}}}$ for all $T\in \mathcal{T}_h$, the Orlicz-stability of $\Pi_h$ (\textit{cf.}\ \cite[Cor.\ A.8,~(A.12)]{kr-phi-ldg}), and the shift change \eqref{lem:shift-change.3}, for every $v_h\in\mathcal{S}^1_D(\mathcal{T}_h)$, yield\enlargethispage{10mm}
		 \begin{align}\label{cor:rate_improve.1}
		  \begin{aligned}
		  \textrm{osc}_h^2(f,v_h) 
    \lesssim \rho_{((\varphi_h)_{\vert \nabla u\vert})^*,\Omega}(h_{\mathcal{T}}f )+\|F_h(\cdot,\nabla v_h)-F_h(\cdot,\nabla u)\|_{2,\Omega}^2\,.
		  \end{aligned}
		 \end{align}
 		 Using \eqref{cor:rate_improve.1} in Theorem \ref{thm:best-approx}, for every $v_h\in\mathcal{S}^1_D(\mathcal{T}_h)$, we find that
  \begin{align}
 \begin{aligned}
		  \|F_h(\cdot,\nabla_{\!h} u_h^{\textit{cr}})-F_h(\cdot,\nabla u)\|_{2,\Omega}^2&\lesssim \|F_h(\cdot,\nabla v_h)-F_h(\cdot,\nabla u)\|_{2,\Omega}^2\\&\quad+h_{\max}^{2\alpha}\,\rho_{p_h(\cdot)s,\Omega}(\nabla v_h)+h_{\max}^{2\alpha}\,\rho_{p_h(\cdot)s,\Omega}(\nabla_{\! h} u_h^{\textit{cr}})
  \\&\quad+h_{\max}^{2\alpha}\,\big(1+\rho_{p(\cdot)s,\Omega}(\nabla u)\big)+\rho_{((\varphi_h)_{\vert \nabla u\vert})^*,\Omega}(h_{\mathcal{T}}f ) \,.
 \end{aligned}\label{cor:rate_improve.2}
		 \end{align}
  Using  that, appealing to \cite[Lem.\  4.3 \& Corollary 3.6]{BDS15},  it holds that
  \begin{align}\label{cor:rate_improve.2.0}
 \begin{aligned}
 \|F_h(\cdot,\nabla \Pi_h^{\textit{sz}}u)-F_h(\cdot,\nabla u)\|_{2,\Omega}^2&\lesssim h_{\max}^2\|\nabla F(\cdot,\nabla u)\|_{2,\Omega}^2+h_{\max}^{2\alpha}\,\big(1+\rho_{p(\cdot)s,\Omega}(\nabla u)\big)\,,\\
 \rho_{p_h(\cdot)s,\Omega}(\nabla \Pi_h^{\textit{sz}}u)&\lesssim 1+\rho_{p(\cdot)s,\Omega}(\nabla \Pi_h^{\textit{sz}}u)\lesssim 1+\rho_{p(\cdot)s,\Omega}(\nabla u)\,,
 \end{aligned}
  \end{align}
  choosing $v_h=\Pi_h^{\textit{sz}}u\in \mathcal{S}^1_D(\mathcal{T}_h)$ in \eqref{cor:rate_improve.2}, we arrive at
  \begin{align}\label{cor:rate_improve.3}
  \begin{aligned}
 \|F_h(\cdot,\nabla_{\!h} u_h^{\textit{cr}})-F_h(\cdot,\nabla u)\|_{2,\Omega}^2&\lesssim h_{\max}^2\|\nabla F(\cdot,\nabla u)\|_{2,\Omega}^2+ h_{\max}^{2\alpha}\,\big(1+\rho_{p(\cdot)s,\Omega}(\nabla u)\big)\\&\quad +h_{\max}^{2\alpha}\,\rho_{p_h(\cdot)s,\Omega}(\nabla_{\! h} u_h^{\textit{cr}})+\rho_{((\varphi_h)_{\vert \nabla u\vert})^*,\Omega}(h_{\mathcal{T}}f )\,.
 \end{aligned}
  \end{align}
  By the aid of Lemma \ref{lem:A-Ah}\eqref{eq:phih-phi}, also using that $2\alpha\leq {2\wedge (p^+)'}+\alpha$, it holds that
  \begin{align}\label{cor:rate_improve.3.2}
 \begin{aligned}
 \rho_{((\varphi_h)_{\vert \nabla u\vert})^*,\Omega}(h_{\mathcal{T}} f ) 
 \lesssim \rho_{(\varphi_{\vert \nabla u\vert})^*,\Omega}(h_{\mathcal{T}} f )+h_{\max}^{2\alpha}\,\big(1+\rho_{p'(\cdot)s,\Omega}(f)+\rho_{p(\cdot)s,\Omega}(\nabla u)\big)\,.
 \end{aligned}
  \end{align}
  Next, for every $T\in \mathcal{T}_h$ and $x\in T$, we need to distinguish the cases $p(x)\in (1,2)$~and~$p(x)\in [2,\infty)$:

  \textit{Case $p(x)\in (1,2]$.} If $p(x)\in (1,2]$, then, there holds the elementary inequality
		 \begin{align*}
		 (\varphi_{\vert a\vert })^*(x,h\,t)\lesssim \lambda^2\,\big(\varphi^*(x,t)+\varphi(x,\vert a\vert)\big)\quad\textup{ for all }a\in \mathbb{R}^d\,,\; t\ge 0\,,\; \lambda\in [0,1]\,,
		  \end{align*}
		  which follows from the definition of shifted \mbox{$N$-functions} (\textit{cf.}\ \eqref{eq:phi_shifted}) and the shift change~\eqref{lem:shift-change.3}~(\textit{i.e.}, with $b = 0$ and using that $\vert F(x,a)\vert^2=\varphi(x,\vert a\vert )$ for all $a\in \mathbb{R}^d$), so that
  \begin{align}\label{cor:rate_improve.4}
  \begin{aligned} 
 (\varphi_{\vert \nabla u(x)\vert })^*(x,h_T\vert f(x)\vert )&\lesssim h^2_T\,\big(\varphi^*(x,\vert f(x)\vert )+\varphi(x,\vert \nabla u(x)\vert)\big)\,.
 \end{aligned}
  \end{align}

  \textit{Case $p(x)\in (2,\infty)$.} 
  Since  $(\varphi_{\vert a\vert})^*(x,\lambda\,t)\lesssim \lambda^2\,(\delta^{p(x)-1}+\vert a\vert)^{p'(x)-2}t^2$ for all $a\in \mathbb{R}^d$~and~${t,\lambda\ge 0}$, we have that\vspace{-1mm}
  \begin{align}\label{cor:rate_improve.5}
		 (\varphi_{\smash{\vert \nabla u(x)\vert}})^*(x,h_T \vert f(x)\vert )
 &\lesssim h_T^2\,(\delta^{p(x)-1}+\vert z(x)\vert)^{p'(x)-2} \vert f(x)\vert^2 \,.
		 \end{align}
		  Combining \eqref{cor:rate_improve.4} and \eqref{cor:rate_improve.5}, we deduce that
 \begin{align}\label{cor:rate_improve.8}
  \begin{aligned}
 \rho_{(\varphi_{\smash{\vert \nabla u\vert}})^*,\Omega}(h_{\mathcal{T}}f)&\lesssim h_{\max}^2\,\big(
 \rho_{\varphi^*,\Omega}(f)+\rho_{\varphi,\Omega}(\nabla u)+\|(\delta^{p(\cdot)-1}+\vert z\vert)^{p'(\cdot)-2} \vert f\vert^2\|_{1,\{p>2\}}\big)\,.
 \end{aligned}
 \end{align}
 Using \eqref{cor:rate_improve.8} in \eqref{cor:rate_improve.3.2}, we find that
 \begin{align}\label{cor:rate_improve.9}
 \begin{aligned}
 \rho_{((\varphi_h)_{\vert \nabla u\vert})^*,\Omega}(h_{\mathcal{T}} f )&\lesssim h_{\max}^{2\alpha}\,\big(1+\rho_{p'(\cdot)s,\Omega}(f)+\rho_{p(\cdot)s,\Omega}(\nabla u)\big)\\&\quad + h_{\max}^2\,
 \|(\delta^{p(\cdot)-1}+\vert z\vert)^{p'(\cdot)-2} \vert f\vert^2\|_{1,\{p>2\}}\,.
 \end{aligned}
 \end{align}
 In addition, 
 due to Lemma \ref{lem:a_priori}, \cite[Lem.\  12.1]{EG21}, $h_{\max}\sim h_{\mathcal{T}}$, and $p_h\lesssim 1+\varphi_h$, we have that
 \begin{align}\label{cor:rate_improve.10}
    \begin{aligned}
     \rho_{p_h(\cdot)s,\Omega}(\nabla_{\! h} u_h^{\textit{cr}})&\lesssim 
     h_{\max}^{d(1-s)}\rho_{p_h(\cdot),\Omega}(\nabla_{\! h} u_h^{\textit{cr}})^s
     \lesssim h_{\max}^{d(1-s)}\big(1+ \sigma(f;s)\big)^s\,.
     \end{aligned}
 \end{align}
 Next, abbreviating
 \begin{align*}
 	\Theta(s)&\coloneqq 1+\|\nabla F(\cdot,\nabla u)\|_{2,\Omega}^{2}+\|(\delta^{p(\cdot)-1}+\vert z\vert)^{p'(\cdot)-2}\vert f\vert^2\|_{1,\{p>2\}}
 	\\&\quad+\sigma(f;s)+\rho_{p'(\cdot)s,\Omega}(f)+\rho_{p(\cdot)s,\Omega}( \nabla u)\,,
 \end{align*}
	using \eqref{cor:rate_improve.9}  and \eqref{cor:rate_improve.10} in \eqref{cor:rate_improve.3}, we arrive at 
 \begin{align}
 \label{cor:rate_improve.11}
    \begin{aligned}
     \|F_h(\cdot,\nabla_{\!h} u_h^{\textit{cr}})-F_h(\cdot,\nabla u)\|_{2,\Omega}^2&\lesssim h_{\max}^{2\alpha +d(1-s)}\,\Theta(s)^s\,.
 \end{aligned}
 \end{align}
 Using the \textit{a priori} error estimate \eqref{cor:rate_improve.11}, we can improve the \textit{a priori} estimate \eqref{cor:rate_improve.10} and,~in~turn, the \textit{a priori} error estimate \eqref{cor:rate_improve.11}.
 First, due to \eqref{cor:rate_improve.2.0}$_2$, $h_{\max}\hspace*{-0.1em}\sim\hspace*{-0.1em} h_{\mathcal{T}}$, and \cite[Lem.\  12.1]{EG21},~we~have~that
 \begin{align}\label{cor:rate_improve.12}
    \begin{aligned}
     \rho_{p_h(\cdot)s,\Omega}(\nabla_{\! h} u_h^{\textit{cr}})&\lesssim  \rho_{p_h(\cdot)s,\Omega}(\nabla_{\! h} u_h^{\textit{cr}}-\nabla \Pi_h^{\textit{sz}} u)+\rho_{p_h(\cdot)s,\Omega}(\nabla \Pi_h^{\textit{sz}} u)
     \\&\lesssim 1+\rho_{p_h(\cdot)s,\Omega}(\nabla_{\! h} u_h^{\textit{cr}}-\nabla \Pi_h^{\textit{sz}} u)+\rho_{p(\cdot)s,\Omega}(\nabla u)\,.
      \\&\lesssim 1+h^{d(1-s)}_{\max}\rho_{p_h(\cdot),\Omega}(\nabla_{\! h} u_h^{\textit{cr}}-\nabla \Pi_h^{\textit{sz}} u)^s+\rho_{p(\cdot)s,\Omega}(\nabla u)\,.
     \end{aligned}
 \end{align}
 Next, we need to distinguish the cases $p_h\ge 2$ and $p_h< 2$:\enlargethispage{2mm}

 \textit{Case $p_h\hspace*{-0.1em}\ge\hspace*{-0.1em} 2$.} Due to \eqref{eq:hammera}, $(\varphi_h)_a(x,t)\hspace*{-0.1em}\gtrsim\hspace*{-0.1em}  t^{p_h(x)}$ for all $a,t\hspace*{-0.1em}\ge\hspace*{-0.1em} 0$ and ${x\hspace*{-0.1em}\in\hspace*{-0.1em} \{p_h\hspace*{-0.1em}\ge\hspace*{-0.1em} 2\}}$,~\eqref{cor:rate_improve.11},~and~\eqref{cor:rate_improve.3}$_1$, we have that
 					 \begin{align}\label{cor:improved_stability.3}
 					 	\begin{aligned}
 						 	\rho_{p_h(\cdot),\{p_h\ge 2\}}(\nabla_{\! h} u_h^{\textit{cr}}-\nabla  \Pi_h^{\textit{sz}}  u)&\lesssim 	\|F_h(\cdot,\nabla_{\! h} u_h^{\textit{cr}})-F_h(\cdot,\nabla  \Pi_h^{\textit{sz}}  u)\|_{2,\{p_h\ge 2\}}^2	\\&	\lesssim \|F_h(\cdot,\nabla_{\! h} u_h^{\textit{cr}})-F_h(\cdot,\nabla  u)\|_{2,\Omega}^2\\&\quad+\|F_h(\cdot,\nabla  u)-F_h(\cdot,\nabla  \Pi_h^{\textit{sz}}  u)\|_{2,\Omega}^2\\&\lesssim h_{\max}^{2\alpha +d(1-s)}\,\Theta(s)^s\,.
 						 \end{aligned}
 						 \end{align}

 \textit{Case $p_h\hspace*{-0.1em}< \hspace*{-0.1em}2$.} Due to  \cite[Lem.\  B.1]{BK23er_fluids},  $\rho_{p_h(\cdot)}(\vert \nabla_{\! h} u_h^{\textit{cr}}\vert +\vert \nabla  u\vert)\hspace*{-0.1em}\leq\hspace*{-0.1em} c$ (\textit{cf.}\ Lemma \ref{lem:a_priori}), \eqref{cor:rate_improve.11},~and~\eqref{cor:rate_improve.3}$_1$, we have that 
 				 	\begin{align*}
 				 		\begin{aligned}
 					 			\|\nabla_{\! h} u_h^{\textit{cr}}-\nabla \Pi_h^{\textit{sz}}  u\|_{p_h(\cdot),\{p_h< 2\}}^2&\lesssim  \|F_h(\cdot,\nabla_{\! h} u_h^{\textit{cr}})-F_h(\cdot,\nabla  \Pi_h^{\textit{sz}}  u)\|_{2,\{p_h< 2\}}^2\\&\quad\times \big(1+\rho_{p_h(\cdot)}(\vert \nabla_{\! h} u_h^{\textit{cr}}\vert +\vert \nabla  u\vert)\big)^{\smash{1/p^-}}
 					 		\\&	\lesssim \|F_h(\cdot,\nabla_{\! h} u_h^{\textit{cr}})-F_h(\cdot,\nabla  u)\|_{2,\Omega}^2\\&\quad+\|F_h(\cdot,\nabla  u)-F_h(\cdot,\nabla  \Pi_h^{\textit{sz}}  u)\|_{2,\Omega}^2\\&\lesssim h_{\max}^{2\alpha +d(1-s)}\,\Theta(s)^s\,,
 					 	\end{aligned}
 					 	\end{align*}
 			 		which, resorting to \cite[Lem.\  3.2.4]{DHHR11}, implies that
 			 		\begin{align}\label{cor:improved_stability.4}
 			 			\begin{aligned}
 			 			\rho_{p_h(\cdot),\{p_h< 2\}}(\nabla  u_h-\nabla \Pi_h^{\textit{sz}} u)\lesssim	\|\nabla u_h-\nabla \Pi_h^{\textit{sz}}  u\|_{p_h(\cdot),\{p_h< 2\}}\lesssim h_{\max}^{\alpha +d(1-s)/2}\,\Theta(s)^{s/2}\,.
 			 		\end{aligned}
 			 		\end{align}
 		 		Eventually, using \eqref{cor:improved_stability.3} and \eqref{cor:improved_stability.4} in \eqref{cor:rate_improve.12}, for $s>1$ sufficiently small, 
 		 		from \eqref{cor:rate_improve.2.0} and \eqref{cor:rate_improve.9}, 
 		 		we conclude the claimed \textit{a priori} error estimate.
		\end{proof}
\newpage

 Aided by the (discrete) convex optimality relations \eqref{eq:pDirichletOptimality2} and \eqref{eq:pDirichletOptimalityCR2}, together with the generalized Marini formula (\textit{cf.}\ \eqref{eq:gen_marini}), we can derive from Corollary \ref{thm:rate_u} an \textit{a priori} error estimate for the error between the  dual solution and the discrete dual solution. 
 
 \begin{lemma}\label{lem:rate_z} 
		 Let $p\in C^0(\overline{\Omega})$ with $p^->1$ and $\delta\ge 0$ and let $f\in L^{p'(\cdot)}(\Omega)\cap \bigcap_{h\in (0,h_0]}{L^{p_h'(\cdot)}(\Omega)}$ for some $h_0>0$. Then, there exists some $s>1$,  which can chosen to be close to $1$ if $h>0$~is~close~to~$0$, such that if $u\in W^{1,p(\cdot)s}_D(\Omega)$, then for every $h\in (0,h_0]$, it holds that
			\begin{align*}
				\|F^*_h(\cdot,z_h^{\textit{rt}})-F^*_h(\cdot,z)\|^2_{2,\Omega}&\lesssim \|F_h(\cdot,\nabla_{\!h} u_h^{\textit{cr}})-F_h(\cdot,\nabla u)\|^2_{2,\Omega}+ h_{\max}^{2\alpha}\,\big(1+\rho_{p(\cdot)s,\Omega}(\nabla u)\big)\\&\quad+
				\rho_{((\varphi_h)_{\smash{\vert \nabla u\vert}})^*,\Omega}(h_{\mathcal{T}}f)
				\,,
			\end{align*}
   where the hidden constant in $\lesssim$ also depends on $s>1$.
		\end{lemma}
		
		\begin{proof}
		 Using the discrete convex optimality relations \eqref{eq:pDirichletOptimality2} and \eqref{eq:pDirichletOptimalityCR2}, two equivalences~in~\eqref{eq:hammera}, the generalized Marini formula \eqref{eq:gen_marini}, again, the discrete convex optimality relations \eqref{eq:pDirichletOptimalityCR2}, and the Orlicz-stability of $\Pi_h$ (\textit{cf.}\ \cite[Cor.\ A.8, (A.12)]{kr-phi-ldg}), we find that
		 \begin{align}\label{cor:convex_rate_z.1} 
		  \begin{aligned}
		 	\|F^*_h(\cdot,z_h^{\textit{rt}})-F^*_h(\cdot,z)\|^2_{2,\Omega}&\lesssim \|F^*_h(\cdot,\Pi_h z_h^{\textit{rt}})-F^*_h(\cdot,z)\|^2_{2,\Omega}
		 	+\|F^*_h(\cdot,z_h^{\textit{rt}})-F^*_h(\cdot,\Pi_h z_h^{\textit{rt}})\|^2_{2,\Omega}
		 	\\&
		 	\lesssim \|F^*_h(\cdot,\AAA_h(\cdot,\nabla_{\!h} u_h^{\textit{cr}}))-F^*_h(\cdot,\AAA_h(\cdot,\nabla u))\|^2_{2,\Omega}\\&\quad+ \|F^*_h(\cdot,\AAA_h(\cdot,\nabla u)-F^*_h(\cdot,\AAA(\cdot,\nabla u))\|^2_{2,\Omega}\\&\quad+\rho_{((\varphi_h)^*)_{\smash{\vert \Pi_h z_h^{\textit{rt}}\vert}},\Omega}(z_h^{\textit{rt}}-\Pi_h z_h^{\textit{rt}}) 
		 	\\&
		 	\lesssim \|F_h(\cdot,\nabla_{\!h} u_h^{\textit{cr}})-F_h(\cdot,\nabla u)\|^2_{2,\Omega}
 + h_{\max}^{2\alpha}\,\big(1+\rho_{p(\cdot)s}(\nabla u)\big)\\&\quad+\rho_{((\varphi_h)^*)_{\smash{\vert \AAA_h(\cdot,\nabla_{\! h }u_h^{\textit{cr}})\vert}},\Omega}(h_{\mathcal{T}}f)\,.
		 	 \end{aligned}
		 \end{align}
		 Using that $(\varphi^*)_{\vert \AAA(\xi_T,a)\vert }(\xi_T,\cdot)\sim (\varphi_{\vert a\vert })^*(\xi_T,\cdot)$ for all $T\in \mathcal{T}_h$ and $a\in \mathbb{R}^d$ (\textit{cf.}\ Lemma \ref{lem:hammer}\eqref{eq:hammerg}) and the shift change \eqref{lem:shift-change.3},~we~observe~that
		 \begin{align}\label{cor:convex_rate_z.2} 
		 \begin{aligned}
		  \rho_{((\varphi_h)^*)_{\smash{\vert \AAA_h(\cdot,\nabla_{\! h }u_h^{\textit{cr}})\vert}},\Omega}(h_{\mathcal{T}}f)
            \lesssim \rho_{((\varphi_h)_{\smash{\vert \nabla u\vert}})^*,\Omega}(h_{\mathcal{T}}f)+\|F_h(\cdot,\nabla_{\!h} u_h^{\textit{cr}})-F_h(\cdot,\nabla u)\|^2_{2,\Omega}\,.
		 \end{aligned}
		 \end{align}
		 Eventually, combining \eqref{cor:convex_rate_z.1} and \eqref{cor:convex_rate_z.2}, we arrive at the claimed inequality.
		\end{proof}
		
		\begin{theorem}\label{thm:rate_z} Let $p\in C^{0,\alpha}(\overline{\Omega})$ with $\alpha\in (0,1]$ and $p^->1$ and $\delta\ge 0$.~Moreover,~let~$F(\cdot,\nabla u)\in W^{1,2}(\Omega;\mathbb{R}^d)$, $(\delta^{p(\cdot)-1}+\vert z\vert )^{p'(\cdot)-2}\vert f\vert^2\in L^1(\{p> 2\})$, and $f\in L^{(p^-)'}(\Omega)$ or $\Gamma_D=\partial \Omega$.~Then, assuming that $h_{\max}\sim h_{\mathcal{T}}$, there exists some $s>1$, which can chosen to be close to $1$ if $h_{\max}>0$ is close to $0$, such that if $f\in L^{p'(\cdot)s}(\Omega)$, then
		\begin{align*}
		 	\|F^*_h(\cdot,z_h^{\textit{rt}})-F^*_h(\cdot,z)\|^2_{2,\Omega}&\lesssim h_{\max}^2\,\big(1+\|\nabla F(\cdot,\nabla u)\|_{2,\Omega}^{2}+\|(\delta^{p(\cdot)-1}+\vert z\vert)^{p'(\cdot)-2}\vert f\vert^2\|_{1,\{p>2\}}\\&\quad+
            \sigma(f;s)+\rho_{p'(\cdot)s,\Omega}(f)+\rho_{p(\cdot)s,\Omega}( \nabla u)\big)^s\,,
		\end{align*}
        where the hidden constant in $\lesssim$ also depends on $s>1$ and $\sigma(f;s)$ is defined as in Theorem \ref{thm:rate_u}.
		\end{theorem}

        \begin{proof}
            Immediate consequence of Lemma \ref{lem:rate_z} in conjunction with Theorem \ref{thm:rate_u}.
        \end{proof}

        \begin{corollary}\label{cor:rate}
        Let $p\in C^{0,1}(\overline{\Omega})$ with $p^->1$ and $\delta> 0$. Moreover, let $F(\cdot,\nabla u)\in W^{1,2}(\Omega;\mathbb{R}^d)$ and $f\hspace{-0.1em}\in\hspace{-0.1em} L^{(p^-)'}(\Omega)$ or $\Gamma_D\hspace{-0.1em}=\hspace{-0.1em}\partial \Omega$.
		Then, assuming that $h_{\max}\hspace{-0.1em}\sim\hspace{-0.1em} h_{\mathcal{T}}$, there exists some $s\hspace{-0.1em}>\hspace{-0.1em}1$,~which~can chosen to be close to $1$ if $h_{\max}>0$ is close to $0$, such that if $f\hspace{-0.1em}\in\hspace{-0.1em} L^{p'(\cdot)s}(\Omega)$,~then
        \begin{align*}
		 	&\|F_h(\cdot,\nabla_{\!h} u_h^{\textit{cr}})-F_h(\cdot,\nabla u)\|_{2,\Omega}^2+\|F^*_h(\cdot,z_h^{\textit{rt}})-F^*_h(\cdot,z)\|^2_{2,\Omega}\\&\quad\lesssim h_{\max}^2\,\big(1+\|\nabla F(\cdot,\nabla u)\|_{2,\Omega}^2+
           \sigma(f;s)+ \rho_{p'(\cdot)s,\Omega}(f)+\rho_{p(\cdot)s,\Omega}( \nabla u)\big)^s\,,
		\end{align*}
        where the hidden constant in $\lesssim$ also depends on $s>1$ and $\sigma(f;s)$ is defined as in Theorem \ref{thm:rate_u}.\enlargethispage{5mm}
        \end{corollary}
        
        \begin{proof}
            Due to Lemma \ref{lem:reg_equiv}, from $F(\cdot,\nabla u)\in W^{1,2}(\Omega;\mathbb{R}^d)$, it follows that $F^*(\cdot,z)\in W^{1,2}(\Omega;\mathbb{R}^d)$ and $\vert \nabla F(\cdot,\nabla u)\vert^2+(1+\vert \nabla v\vert^{p(\cdot)s})\sim \vert \nabla F^*(\cdot,z)\vert^2+(1+\vert z\vert^{p'(\cdot)s})$ a.e.\ in $\Omega$ for some $s>1$ which can chosen to be close to $1$. In addition, due to Lemma \ref{lem:reg_dual_source}, it holds that $\vert \nabla F^*(\cdot,z)\vert^2+(1+\vert z\vert^{p'(\cdot)s})\sim (\delta^{p(\cdot)-1}+\vert z\vert)^{p'(\cdot)-2} \vert \nabla z\vert^2$ a.e.\ in $\Omega$. As a result, the claimed \textit{a priori} error estimate follows from Theorem \ref{thm:rate_u} together with Theorem \ref{thm:rate_z}.
        \end{proof}

        \section{Numerical experiments}\label{sec:experiments}
        
	\hspace{5mm}In this section, we review the theoretical findings of  Section \ref{sec:a_priori} via numerical experiments.
    All experiments were carried out using the finite element software package \mbox{\textsf{FEniCS}} (version~2019.1.0, \textit{cf}.\ \cite{LW10}).

	We apply the $\smash{\mathcal{S}^{1,\textit{cr}}_D(\mathcal{T}_h)}$-approximation \eqref{eq:pDirichletS1crD} of the variational 
	$p(\cdot)$-Dirichlet problem \eqref{eq:pDirichletW1p} with 
    $\delta\coloneqq 1.0\times 10^{-4}$ and $p\in C^{0,\alpha}(\overline{\Omega})$, where $\alpha \in (0,1]$ and $p^->1$, for every $x\in \overline{\Omega}$ defined by
    \begin{align*}
        p(x)\coloneqq p^-+\varepsilon\,\vert x\vert^{\alpha}\,,
    \end{align*}
    where $\varepsilon>0$. As quadrature points of the one-point quadrature rule used to discretize $p\in C^{0,\alpha}(\overline{\Omega})$, we employ barycenters  of elements, \textit{i.e.}, for every $T\in \mathcal{T}_h$, we employ $\xi_T\coloneqq x_T= \frac{1}{d+1}\sum_{z\in \mathcal{N}_h\cap T}{z}$.
    
    Then, we approximate the discrete primal solution ${u_h^{\textit{cr}}\in \smash{\mathcal{S}^{1,\textit{cr}}_D(\mathcal{T}_h)}}$  deploying the Newton line search algorithm  of \mbox{\textsf{PETSc}} (version     3.17.3, \textit{cf}.\ \cite{LW10}), with an absolute tolerance~of~${\tau_{abs}= 1.0\times 10^{-8}}$ and a relative tolerance of $\tau_{rel}=1.0\times 10^{-10}$. The linear system  emerging  in each Newton step is solved using a sparse direct solver from \textup{\textsf{MUMPS}} (version 5.5.0, \textit{cf}.\ \cite{mumps}).
    
    For  our numerical experiments, we choose $\Omega= (-1,1)^2$, $\Gamma_D=\partial \Omega$, and as a manufactured solution of \eqref{eq:pDirichlet}, the function $u\in W^{1,p(\cdot)}_D(\Omega)$, for every $x\coloneqq (x_1,x_2)^\top\in \Omega$ defined by
    \begin{align*}
    	\smash{u(x)\coloneqq d(x)\,\vert x\vert^{\beta}}\,,
    \end{align*}
    \textit{i.e.}, $f\coloneqq -\textup{div}\,\AAA(\cdot,\nabla u)$,
    where 
      $d\in C^\infty(\overline{\Omega})$, defined by $d(x)\coloneqq (1-x_1^2)\,(1-x_2^2)$ for every $x\coloneqq(x_1,x_2)^\top \in\overline{\Omega}$,
    is a smooth cut-off function enforcing homogeneous Dirichlet~\mbox{boundary}~\mbox{conditions}. Moreover, we  choose $\beta=1.01$, which just yields that $u\in W^{1,p(\cdot)}_D(\Omega)$ satisfies 
    \begin{align*}
        \smash{F(\cdot,\nabla u),F^*(\cdot,z)\in W^{1,2}(\Omega;\mathbb{R}^2)}\quad\text{ and }\quad \smash{(\delta^{p(\cdot)-1}+\vert z\vert)^{p'(\cdot)-2}\vert \nabla z\vert^2}\in L^1(\Omega)\,.
    \end{align*}
    By \hspace{-0.1mm}Theorem \hspace{-0.1mm}\ref{thm:rate_u} \hspace{-0.1mm}and \hspace{-0.1mm}Theorem \hspace{-0.1mm}\ref{thm:rate_z}, \hspace{-0.1mm}we \hspace{-0.1mm}expect \hspace{-0.1mm}the
    \hspace{-0.1mm}convergence \hspace{-0.1mm}rate \hspace{-0.1mm}$\alpha$ \hspace{-0.1mm}for \hspace{-0.1mm}the~\hspace{-0.1mm}error~\hspace{-0.1mm}quantities~\hspace{-0.1mm}\eqref{errors}.

    An initial triangulation $\mathcal
    T_{h_0}$, $h_0=\smash{\frac{3}{2\sqrt{2}}}$, is constructed by subdividing a rectangular~Cartesian grid into regular triangles with different orientations.  Refined     triangulations $\mathcal T_{h_k}$,~$k=1,\dots,10$, where $h_{k+1}=\frac{h_k}{2}$ for all $k=1,\dots,10$, are 
    obtained by  
    applying the red-refinement rule (\textit{cf.}\ \cite{Car04}). 
    
    Then, for the resulting series of triangulations $\mathcal T_k\coloneqq \mathcal T_{h_k}$, $k=1,\dots,10$, we apply the above Newton scheme to compute the discrete primal solution $u_k^{\textit{cr}}\coloneqq u_{h_k}^{\textit{cr}}\in \mathcal{S}^{1,\textit{cr}}_D(\mathcal{T}_k)$,~${k=1,\dots,10}$,~and, resorting to the generalized Marini formula \eqref{eq:gen_marini}, the discrete dual solution ${z_k^{\textit{rt}}\coloneqq z_{h_k}^{\textit{rt}}\in\mathcal{R}T^0_N(\mathcal{T}_k)}$, $k=1,\dots,10$. Subsequently, we compute the error quantities
    \begin{align}\label{errors}
    	\left.\begin{aligned}
    		e_{F,k}&\coloneqq \|F_h(\cdot,\nabla_{\!h_k}u_k^{\textit{cr}} )-F_h(\cdot,\nabla u)\|_{2,\Omega}^2\,,\\
    		e_{F^*,k}&\coloneqq \|F^*_h(\cdot,z_k^{\textit{cr}})-F^*_h(\cdot,z)\|_{2,\Omega}^2\,,
    \end{aligned}\quad\right\}\quad k=1,\dots,10\,.
    \end{align}
    For the determination of the convergence rates,  the experimental order of convergence~(EOC)
    \begin{align*}
    	\texttt{EOC}_k(e_k)\coloneqq \frac{\log(e_k/e_{k-1})}{\log(h_k/h_{k-1})}\,, \quad k=1,\dots,10\,,
    \end{align*}
    where for every $k= 1,\dots,10$, we denote by $e_k$,
    either $e_{F,k}$ or 
    $e_{F^*,k}$, 
    ~\mbox{respectively},~is~recorded.\enlargethispage{7mm}
    
    For different values of $p^-\in \{1.5, 2, 2.5\}$, $\alpha\in \{0.1,0.25,0.5,1.0\}$, $\varepsilon\in \{0.5,1.0\}$, and a
    series of triangulations \hspace{-0.1mm}$\mathcal{T}_k$, \hspace{-0.1mm}$k = 1,\dots,10$,
    \hspace{-0.1mm}obtained \hspace{-0.1mm}by \hspace{-0.1mm}uniform \hspace{-0.1mm}mesh \hspace{-0.1mm}refinement~\hspace{-0.1mm}as~\hspace{-0.1mm}\mbox{described}~\hspace{-0.1mm}above,~\hspace{-0.1mm}the~\hspace{-0.1mm}EOC is
    computed and for $k = 5,\dots,10$ presented in Table~\ref{tab1}, Table~\ref{tab2}, Table~\ref{tab3}~and~Table~\ref{tab4},   
    respectively. In each case, we report a convergence ratio of about $\texttt{EOC}_k(e_k)\approx 1$, $k=5,\dots,10$, conforming the quasi-optimality of the  \textit{a priori} error estimates in Corollary~\ref{cor:rate} and but not indicating the quasi-optimality of the \textit{a priori} error estimates in Theorem \ref{thm:rate_u} and Theorem \ref{thm:rate_z} for $\alpha\in (0,1)$. We believe that this can be traced back to an imbalance of between the regularity assumptions $F(\cdot,\nabla u)\in W^{1,2}(\Omega;\mathbb{R}^d)$ and $p\in C^{0,\alpha}(\overline{\Omega})$ with $\alpha\in (0,1)$ in Theorem \ref{thm:rate_u} and Theorem \ref{thm:rate_z};~and expect that  the assumptions $F(\cdot,\nabla u)\in \mathcal{N}^{\alpha,2}(\Omega)^d$ and $p\in C^{0,\alpha}(\overline{\Omega})$ with $\alpha\in (0,1)$ 
    are more balanced and may lead to optimal \textit{a priori} error estimates. The examination of this hypothesis, however, is beyond the scope of this article and, therefore, left open for follow-up research.
\begin{table}[H]
     \setlength\tabcolsep{3pt}
 	\centering
 	\begin{tabular}{c |c|c|c|c|c|c|c|c|c|c|c|c|} 
 	\hline 
    \multicolumn{1}{|c||}{\cellcolor{lightgray}$\varepsilon$}	
    & \multicolumn{12}{c|}{\cellcolor{lightgray}$1.0$}\\\hline 
    \multicolumn{1}{|c||}{\cellcolor{lightgray}$\alpha$}	
    & \multicolumn{3}{c||}{\cellcolor{lightgray}$0.1$} & \multicolumn{3}{c||}{\cellcolor{lightgray}$0.25$} & \multicolumn{3}{c||}{\cellcolor{lightgray}$0.5$} & \multicolumn{3}{c|}{\cellcolor{lightgray}$1.0$}\\\hline 
\multicolumn{1}{|c||}{\cellcolor{lightgray}\diagbox[height=1.1\line,width=0.11\dimexpr\linewidth]{\vspace{-0.6mm}$k$}{\\[-5mm] $p^-$}}
& \cellcolor{lightgray}1.5
& \cellcolor{lightgray}2.0 &  \multicolumn{1}{c||}{\cellcolor{lightgray}2.5}  &  \multicolumn{1}{c|}{\cellcolor{lightgray}1.5} &\cellcolor{lightgray}2.0  & \multicolumn{1}{c||}{\cellcolor{lightgray}2.5}  &  \multicolumn{1}{c|}{\cellcolor{lightgray}1.5} & \cellcolor{lightgray}2.0  &  \multicolumn{1}{c||}{\cellcolor{lightgray}2.5}   &  \multicolumn{1}{c|}{\cellcolor{lightgray}1.5} & \cellcolor{lightgray}2.0    & \cellcolor{lightgray}2.5   \\ \hline\hline
\multicolumn{1}{|c||}{\cellcolor{lightgray}$5$} & 0.978 & 0.986 & \multicolumn{1}{c||}{0.987} & \multicolumn{1}{c|}{0.958} & 0.972 & \multicolumn{1}{c||}{0.974} & \multicolumn{1}{c|}{0.971} & 0.986 & \multicolumn{1}{c||}{0.989} & \multicolumn{1}{c|}{0.969} & 0.988 & 0.992 \\ \hline
\multicolumn{1}{|c||}{\cellcolor{lightgray}$6$} & 0.979 & 0.989 & \multicolumn{1}{c||}{0.992} & \multicolumn{1}{c|}{0.975} & 0.985 & \multicolumn{1}{c||}{0.985} & \multicolumn{1}{c|}{0.971} & 0.987 & \multicolumn{1}{c||}{0.992} & \multicolumn{1}{c|}{0.970} & 0.988 & 0.993 \\ \hline
\multicolumn{1}{|c||}{\cellcolor{lightgray}$7$} & 0.981 & 0.990 & \multicolumn{1}{c||}{0.994} & \multicolumn{1}{c|}{0.975} & 0.988 & \multicolumn{1}{c||}{0.988} & \multicolumn{1}{c|}{0.972} & 0.988 & \multicolumn{1}{c||}{0.993} & \multicolumn{1}{c|}{0.972} & 0.988 & 0.994 \\ \hline
\multicolumn{1}{|c||}{\cellcolor{lightgray}$8$} & 0.981 & 0.990 & \multicolumn{1}{c||}{0.994} & \multicolumn{1}{c|}{0.976} & 0.989 & \multicolumn{1}{c||}{0.989} & \multicolumn{1}{c|}{0.973} & 0.988 & \multicolumn{1}{c||}{0.994} & \multicolumn{1}{c|}{0.973} & 0.989 & 0.994 \\ \hline
\multicolumn{1}{|c||}{\cellcolor{lightgray}$9$} & 0.981 & 0.990 & \multicolumn{1}{c||}{0.994} & \multicolumn{1}{c|}{0.976} & 0.989 & \multicolumn{1}{c||}{0.990} & \multicolumn{1}{c|}{0.974} & 0.989 & \multicolumn{1}{c||}{0.994} & \multicolumn{1}{c|}{0.974} & 0.989 & 0.994 \\ \hline
\multicolumn{1}{|c||}{\cellcolor{lightgray}$10$}& 0.982 & 0.990 & \multicolumn{1}{c||}{0.995} & \multicolumn{1}{c|}{0.977} & 0.989 & \multicolumn{1}{c||}{0.990} & \multicolumn{1}{c|}{0.975} & 0.989 & \multicolumn{1}{c||}{0.994} & \multicolumn{1}{c|}{0.975} & 0.989 & 0.994 \\ \hline\hline
\multicolumn{1}{|c||}{\cellcolor{lightgray}\small \textrm{expected}}   & 0.10  & 0.10  & \multicolumn{1}{c||}{0.10}  & \multicolumn{1}{c|}{0.25} & 0.25  & \multicolumn{1}{c||}{0.25}  & \multicolumn{1}{c|}{0.50}  & 0.50  & \multicolumn{1}{c||}{0.50}  & \multicolumn{1}{c|}{1.00}  & 1.00  & 1.00 \\ \hline
\end{tabular}\vspace{-2mm}
 	\caption{Experimental order of convergence: $\texttt{EOC}_k(e_{F,k})$,~${k=5,\dots,10}$.} 
 	\label{tab1}
 \end{table}\vspace{-5mm}

\begin{table}[H]
     \setlength\tabcolsep{3pt}
 	\centering
 	\begin{tabular}{c |c|c|c|c|c|c|c|c|c|c|c|c|} 
 	\hline 
    \multicolumn{1}{|c||}{\cellcolor{lightgray}$\varepsilon$}	
    & \multicolumn{12}{c|}{\cellcolor{lightgray}$1.0$}\\\hline 
    \multicolumn{1}{|c||}{\cellcolor{lightgray}$\alpha$}	
    & \multicolumn{3}{c||}{\cellcolor{lightgray}$0.1$} & \multicolumn{3}{c||}{\cellcolor{lightgray}$0.25$} & \multicolumn{3}{c||}{\cellcolor{lightgray}$0.5$} & \multicolumn{3}{c|}{\cellcolor{lightgray}$1.0$}\\\hline 
\multicolumn{1}{|c||}{\cellcolor{lightgray}\diagbox[height=1.1\line,width=0.11\dimexpr\linewidth]{\vspace{-0.6mm}$k$}{\\[-5mm] $p^-$}}
& \cellcolor{lightgray}1.5
& \cellcolor{lightgray}2.0 &  \multicolumn{1}{c||}{\cellcolor{lightgray}2.5}  &  \multicolumn{1}{c|}{\cellcolor{lightgray}1.5} &\cellcolor{lightgray}2.0  & \multicolumn{1}{c||}{\cellcolor{lightgray}2.5}  &  \multicolumn{1}{c|}{\cellcolor{lightgray}1.5} & \cellcolor{lightgray}2.0  &  \multicolumn{1}{c||}{\cellcolor{lightgray}2.5}   &  \multicolumn{1}{c|}{\cellcolor{lightgray}1.5} & \cellcolor{lightgray}2.0    & \cellcolor{lightgray}2.5   \\ \hline\hline
\multicolumn{1}{|c||}{\cellcolor{lightgray}$5$} & 0.978 & 0.986 & \multicolumn{1}{c||}{0.987} & \multicolumn{1}{c|}{0.975} & 0.985 & \multicolumn{1}{c||}{0.988} & \multicolumn{1}{c|}{0.971} & 0.986 & \multicolumn{1}{c||}{0.989} & \multicolumn{1}{c|}{0.969} & 0.988 & 0.992 \\ \hline
\multicolumn{1}{|c||}{\cellcolor{lightgray}$6$} & 0.979 & 0.989 & \multicolumn{1}{c||}{0.992} & \multicolumn{1}{c|}{0.975} & 0.988 & \multicolumn{1}{c||}{0.992} & \multicolumn{1}{c|}{0.971} & 0.987 & \multicolumn{1}{c||}{0.992} & \multicolumn{1}{c|}{0.970} & 0.988 & 0.993 \\ \hline
\multicolumn{1}{|c||}{\cellcolor{lightgray}$7$} & 0.981 & 0.990 & \multicolumn{1}{c||}{0.994} & \multicolumn{1}{c|}{0.976} & 0.989 & \multicolumn{1}{c||}{0.993} & \multicolumn{1}{c|}{0.972} & 0.988 & \multicolumn{1}{c||}{0.993} & \multicolumn{1}{c|}{0.972} & 0.988 & 0.994 \\ \hline
\multicolumn{1}{|c||}{\cellcolor{lightgray}$8$} & 0.981 & 0.990 & \multicolumn{1}{c||}{0.994} & \multicolumn{1}{c|}{0.976} & 0.989 & \multicolumn{1}{c||}{0.994} & \multicolumn{1}{c|}{0.973} & 0.988 & \multicolumn{1}{c||}{0.994} & \multicolumn{1}{c|}{0.973} & 0.989 & 0.994 \\ \hline
\multicolumn{1}{|c||}{\cellcolor{lightgray}$9$} & 0.981 & 0.990 & \multicolumn{1}{c||}{0.994} & \multicolumn{1}{c|}{0.977} & 0.989 & \multicolumn{1}{c||}{0.994} & \multicolumn{1}{c|}{0.974} & 0.989 & \multicolumn{1}{c||}{0.994} & \multicolumn{1}{c|}{0.974} & 0.989 & 0.994 \\ \hline
\multicolumn{1}{|c||}{\cellcolor{lightgray}$10$}& 0.982 & 0.990 & \multicolumn{1}{c||}{0.995} & \multicolumn{1}{c|}{0.977} & 0.989 & \multicolumn{1}{c||}{0.994} & \multicolumn{1}{c|}{0.975} & 0.989 & \multicolumn{1}{c||}{0.994} & \multicolumn{1}{c|}{0.975} & 0.989 & 0.994 \\ \hline\hline
\multicolumn{1}{|c||}{\cellcolor{lightgray}\small \textrm{expected}}   & 0.10  & 0.10  & \multicolumn{1}{c||}{0.10}  & \multicolumn{1}{c|}{0.25} & 0.25  & \multicolumn{1}{c||}{0.25}  & \multicolumn{1}{c|}{0.50}  & 0.50  & \multicolumn{1}{c||}{0.50}  & \multicolumn{1}{c|}{1.00}  & 1.00  & 1.00 \\ \hline
\end{tabular}\vspace{-2mm}
 	\caption{Experimental order of convergence: $\texttt{EOC}_k(e_{F^*,k})$,~${k=5,\dots,10}$.} 
 	\label{tab2}
 \end{table}\vspace{-5mm}

\begin{table}[H]
     \setlength\tabcolsep{3pt}
 	\centering
 	\begin{tabular}{c |c|c|c|c|c|c|c|c|c|c|c|c|} 
 	\hline 
    \multicolumn{1}{|c||}{\cellcolor{lightgray}$\varepsilon$}	
    & \multicolumn{12}{c|}{\cellcolor{lightgray}$0.5$}\\\hline 
    \multicolumn{1}{|c||}{\cellcolor{lightgray}$\alpha$}	
    & \multicolumn{3}{c||}{\cellcolor{lightgray}$0.1$} & \multicolumn{3}{c||}{\cellcolor{lightgray}$0.25$} & \multicolumn{3}{c||}{\cellcolor{lightgray}$0.5$} & \multicolumn{3}{c|}{\cellcolor{lightgray}$1.0$}\\\hline 
\multicolumn{1}{|c||}{\cellcolor{lightgray}\diagbox[height=1.1\line,width=0.11\dimexpr\linewidth]{\vspace{-0.6mm}$k$}{\\[-5mm] $p^-$}}
& \cellcolor{lightgray}1.5
& \cellcolor{lightgray}2.0 &  \multicolumn{1}{c||}{\cellcolor{lightgray}2.5}  &  \multicolumn{1}{c|}{\cellcolor{lightgray}1.5} &\cellcolor{lightgray}2.0  & \multicolumn{1}{c||}{\cellcolor{lightgray}2.5}  &  \multicolumn{1}{c|}{\cellcolor{lightgray}1.5} & \cellcolor{lightgray}2.0  &  \multicolumn{1}{c||}{\cellcolor{lightgray}2.5}   &  \multicolumn{1}{c|}{\cellcolor{lightgray}1.5} & \cellcolor{lightgray}2.0    & \cellcolor{lightgray}2.5   \\ \hline\hline
\multicolumn{1}{|c||}{\cellcolor{lightgray}$5$} & 0.959 & 0.979 & \multicolumn{1}{c||}{0.986} & \multicolumn{1}{c|}{0.975} & 0.985 & \multicolumn{1}{c||}{0.988} & \multicolumn{1}{c|}{0.971} & 0.978 & \multicolumn{1}{c||}{0.986} & \multicolumn{1}{c|}{0.944} & 0.979 & 0.988 \\ \hline
\multicolumn{1}{|c||}{\cellcolor{lightgray}$6$} & 0.965 & 0.981 & \multicolumn{1}{c||}{0.989} & \multicolumn{1}{c|}{0.975} & 0.988 & \multicolumn{1}{c||}{0.992} & \multicolumn{1}{c|}{0.971} & 0.979 & \multicolumn{1}{c||}{0.989} & \multicolumn{1}{c|}{0.958} & 0.980 & 0.989 \\ \hline
\multicolumn{1}{|c||}{\cellcolor{lightgray}$7$} & 0.966 & 0.983 & \multicolumn{1}{c||}{0.990} & \multicolumn{1}{c|}{0.976} & 0.989 & \multicolumn{1}{c||}{0.993} & \multicolumn{1}{c|}{0.972} & 0.981 & \multicolumn{1}{c||}{0.990} & \multicolumn{1}{c|}{0.958} & 0.981 & 0.990 \\ \hline
\multicolumn{1}{|c||}{\cellcolor{lightgray}$8$} & 0.968 & 0.983 & \multicolumn{1}{c||}{0.991} & \multicolumn{1}{c|}{0.976} & 0.989 & \multicolumn{1}{c||}{0.994} & \multicolumn{1}{c|}{0.973} & 0.982 & \multicolumn{1}{c||}{0.990} & \multicolumn{1}{c|}{0.962} & 0.982 & 0.990 \\ \hline
\multicolumn{1}{|c||}{\cellcolor{lightgray}$9$} & 0.969 & 0.984 & \multicolumn{1}{c||}{0.991} & \multicolumn{1}{c|}{0.977} & 0.989 & \multicolumn{1}{c||}{0.994} & \multicolumn{1}{c|}{0.974} & 0.982 & \multicolumn{1}{c||}{0.990} & \multicolumn{1}{c|}{0.964} & 0.983 & 0.991 \\ \hline
\multicolumn{1}{|c||}{\cellcolor{lightgray}$10$}& 0.970 & 0.984 & \multicolumn{1}{c||}{0.991} & \multicolumn{1}{c|}{0.977} & 0.989 & \multicolumn{1}{c||}{0.994} & \multicolumn{1}{c|}{0.975} & 0.983 & \multicolumn{1}{c||}{0.991} & \multicolumn{1}{c|}{0.966} & 0.983 & 0.991 \\ \hline\hline
\multicolumn{1}{|c||}{\cellcolor{lightgray}\small \textrm{expected}}   & 0.10  & 0.10  & \multicolumn{1}{c||}{0.10}  & \multicolumn{1}{c|}{0.25} & 0.25  & \multicolumn{1}{c||}{0.25}  & \multicolumn{1}{c|}{0.50}  & 0.50  & \multicolumn{1}{c||}{0.50}  & \multicolumn{1}{c|}{1.00}  & 1.00  & 1.00 \\ \hline
\end{tabular}\vspace{-2mm}
 	\caption{Experimental order of convergence: $\texttt{EOC}_k(e_{F,k})$,~${k=5,\dots,10}$.} 
 	\label{tab3}
 \end{table}\vspace{-5mm}

\begin{table}[H]
     \setlength\tabcolsep{3pt}
 	\centering
 	\begin{tabular}{c |c|c|c|c|c|c|c|c|c|c|c|c|} 
 	\hline 
    \multicolumn{1}{|c||}{\cellcolor{lightgray}$\varepsilon$}	
    & \multicolumn{12}{c|}{\cellcolor{lightgray}$0.5$}\\\hline 
    \multicolumn{1}{|c||}{\cellcolor{lightgray}$\alpha$}	
    & \multicolumn{3}{c||}{\cellcolor{lightgray}$0.1$} & \multicolumn{3}{c||}{\cellcolor{lightgray}$0.25$} & \multicolumn{3}{c||}{\cellcolor{lightgray}$0.5$} & \multicolumn{3}{c|}{\cellcolor{lightgray}$1.0$}\\\hline 
\multicolumn{1}{|c||}{\cellcolor{lightgray}\diagbox[height=1.1\line,width=0.11\dimexpr\linewidth]{\vspace{-0.6mm}$k$}{\\[-5mm] $p^-$}}
& \cellcolor{lightgray}1.5
& \cellcolor{lightgray}2.0 &  \multicolumn{1}{c||}{\cellcolor{lightgray}2.5}  &  \multicolumn{1}{c|}{\cellcolor{lightgray}1.5} &\cellcolor{lightgray}2.0  & \multicolumn{1}{c||}{\cellcolor{lightgray}2.5}  &  \multicolumn{1}{c|}{\cellcolor{lightgray}1.5} & \cellcolor{lightgray}2.0  &  \multicolumn{1}{c||}{\cellcolor{lightgray}2.5}   &  \multicolumn{1}{c|}{\cellcolor{lightgray}1.5} & \cellcolor{lightgray}2.0    & \cellcolor{lightgray}2.5   \\ \hline\hline
\multicolumn{1}{|c||}{\cellcolor{lightgray}$5$} & 0.959 & 0.979 & \multicolumn{1}{c||}{0.986} & \multicolumn{1}{c|}{0.953} & 0.978 & \multicolumn{1}{c||}{0.986} & \multicolumn{1}{c|}{0.947} & 0.978 & \multicolumn{1}{c||}{0.986} & \multicolumn{1}{c|}{0.943} & 0.979 & 0.988 \\ \hline
\multicolumn{1}{|c||}{\cellcolor{lightgray}$6$} & 0.965 & 0.981 & \multicolumn{1}{c||}{0.989} & \multicolumn{1}{c|}{0.960} & 0.980 & \multicolumn{1}{c||}{0.989} & \multicolumn{1}{c|}{0.956} & 0.979 & \multicolumn{1}{c||}{0.989} & \multicolumn{1}{c|}{0.958} & 0.980 & 0.989 \\ \hline
\multicolumn{1}{|c||}{\cellcolor{lightgray}$7$} & 0.966 & 0.983 & \multicolumn{1}{c||}{0.990} & \multicolumn{1}{c|}{0.961} & 0.981 & \multicolumn{1}{c||}{0.990} & \multicolumn{1}{c|}{0.958} & 0.981 & \multicolumn{1}{c||}{0.990} & \multicolumn{1}{c|}{0.958} & 0.981 & 0.990 \\ \hline
\multicolumn{1}{|c||}{\cellcolor{lightgray}$8$} & 0.968 & 0.983 & \multicolumn{1}{c||}{0.991} & \multicolumn{1}{c|}{0.963} & 0.982 & \multicolumn{1}{c||}{0.990} & \multicolumn{1}{c|}{0.961} & 0.982 & \multicolumn{1}{c||}{0.990} & \multicolumn{1}{c|}{0.962} & 0.982 & 0.990 \\ \hline
\multicolumn{1}{|c||}{\cellcolor{lightgray}$9$} & 0.969 & 0.984 & \multicolumn{1}{c||}{0.991} & \multicolumn{1}{c|}{0.964} & 0.983 & \multicolumn{1}{c||}{0.991} & \multicolumn{1}{c|}{0.963} & 0.982 & \multicolumn{1}{c||}{0.990} & \multicolumn{1}{c|}{0.964} & 0.983 & 0.991 \\ \hline
\multicolumn{1}{|c||}{\cellcolor{lightgray}$10$}& 0.970 & 0.984 & \multicolumn{1}{c||}{0.991} & \multicolumn{1}{c|}{0.966} & 0.983 & \multicolumn{1}{c||}{0.991} & \multicolumn{1}{c|}{0.965} & 0.983 & \multicolumn{1}{c||}{0.991} & \multicolumn{1}{c|}{0.966} & 0.983 & 0.991 \\ \hline\hline
\multicolumn{1}{|c||}{\cellcolor{lightgray}\small \textrm{expected}}   & 0.10  & 0.10  & \multicolumn{1}{c||}{0.10}  & \multicolumn{1}{c|}{0.25} & 0.25  & \multicolumn{1}{c||}{0.25}  & \multicolumn{1}{c|}{0.50}  & 0.50  & \multicolumn{1}{c||}{0.50}  & \multicolumn{1}{c|}{1.00}  & 1.00  & 1.00 \\ \hline
\end{tabular}\vspace{-2mm}
 	\caption{Experimental order of convergence: $\texttt{EOC}_k(e_{F^*,k})$,~${k=5,\dots,10}$.} 
 	\label{tab4}
 \end{table}\vspace{-5mm}

\appendix

\section{Discrete-to-continuous-and-vice-versa inequalities}

 \hspace{5mm}The following lemma is of crucial importance for the hereinafter analysis; it bounds the error resulting from switching from $\AAA_h\colon\Omega\times\mathbb{R}^d\to \mathbb{R}^d$, $h>0$, to $\AAA\colon\Omega\times\mathbb{R}^d\to \mathbb{R}^d$ or from~switching~from
 $F_h\colon\Omega\times\mathbb{R}^d\to \mathbb{R}^d$, $h>0$, to $F\colon\Omega\times\mathbb{R}^d\to \mathbb{R}^d$ and vice versa, respectively.\enlargethispage{5mm}
 
 \begin{proposition}\label{lem:A-Ah}
 Let $p\in C^0(\overline{\Omega})$ with $p^->1$ and let $\delta\ge 0$.   Then, there exists~some~${s>1}$,~which can chosen to be close to $0$ if $h_T>0$ is close to $0$, such that
  for every $T\in \mathcal{T}_h$, $g\in L^{p'(\cdot)s}(T)$, $v\in W^{1,p(\cdot)s}(T)$, and $\lambda\in [0,1]$, it holds that
 \begin{align}
  \| F_h(\cdot,\nabla v)-F(\cdot,\nabla v)\|_{2,T}^2&\lesssim  \|\omega_{p,T}(h_T)^{2}\,(1+\vert \nabla v\vert^{p(\cdot)s})\|_{1,T}
  \,,\label{eq:Fh-F}\\
  \| F_h^*(\cdot,\AAA_h(\cdot,\nabla v))-F^*_h(\cdot,\AAA(\cdot,\nabla v))\|_{2,T}^2&\lesssim \|\omega_{p,T}(h_T)^{2}\,(1+\vert \nabla v\vert^{p(\cdot)s})\|_{1,T}
  \,,\label{eq:Ah-A}\\
  \rho_{((\varphi_h)_{\vert \nabla v \vert})^*,T}(\lambda\,g) &\lesssim \rho_{(\varphi_{\vert \nabla v \vert})^*,T}(\lambda\,g)\label{eq:phih-phi}\\&\quad+\smash{\lambda^{\smash{{2\wedge (p^+)'}}}}\,\|\omega_{p,T}(h_T)\,(1\hspace{-0.05em}+\hspace{-0.05em}\vert \nabla v\vert^{p(\cdot)s}\hspace{-0.05em}+\hspace{-0.05em}\vert g\vert^{p'(\cdot)s})\|_{1,T}\,,\notag
 \end{align} 
  where the hidden constants in $\lesssim$ also depend on $s>1$ and the chunkiness $\omega_0>0$.  
 \end{proposition}

Proposition \ref{lem:A-Ah} is based on the following point-wise estimates.\vspace{-1mm}\enlargethispage{5.5mm}

 \begin{lemma}\label{lem:Ax-Axh}
  Let $p\in C^0(\overline{\Omega})$ with $p^->1$ and let $\delta\ge 0$. 
  Then, there exists some $s>1$, which can chosen to be close to $1$ if $\vert x-y\vert$ is close to $0$,  such that 
  for every $x,y\in \overline{\Omega}$, $t\ge 0$, $a,b\in \mathbb{R}^d$, and $\lambda\in [0,1]$, it holds that
 \begin{align}
  \vert F(x,a)-F(y,a)\vert^2&\lesssim \vert p(x)-p(y)\vert^2\,(1 +\vert a\vert^{p(x)s})\,,\label{eq:Fxh-Fx}\\
   \vert F^*(x,a)-F^*(y,a)\vert^2&\lesssim \vert p(x)-p(y)\vert^2\,(1 +\vert a\vert^{p(x)s})\,,\label{eq:F*xh-F*x}\\
 \vert F^*(x,\AAA(x,a))-F^*(x,\AAA(y,a))\vert^2&\lesssim\vert p(x)-p(y)\vert^2\,(1 +\vert a\vert^{p(x)s}) \,,\label{eq:Axh-Ax}\\
  \begin{split}
  (\varphi_{\vert b\vert})^*(x,\lambda\,t) &\lesssim (\varphi_{\vert b\vert})^*(y,\lambda\,t)\\[-0.5mm]&\quad+
  \smash{\lambda^{\smash{{2\wedge (p^+)'}}}}\,\vert p(x)-p(y)\vert \,(1+\vert b\vert^{p(y)s}+t^{p'(y)s})\,,
  \end{split}\label{eq:phixh-phix}
 \end{align}
 where the hidden constants in $\lesssim$ also depend on $s>1$ and the chunkiness $\omega_0>0$.
 \end{lemma}

 \begin{proof}\let\qed\relax
    \textit{ad \eqref{eq:Fxh-Fx}.}
 The Newton--Leibniz formula yields~for~all~${x,y\in \overline{\Omega}}$~and~${a\in \mathbb{R}^d}$~that
  \begin{align}\label{eq:Fh-F.1}
  \begin{aligned}
\vert F(x,a)-F(y,a)\vert^2&\lesssim \vert p(x)-p(y)\vert^2\,\ln(\delta +\vert a\vert)^2\,(
 \varphi(x,\vert a\vert )+\varphi(y,\vert a\vert ))
 \\&\lesssim 
  \vert p(x)-p(y)\vert^2\,\ln(\delta +\vert a\vert)^2\,(1+(\delta +\vert a\vert)^{p(y)-p(x)})\,\varphi(x,\vert a\vert) 
  \\&\lesssim \vert p(x)-p(y)\vert^2\,(1 +\vert a\vert^{p(x)s})\,,
  \end{aligned}
  \end{align} 
 for some $s>1$, which can chosen to be close to $1$ if $\vert x-y\vert$ is close to $0$.
 
  \textit{ad \eqref{eq:Axh-Ax}.} The Newton--Leibniz formula yields~for~all~${x,y\in  \overline{\Omega}}$~and~${a\in \mathbb{R}^d}$ that
  \begin{align}\label{eq:Ah-A.1}
  \begin{aligned}
 \vert \AAA(x,a)-\AAA(y,a)\vert&\lesssim \vert p(x)-p(y)\vert\,\vert\ln(\delta +\vert a\vert)\vert\,(
 \varphi'(x,\vert a \vert)+\varphi'(y,\vert a\vert))\\&\lesssim 
 \vert p(x)-p(y)\vert\, \vert \ln(\delta +\vert a\vert)\vert\, (1+(\delta +\vert a\vert)^{p(y)-p(x)})\, \varphi'(x,\vert a\vert)\,.
 \end{aligned}
  \end{align}
  Using \eqref{eq:Ah-A.1}, \eqref{eq:hammerf},  \eqref{eq:hammerg}, the monotonicity of $(\varphi_{\vert a\vert })^*(x,\cdot)$, that 
  $\smash{\Delta_2((\varphi_{\vert a\vert })^*(x,\cdot))\lesssim 2^{\max\{2,(p^-)'\}}}$, (\textit{cf}.\ Remark \ref{rem:phi_a}),\vspace{-1.5mm}
  \begin{align}
      (\varphi_{\vert a\vert })^*(x,\lambda\, t)&\lesssim \smash{\max\{\lambda^{p'(x)},\lambda^2\}} (\varphi_{\vert a\vert })^*(x, t)\,,\label{estimate}\\
      (\varphi_{\vert a\vert})^*(x,\lambda\,\varphi'(x,\vert a\vert))&\sim \smash{\lambda^2}\, \varphi(x,\vert a\vert)\,,
  \end{align}
  (\textit{cf}.\ \cite[Lem.\  A.7 \& Lem.\  A.8]{BDS15}), in conjunction with Remark \ref{rem:uniform}, we deduce that
  \begin{align*}
 \vert F^*(x,\,&\AAA(x,a))-F^*(x,\AAA(y,a))\vert^2\lesssim (\varphi_{\vert a\vert })^*(x,\vert \AAA(x,a)-\AAA(y,a)\vert ) 
  \\&\lesssim 
  ((1+\vert \ln(\delta +\vert a\vert)\vert)\, (1+(\delta +\vert a\vert)^{p(y)-p(x)}))^{\max\{2,p'(x)\}}\,(\varphi_{\vert a\vert })^*(x,\vert p(x)-p(y)\vert\,\varphi'(x,\vert a\vert))
  \\&\lesssim 
  \vert p(x)-p(y)\vert^2\,((1+\vert \ln(\delta +\vert a\vert)\vert) \,(1+(\delta +\vert a\vert)^{p(y)-p(x)}))^{\max\{2,p'(x)\}}\,\varphi(x,\vert a\vert)
  \\&\lesssim \vert p(x)-p(y)\vert^2\,(1 +\vert a\vert^{p(x)s})\,,
 \end{align*}
 for some $s>1$, which can chosen to be close to $1$ if $\vert x-y\vert$ is close to $0$.\enlargethispage{2.5mm}

  \textit{ad \eqref{eq:phixh-phix}.} 
  Using twice the Newton--Leibniz formula yields for all  $x,y\in  \overline{\Omega}$ and $a\in \mathbb{R}^d$ that\enlargethispage{5mm}
  \begin{align}\label{eq:phih-phi.1}
  \begin{aligned}
  &(\varphi_{\vert b\vert})^*(x,\lambda\, t) 
  \lesssim ((\delta+\vert b\vert)^{p(x)-1}+\lambda\,t)^{p'(x)-2}(\lambda\,t)^2
  \\& = ((\delta+\vert b\vert)^{p(y)-1}+\lambda\,t)^{p'(x)-2}(\lambda\,t)^2
  \\&\quad+ \int_{p(x)\wedge p(y)}^{p(x)\vee p(y)}{(p'(x)-2)(\delta+\vert b\vert)^{r-1}\ln(\delta+\vert b\vert)((\delta+\vert b\vert)^{r-1}+\lambda\,t)^{p'(x)-3}(\lambda\,t)^2\,\textrm{d}r} 
  \\&=((\delta+\vert b\vert)^{p(y)-1}+\lambda\,t)^{p'(y)-2}(\lambda\,t)^2
 \\&\quad+ \int_{p(x)\wedge p(y)}^{p(x)\vee p(y)}{(p'(x)-2)(\delta+\vert b\vert)^{r-1}\ln(\delta+\vert b\vert)((\delta+\vert b\vert)^{r-1}+\lambda\,t)^{p'(x)-3}(\lambda\,t)^2\,\textrm{d}r}
  \\&\quad+\int_{p'(x)\wedge p'(y)}^{p'(x)\vee p'(y)}{\ln((\delta+\vert b\vert)^{p(y)-1}+\lambda\,t)((\delta+\vert b\vert)^{p(y)-1}+\lambda\,t)^{r-2}(\lambda\,t)^2\,\textrm{d}r}\,. 
  \\&\eqqcolon ((\delta+\vert b\vert)^{p(y)-1}+\lambda\,t)^{p'(y)-2}(\lambda\,t)^2+I_h^1+I_h^2\,.
  \end{aligned}\hspace{-1cm}
  \end{align}
  Next, we need to estimate the terms $I_h^1$ and $I_h^2$:

  \textit{ad $I_h^1$.} Using $((\delta+\vert b\vert)^{r-1}+\lambda\,t)^{p'(x)-2}(\lambda\,t)^2\lesssim \lambda^{\min\{2,p'(x)\}} ((\delta+\vert b\vert)^{r-1}+t)^{p'(x)-2}t^2$,~we~obtain
  \begin{align}\label{eq:phih-phi.2}
  \begin{aligned}
 \hspace{-1mm}I_h^1&\lesssim \vert \ln(\delta+\vert b\vert)\vert \int_{p(x)\wedge p(y)}^{p(x)\vee p(y)}{((\delta+\vert b\vert)^{r-1}+\lambda\,t)^{p'(x)-2}(\lambda\,t)^2\,\textrm{d}r} 
  \\&\lesssim
  \lambda^{\smash{2\wedge p'(x)}}\,\vert p(x)-p(y)\vert\,\vert \ln(\delta+\vert b\vert)\vert \,((\delta+\vert b\vert)^{p(x)-1}+(\delta+\vert b\vert)^{p(y)-1}+t)^{p'(x)}
 \\&\lesssim
  \lambda^{\smash{{2\wedge (p^+)'}}}\,\vert p(x)-p(y)\vert\, (1+(\delta+\vert b\vert)^{p(x)-1}+(\delta+\vert b\vert)^{p(y)-1}+t)^{p'(x)s}
  \\&\lesssim
  \lambda^{\smash{{2\wedge (p^+)'}}}\,\vert p(x)-p(y)\vert\, (1+(\delta+\vert b\vert)^{p(x)s}+(\delta+\vert b\vert)^{p(y)(p'(x)/p'(y))s}+t^{p'(y)(p'(x)/p'(y))s})
  \\&\lesssim
  \lambda^{\smash{{2\wedge (p^+)'}}}\, \vert p(x)-p(y)\vert\,(1+\vert b\vert^{p(y)s}+t^{p'(y)s})\,,
  \end{aligned}\hspace{-2.5mm}
  \end{align}
  for some $s>1$, which can chosen to be close to $1$ if $\vert x-y\vert$ is close to $0$.
  
  \textit{ad $I_h^2$.} Using $((\delta+\vert b\vert)^{p(y)-1}+\lambda\,t)^{r-2}(\lambda\,t)^2\lesssim \lambda^{\min\{2,r\}} ((\delta+\vert b\vert)^{p(y)-1}+t)^{r-2}t^2$,~we~obtain
  \begin{align}
  \label{eq:phih-phi.3}
  \begin{aligned}
  I_h^2&\lesssim \lambda^{{2\wedge (p^+)'}}
  \vert p'(x)-p'(y)\vert\,\vert\ln((\delta+\vert b\vert)^{p(y)-1}+t)\vert\,(1+(\delta+\vert b\vert)^{p(y)-1}+t)^{p'(x)-2}t^2
  \\&\quad+\lambda^{{2\wedge (p^+)'}}\vert p'(x)-p'(y)\vert\,\vert\ln((\delta+\vert b\vert)^{p(y)-1}+t)\vert\,(1+(\delta+\vert b\vert)^{p(y)-1}+t)^{p'(y)-2}t^2\\ 
  &\lesssim \lambda^{\smash{{2\wedge (p^+)'}}}\,\tfrac{\vert p(x)-p(y)\vert}{(p(x)-1)(p(y)-1)}\,\vert\ln((\delta+\vert b\vert)^{p(y)-1}+t)\vert\,(1+(\delta+\vert b\vert)^{p(y)-1}+t)^{p'(x)}
  \\&\quad+\lambda^{\smash{{2\wedge (p^+)'}}}\,\tfrac{\vert p(x)-p(y)\vert}{(p(x)-1)(p(y)-1)}\,\vert\ln((\delta+\vert b\vert)^{p(y)-1}+t)\vert\,(1+(\delta+\vert b\vert)^{p(y)-1}+t)^{p'(y)}
 \\&\lesssim \lambda^{\smash{{2\wedge (p^+)'}}}\,\vert p(x)-p(y)\vert\, \vert\ln((\delta+\vert b\vert)^{p(y)-1}+t)\vert\,(1+(\delta+\vert b\vert)^{p(y)-1}+t)^{p'(y)(p'(x)/p'(y))}
  \\&\quad+\lambda^{\smash{{2\wedge (p^+)'}}}\,\vert p(x)-p(y)\vert\,\vert\ln((\delta+\vert b\vert)^{p(y)-1}+t)\vert\,(1+(\delta+\vert b\vert)^{p(y)-1}+t)^{p'(y)} 
 \\&\lesssim \lambda^{\smash{{2\wedge (p^+)'}}}\,\vert p(x)-p(y)\vert\, (1+\vert b\vert^{p(y)s}+t^{p'(y)s})\,,
  \end{aligned} \hspace{-2.5mm}
  \end{align}
  for some $s>1$, which can chosen to be close to $1$ if $\vert x-y\vert$ is close to $0$.
  
  Eventually, combining \eqref{eq:phih-phi.2} and \eqref{eq:phih-phi.3} in \eqref{eq:phih-phi.1}, appealing to Remark \ref{rem:phi_a}, we conclude~that
  \begin{align*}
  (\varphi_{\vert b\vert})^*(x, \lambda\,t)\lesssim (\varphi_{\vert b\vert})^*(y, \lambda\,t)+\lambda^{{2\wedge (p^+)'}}\,\vert p(x)-p(y)\vert (1+\vert b\vert^{p(y)s}+t^{p'(y)s})\,.
  \end{align*}

  \textit{ad \eqref{eq:F*xh-F*x}.} Using the Newton--Leibniz formula and proceeding as for \eqref{eq:F*xh-F*x} and \eqref{eq:phih-phi.3} yields for all  $x,y\in  \overline{\Omega}$ and $a\in \mathbb{R}^d$ that
  \begin{align*}
      \vert F^*(x,a)-F^*(y,a)\vert^2 &\lesssim  \vert (\delta^{p(x)-1}+\vert a\vert)^{\frac{p'(x)-2}{2}}a-(\delta^{p(x)-1}+\vert a\vert)^{\frac{p'(y)-2}{2}}a\vert^2
      \\&\quad +\int_{p(x)\wedge p(y)}^{p(x)\vee p(y)}{\vert \ln(\delta)\vert(\delta^{r-1}+\vert a\vert)^{p'(x)-2}\vert a\vert^2\,\textrm{d}r}
      \\&\lesssim \vert p(x)-p(y)\vert \,(1+\vert a\vert^{p(y)s})\,.\tag*{$\qedsymbol$}
  \end{align*}
 \end{proof}

 The authors  would like to thank the reviewer for careful reading of our paper, valuable remarks and suggestions.

 {\setlength{\bibsep}{0pt plus 0.0ex}\small

  \providecommand{\bysame}{\leavevmode\hbox to3em{\hrulefill}\thinspace}
  \providecommand{\noopsort}[1]{}
  \providecommand{\mr}[1]{\href{http://www.ams.org/mathscinet-getitem?mr=#1}{MR~#1}}
  \providecommand{\zbl}[1]{\href{http://www.zentralblatt-math.org/zmath/en/search/?q=an:#1}{Zbl~#1}}
  \providecommand{\jfm}[1]{\href{http://www.emis.de/cgi-bin/JFM-item?#1}{JFM~#1}}
  \providecommand{\arxiv}[1]{\href{http://www.arxiv.org/abs/#1}{arXiv~#1}}
  \providecommand{\doi}[1]{\url{https://doi.org/#1}}
  \providecommand{\MR}{\relax\ifhmode\unskip\space\fi MR }
  \providecommand{\MRhref}[2]{%
  	\href{http://www.ams.org/mathscinet-getitem?mr=#1}{#2}
  }
  \providecommand{\href}[2]{#2}

}


\end{document}